\documentclass{amsart}
\usepackage{amsmath,amssymb,amsthm,amsfonts,amscd,mathrsfs,mathtools}
\usepackage[utf8]{inputenc}
\usepackage[hyphens]{url}
\usepackage[pagebackref=true]{hyperref}
\usepackage{tikz-cd}
\usepackage{tikz}
\usepgflibrary{arrows}
\usetikzlibrary{shapes.arrows,trees,backgrounds,decorations.markings,shadows,calc, patterns, positioning}
\usepackage{enumitem}

\setlist[itemize,description]{leftmargin=*}

\theoremstyle{plain}
    \newtheorem{thm}{Theorem}[section]
    \newtheorem{prop}[thm]  {Proposition}
    \newtheorem{lem}[thm]   {Lemma}
    \newtheorem{cor}[thm]   {Corollary}

    \newtheorem{athm}{Theorem}

\theoremstyle{definition}
    \newtheorem{defn}[thm]  {Definition}
    \newtheorem{nota}[thm]  {Notation}
    \newtheorem{ex}[thm]    {Example}

\newcommand{\C}{\mathbb{C}} 
\newcommand{\fc}{\mathfrak{c}} 
\newcommand{\tfc}{\tilde{\fc}} 
\newcommand{\fe}{\mathfrak{e}} 
\newcommand{\bF}{\mathbb{F}} 
\newcommand{\cH}{\mathcal{H}} 
\newcommand{\fri}{\mathfrak{i}} 
\newcommand{\fM}{\mathfrak{M}}
\newcommand{\N}{\mathbb{N}} 
\newcommand{\bQ}{\mathbb{Q}} 
\newcommand{\R}{\mathbb{R}} 
\newcommand{\fr}{\mathfrak{r}} 
\newcommand{\bS}{\mathbb{S}} 
\newcommand{\fS}{\mathfrak{S}} 
\newcommand{\cS}{\mathcal{S}} 
\newcommand{\uU}{\underline{U}} 
\newcommand{\fU}{\mathfrak{U}} 
\newcommand{\cX}{\mathcal{X}} 
\newcommand{\ucX}{\underline{\cX}} 
\newcommand{\cY}{\mathcal{Y}} 
\newcommand{\ucY}{\underline{\cY}} 
\newcommand{\Z}{\mathbb{Z}} 
\usepackage{mathbbol}
\DeclareSymbolFontAlphabet{\mathbb}{AMSb}
\DeclareSymbolFontAlphabet{\mathbbl}{bbold}
\newcommand{\one}{\mathbbl{1}} 

\newcommand{\cG}{\mathcal{G}} 
\newcommand{\Q}{\mathcal{Q}} 
\newcommand{\hQ}{\hat{\Q}} 
\newcommand{\PMQ}{\mathbf{PMQ}} 
\newcommand{\Top}{\mathbf{Top}} 
\newcommand{\PMQGrp}{\mathbf{PMQGrp}} 
\newcommand{\FQ}{\bF\bQ} 
\newcommand{\geo}{^{\mathrm{geo}}} 

\newcommand{\ab}{\mathbf{ab}} 
\newcommand{\gen}{f} 
\newcommand{\fg}{\mathfrak{g}} 
\newcommand{\Hur}{\mathrm{Hur}} 

\newcommand{\psiext}{\psi^{\mathrm{ext}}}
\newcommand{\psipext}{(\psi')^{\mathrm{ext}}}
\newcommand{\HurTQ}{\Hur(\fT;\Q)}%
\newcommand{\HurTQG}{\Hur(\fT;\Q,G)}%
\newcommand{\HurTQGp}{\Hur(\fT;\Q',G')}%
\newcommand{\HurTpQG}{\Hur(\fT';\Q,G)}%
\newcommand{\pr}{p} 
\newcommand{\hur}{\mathrm{hur}} 
\newcommand{\arr}{\mathrm{arr}} 
\newcommand{\Farr}{F^{\arr}} 
\newcommand{\fFarr}{\mathfrak{F}^{\arr}} 
\newcommand{\CmP}{\C\setminus P} 
\newcommand{\zleft}{z^{\mathrm{l}}} 
\newcommand{\zright}{z^{\mathrm{r}}} 
\newcommand{\arcleft}{\arc^{\mathrm{l}}} 
\newcommand{\arcright}{\arc^{\mathrm{r}}} 
\newcommand{\genleft}{\gen^{\mathrm{l}}} 
\newcommand{\genright}{\gen^{\mathrm{r}}} 
\newcommand{\Uleft}{U^{\mathrm{l}}} 
\newcommand{\Uright}{U^{\mathrm{r}}} 
\newcommand{\arc}{\zeta} 
\newcommand{\std}{\mathrm{std}} 
\newcommand{\totmon}{\omega} 
\newcommand{\fG}{\mathfrak{G}} 
\newcommand{\fQ}{\mathfrak{Q}} 
\newcommand{\fQext}{\fQ^{\mathrm{ext}}} 
\newcommand{\Hurext}{\Hur^{\mathrm{ext}}} 
\newcommand{\fT}{\mathfrak{C}} 
\newcommand{\ufT}{\underline{\fT}} 
\newcommand{\NC}{\mathbf{NC}} 
\newcommand{\LNC}{\mathbf{LNC}} 
\newcommand{\expl}{\mathscr{E}} 
\newcommand{\Ran}{\mathrm{Ran}} 

\newcommand{\cR}{\mathcal{R}} 
\newcommand{\mcR}{\mathring{\cR}} 
\newcommand{\bH}{\mathbb{H}} 
\newcommand{\bT}{\mathbb{T}} 
\newcommand{\bJ}{\mathbb{J}} 
\newcommand{\zcentre}{z_{\mathrm{c}}} 

\newcommand{\us}{\underline{s}} 
\newcommand{\ut}{\underline{t}} 
\newcommand{\bary}{\mathfrak{bar}} 
\newcommand{\ubar}{\underline{\bary}} 
\newcommand{\ua}{\underline{a}} 
\newcommand{\Arr}{\mathrm{Arr}} 
\newcommand{\NAdm}{\mathrm{NAdm}} 
\newcommand{\ndeg}{\mathrm{ndeg}} 
\newcommand{\hor}{\mathrm{hor}} 
\newcommand{\ver}{\mathrm{ver}} 
\newcommand{\del}{\partial} 
\newcommand{\bDelta}{\breve{\Delta}} 
\newcommand{\borel}{/\! /} 

\newcommand{\pa}[1]{\left(#1\right)}

\newcommand{\abs}[1]{\left|#1\right|}
\newcommand{\set}[1]{\left\{#1\right\}}

\renewcommand{\phi}{\varphi}
\renewcommand{\epsilon}{\varepsilon}

\DeclareMathOperator{\Id}{Id}

\DeclareMathOperator{\Aut}{Aut} 

\title[Hurwitz-Ran spaces]{Hurwitz-Ran spaces}
\author{Andrea Bianchi}
\thanks{
This work was partially supported by the \emph{Deutsche
  Forschungsgemeinschaft} (DFG, German Research Foundation) under Germany’s
Excellence Strategy (\texttt{EXC-2047/1}, \texttt{390685813}), by the
\emph{European Research Council} under the European Union’s Seventh Framework
Programme (\texttt{ERC StG 716424 - CASe}, PI Karim Adiprasito), and by the
\emph{Danish National Research Foundation} through the \emph{Copenhagen Centre for
Geometry and Topology} (\texttt{DNRF151}).
}
\email{anbi@math.ku.dk}
\address{Department of Mathematical Sciences, University of Copenhagen \newline
Universitetsparken 5, Copenhagen, 2100, Denmark}  
\date{\today}


\begin{document}
\begin{abstract}
Given a couple of subspaces $\mathcal{Y}\subset\mathcal{X}$
of the complex plane $\mathbb{C}$ satisfying some mild conditions (a ``nice couple''), and given a PMQ-pair $(\mathcal{Q},G)$, consisting of a partially multiplicative quandle (PMQ) $\mathcal{Q}$  and a group $G$,
we introduce a ``Hurwitz-Ran'' space $\mathrm{Hur}(\mathcal{X},\mathcal{Y};\mathcal{Q},G)$, containing configurations of points in $\mathcal{X}\setminus\mathcal{Y}$ and in $\mathcal{Y}$ with monodromies in $\mathcal{Q}$ and in $G$, respectively.
We further introduce a notion of morphisms between nice couples, and prove that Hurwitz-Ran spaces are functorial both in the nice couple and in the PMQ-group pair.
For a locally finite PMQ $\Q$ we prove a homeomorphism between $\Hur((0,1)^2;\Q_+)$ and the simplicial Hurwitz space $\Hur^{\Delta}(\Q)$, introduced in previous work of the author: this provides in particular $\Hur((0,1)^2;\Q_+)$ with a cell stratification in the spirit of Fox-Neuwirth and Fuchs.
\end{abstract}

\maketitle

\section{Introduction}
In \cite{Bianchi:Hur1} we introduced the notion of \emph{partially multiplicative quandle} (shortly, PMQ)
and for a PMQ $\Q$ we defined a \emph{simplicial Hurwitz space} $\Hur^{\Delta}(\Q)$:
the space $\Hur^{\Delta}(\Q)$ is the difference of the geometric realisations of two bisimplicial sets, and is thus equipped with a cell stratification. The construction requires $\Q$ to be \emph{augmented} as a PMQ (see \cite[Definition 4.9]{Bianchi:Hur1}).

As discussed in \cite[Section 6]{Bianchi:Hur1}, a point in $\Hur^{\Delta}(\Q)$
can be regarded as the datum of a finite subset $P$ of the open unit square $(0,1)^2\subset\C$, together with a \emph{monodromy} $\psi$, defined on certain loops of
$\CmP$ and taking values in $\Q$.

In this article we introduce, for a \emph{semi-algebraic} subspace
$\cX\subset\bH$ of the closed upper half-plane $\bH=\set{\Re\ge0}\subset\C$, and for a PMQ
$\Q$, a ``Hurwitz-Ran'' space $\Hur(\cX;\Q)$.
More generally we introduce, for a \emph{nice couple} $(\cX,\cY)$ of subspaces $\cY\subseteq\cX$ of $\bH$ (see Definition \ref{defn:nicecouple}) and for a PMQ-group pair $(\Q,G)$ (see \cite[Definition 2.15]{Bianchi:Hur1}), a Hurwitz-Ran space $\Hur(\cX,\cY;\Q,G)$.
The part ``Hurwitz'' in the name suggests that elements of $\Hur(\cX,\cY;\Q,G)$ are configurations of points in the plane with monodromy data, like in classical Hurwitz spaces; the par ``Ran'' refers to the topology, defined in such a way that ``collisions'' between points are allowed (at least under certain circumstances), like in classical Ran spaces.

One can think of $\Hur(\cX,\cY;\Q,G)$ as a relative version of $\Hur(\cX;\Q)$. In the classical theory of configuration spaces, points of $P$ disappear when they move inside the relative subspace $\cY$; in our setting they do not disappear, but are ``downgraded'' to points around which only a $G$-valued monodromy, instead of a $\Q$-valued monodromy, is defined. The Hurwitz-Ran construction has the advantage of being natural not only in the PMQ-group pair $(\Q,G)$, but also in the nice couple $(\cX,\cY)$: for nice couples $(\cX,\cY)$ and $(\cX',\cY')$, and for suitable maps $\xi\colon\C\to\C$ sending $\cX\to\cX'$ and $\cY\to\cY'$, we are able to define a corresponding map $\xi_*\colon\Hur(\cX,\cY;\Q,G)\to\Hur(\cX',\cY';\Q,G)$; most notably, we can do this in certain cases in which $\xi$ is \emph{not injective}, in particular it is \emph{not a homeomorphism} of $\C$.

\subsection{Statement of results}
We prove that the construction of $\Hur(\cX;\Q)$ is functorial both in $\Q$ and in $\cX$; as special cases we have the following:
\begin{itemize}
 \item a morphism of PMQs $\Psi\colon\Q\to\Q'$ induces a continuous map
$\Psi_*\colon \Hur(\cX;\Q)\to\Hur(\cX;\Q')$;
 \item a semi-algebraic homeomorphism $\xi\colon\C\to\C$ with compact support in $\bH$,
 sending $\cX$ inside $\cX'$, induces a continuous map $\xi_*\colon\Hur(\cX;\Q)\to\Hur(\cX';\Q)$.
\end{itemize}
In Definitions \ref{defn:mapnicecouples} and \ref{defn:laxmapnicecouples} we introduce \emph{morphisms} and \emph{lax morphisms} of nice couples: both types of morphisms arise as certain continuous maps $\C\to\C$.
The first, main result of the article is the following theorem, combining Theorems \ref{thm:funPMQgroup}, \ref{thm:funnicecouple} and \ref{thm:funlaxnicecouple}.
\begin{athm}
\label{thm:main1}
The Hurwitz-Ran spaces $\Hur(\cX,\cY;\Q,G)$ are functorial both in the nice couple $(\cX,\cY)$, with respect to morphisms of nice couples, and in the PMQ-group pair. If we restrict to complete PMQs, then functoriality holds also with respect to lax morphisms of nice couples.
\end{athm}

We recall that the simplicial Hurwitz space $\Hur^{\Delta}(\Q)$ was defined in \cite{Bianchi:Hur1} only when $\Q$ is an \emph{augmented} PMQ. For such a PMQ, we can set $\cX=(0,1)^2\subset\C$ and consider the Hurwitz-Ran space $\Hur((0,1)^2;\Q)$. Inside the latter, we let $\Hur((0,1)^2;\Q_+)$, be the subspace of configurations $(P,\psi)$ such that the local
monodromy around each point of $P$ lies in $\Q_+:=\Q\setminus\set{\one}$. The second main result of the article is the following theorem, which is Theorem \ref{thm:upsilonhomeo}.
\begin{athm}
 \label{thm:main2}
If $\Q$ is augmented and \emph{locally finite} (see \cite[Definition 4.12]{Bianchi:Hur1}), then the simplicial Hurwitz space $\Hur^{\Delta}(\Q)$ is homeomorphic to $\Hur((0,1)^2;\Q_+)$.
\end{athm}
Thus Hurwitz-Ran spaces generalise simplicial Hurwitz spaces from \cite{Bianchi:Hur1}.

In \cite[Definition 6.19]{Bianchi:Hur1} we also introduced the notion of \emph{Poincar\'e} PMQ: an augmented and locally finite PMQ $\Q$
is Poincar\'e if each connected component of $\Hur^{\Delta}(\Q)$ is a topological manifold.
More generally, for a commutative ring $R$, we introduce the notion
of $R$-Poincar\'e PMQ: each component of $\Hur^{\Delta}(\Q)$ is required to be an $R$-homology manifold (see Definition \ref{defn:RPoincare}).
An $R$-Poincar\'e PMQ $\Q$ has the following advantage:
the $R$-homology groups of a component of $\Hur^{\Delta}(\Q)$
are isomorphic by Poincar\'e-Lefschetz duality to the reduced $R$-cohomology groups of the one point compactification of that component;
this one point compactification is endowed with a finite cell decomposition à la Fox-Neuwirth-Fuchs \cite[Section 6]{Bianchi:Hur1} which can in principle be used for actual cohomology computations.

The third main result of the article is the following theorem,
combining Theorems \ref{thm:Poincare} and \ref{thm:RPoincare}, and
giving a criterion to recognise Poincar\'e and $R$-Poincar\'e PMQs.
Recall that a PMQ $\Q$ embeds into its completion $\hQ$, and $\hQ$ is strictly larger than $\Q$ unless $\Q$ is already complete.
Recall also from \cite[Theorem 6.14]{Bianchi:Hur1} that the connected components of $\Hur^{\Delta}(\Q)$ are in bijection with $\hQ$.
\begin{athm}
 \label{thm:main3}
In order to prove that $\Q$ is Poincar\'e (respectively, $R$-Poincar\'e), it suffices to check that
for all $a\in\Q$ the corresponding component $\Hur^{\Delta}(\Q)(a)$
of $\Hur^{\Delta}(\Q)$ is a topological manifold (respectively, an $R$-homology manifold).
\end{athm}

\subsection{Outline of the article}
In Section \ref{sec:fundamentalPMQ} we introduce the notion of nice couple $\fT=(\cX,\cY)$ of subspaces $\cY\subset\cX$ of the closed
upper half-plane $\bH$.
For each finite subset $P\subset\cX$, we introduce
the \emph{fundamental PMQ} $\fQ_{\fT}(P)$, arising as a subset of $\pi_1(\CmP)$,
which will allow us to define configurations in $\Hur(\fT;\Q,G)$ supported on the set $P$.
In order to define a convenient topology on the set
$\Hur(\fT;\Q,G)$, we introduce also the notion of \emph{covering} $\uU$
of a finite subset $P\subset\cX$, and associate several PMQs also with the datum of a finite set
and a covering of it.

In Section \ref{sec:defnHurPMQ} we 
define the \emph{Hurwitz-Ran space} $\Hur(\fT;\Q,G)$, for each nice couple $\fT$ and each PMQ-group pair $(\Q,G)$, first as a set and
then as a Hausdorff topological space (see Proposition \ref{prop:Hurtopology}). We also discuss
a variation of the definition, using a contractible subspace $\bT\subset\C$ as ``ambient space''
instead of the entire $\C$.

In Section \ref{sec:functoriality} we introduce \emph{morphisms} and \emph{lax morphisms} of nice couples, and we prove Theorem \ref{thm:main1}.

In Section \ref{sec:applfunctoriality} we give some applications of functoriality
of Hurwitz-Ran spaces, in particular we study some local properties of their topology.
We also specialise Hurwitz-Ran spaces $\Hur(\cX,\cY;\Q,G)$ to the case $\cY=\emptyset$: in this case the resulting space only depends on $\Q$ and not on $G$, and we denote it $\Hur(\cX;\Q)$.

In Section \ref{sec:totmongroupactions} we introduce the \emph{total monodromy}, which is a discrete,
continuous invariant of configurations in Hurwitz-Ran spaces. The total monodromy is a map $\omega\colon\Hur(\fT;\Q,G)\to G$ in the relative case, and $\hat\totmon\colon\Hur(\cX;\Q)\to\hQ$ in the absolute case,
where $\hQ$ is the completion of $\Q$.
We also define three actions
of $G$ on Hurwitz-Ran spaces: the first is the action on $\Hur(\fT;\Q,G)$ \emph{by global conjugation};
the other two actions are only defined on certain subspaces of $\Hur(\fT;\Q,G)$, and leverage the possibility
of changing the local monodromy around a single point $z$ in the support of a configuration, provided that $z$ is the leftmost or rightmost point in the support.

In Section \ref{sec:augmented} we introduce, in the hypothesis that $\Q$ is \emph{augmented},
a subspace $\Hur(\fT;\Q_+,G)$ of $\Hur(\fT;\Q,G)$; using the notion of \emph{explosion}
we prove that the inclusion $\Hur(\fT;\Q_+,G)\subseteq\Hur(\fT;\Q,G)$ is in several cases a homotopy equivalence.

In Section \ref{sec:celldecomposition}, for an augmented PMQ $\Q$, we construct a continuous bijection
$\upsilon\colon|\Arr(\Q)|\to \Hur([0,1]^2;\hQ_+)$, and prove Theorem \ref{thm:main2}: here $\hQ$ is the completion of $\Q$, and
$\Arr(\Q)$ is the bisimplicial set from \cite[Definition 6.6]{Bianchi:Hur1}, whose geometric realisation
contains $\Hur^{\Delta}(\Q)$ as an open subspace.

In Section \ref{sec:Poincare} we prove that $\upsilon\colon\Hur^{\Delta}(\Q)\to\Hur((0,1)^2;\Q_+)$
is a homeomorphism under the additional hypothesis that $\Q$ is a \emph{locally finite} PMQ.
We then move our focus to \emph{Poincar\'e} PMQs and prove Theorem \ref{thm:main3}.

Finally, Appendix \ref{sec:deferred} contains the proofs of the most technical lemmas and propositions of the article; these proofs have been deferred to help the reader focus on the general framework.
Throughout the article we make heavy use of the results of \cite[Sections 2-6]{Bianchi:Hur1}: we cite every time which specific fact we are employing, so that the reader does not need to be familiar with the details of \cite{Bianchi:Hur1}.

\subsection{Motivation}
This is the second article in a series about Hurwitz spaces.
Our motivation to define Hurwitz-Ran spaces is two-fold.
\begin{itemize}
 \item The standard setting for Hurwitz spaces takes as input an integer $k\ge0$ and conjugacy-invariant subset $\Q$ in a group $G$, and gives a space $\hur_k(\Q)$ of moduli of branched $G$-covers of the unit square $(0,1)^2$ with exactly $k$ branch points and local monodromies in $\Q$; see for instance \cite{EVW:homstabhur}. The disjoint union $\hur(\Q):=\coprod_{k\ge0}\hur_k(\Q)$ carries a natural $E_1$-algebra structure, and one can try to extract ``stable'' information about a particular space $\hur_k(\Q)$ from the group-completion $\Omega B\hur(\Q)$.
 As we will see in a subsequent article \cite{Bianchi:Hur3}, Hurwitz-Ran spaces are well-suited for modeling the homotopy type of $B\hur(\Q)$, and provide, in a certain sense, even a double delooping ``$B^2(\hur(\Q))$'' of $\hur(Q)$, as if the latter were an $E_2$-algebra.
 \item The connection between Hurwitz spaces and moduli spaces of Riemann surfaces is well-known, and by using the family of PMQs $\fS_d\geo$ from \cite[Section 7]{Bianchi:Hur1}, we will see in a subsequent article \cite{Bianchi:Hur4} that one can in fact model the homotopy type of moduli spaces $\fM_{g,n}$ as components of Hurwitz-Ran spaces; here $\fM_{g,n}$ denotes
the moduli space of Riemann surfaces of genus $g\ge0$ with $n\ge1$ ordered and parametrised boundary components.
This connection
leads to an alternative proof of the Mumford conjecture on the stable rational cohomology of moduli spaces $\fM_{g,n}$, originally proved by Madsen and Weiss \cite{MadsenWeiss}.
\end{itemize}

\subsection{Acknowledgments}
This series of articles is a generalisation
and a further development of my PhD thesis \cite{BianchiPhD}. I am grateful to
my PhD supervisor
Carl-Friedrich B\"odigheimer,
Bastiaan Cnossen,
Florian Kranhold,
Martin Palmer and
Nathalie Wahl
for helpful comments and mathematical explanations related to this article.
I also thank the anonymous referee for suggesting a new, more appealing name for the article and for several comments that helped improve the exposition.

\tableofcontents
\section{Groups and PMQs from configurations in the plane}
\label{sec:fundamentalPMQ}
In this section we introduce nice couples of subspaces of $\C$ and use them to associate several PMQs to a finite configuration $P$ of points in $\C$.
\subsection{Nice couples}
\label{subsec:nicecouples}
\begin{nota}
 \label{nota:bH}
We endow the complex plane $\C$ with the basepoint $*$ corresponding
to the complex number $-\sqrt{-1}$, which is contained in the \emph{lower} half-plane.
The \emph{closed, upper} half-plane is $\bH:=\set{z\in\C\,|\,\Im(z)\ge0}$.
\end{nota}

\begin{defn}
\label{defn:semialgebraic}
A subset $\bJ\subseteq\C$ is \emph{semi-algebraic} if it can be expressed as a finite union
of subsets $\bJ_1,\dots,\bJ_r\subset\C$, such that each $\bJ_i$
is defined by a finite system of polynomial equalities
and (weak or strict) inequalities in the real, affine coordinates $\Re(z)$ and $\Im(z)$ of $\C$.
Given two semi-algebraic sets $\bJ,\bJ'\subseteq\C$, a continuous map $\xi\colon\bJ\to\bJ'$
is \emph{semi-algebraic} if $\bJ$ can be expressed as a finite union
of semi-algebraic subsets $\bJ_1,\dots,\bJ_r\subset\bJ$, and the coordinates of $\xi|_{\bJ_i}$ are expressed,
in the real affine coordinates of $\bJ_i$, by fractions of polynomials with real coefficients.
\end{defn}
Note that a semi-algebraic subset $\bJ\subset\C$ can be written locally as a finite union of subsets of $\C$
that are diffeomorphic to points, open segments and open triangles; finite unions and intersections
of semi-algebraic subsets are again semi-algebraic.
\begin{defn}
 \label{defn:nicecouple}
A \emph{nice couple} $\fT=(\cX,\cY)$
is a couple of subspaces $\emptyset\subseteq\cY\subseteq\cX\subseteq\bH$, such that the following
properties hold:
\begin{itemize}
 \item $\cX$ and $\cY$ are semi-algebraic; 
 \item $\cY$ is closed in $\cX$.
\end{itemize}
By abuse of notation, for $\cX\subset\bH$ we will denote by $\cX$ also the nice couple $(\cX,\emptyset)$.
\end{defn}

\subsection{Configurations and coverings}
We fix a nice couple $\fT=(\cX,\cY)$ for the rest of the section, and
consider finite configurations of points in $\cX\subset\C$.
\begin{nota}
 \label{nota:P}
We will often denote by $P=\set{z_1,\dots,z_k}\subset\cX$ a finite collection of distinct points,
for some $k\geq 0$. We will assume that there is $0\leq l\leq k$ such that $z_1,\dots,z_l$ are precisely the points
of $P$ lying in $\cX\setminus \cY$.
\end{nota}
\begin{defn}
\label{defn:uUcovering}
Let $P$ as in Notation \ref{nota:P}. A \emph{covering} of $P$ is a sequence $\uU=(U_1,\dots,U_\kappa)$ of convex, semi-algebraic, disjoint open subsets of $\C\setminus\set{*}$, satisfying the following conditions:
\begin{itemize}
 \item $P$ is contained in the union $U_1\cup\dots\cup U_\kappa$;
 \item each $U_i$ intersects $P$ at least in one point;
 \item the closures $\bar{U}_i\subseteq\C$ of the sets $U_i$ are disjoint, compact and do not contain $*$.
\end{itemize}
A covering of $P$ is \emph{adapted} (to $P$) if the following additional properties hold:
\begin{itemize}
 \item $\kappa=k$, and each $U_i$ contains exactly one point of $P$;
 \item for all $1\leq i\leq l$, i.e. for all $i$ such that $z_i\notin\cY$, if $z_i\in U_j$ then the closure $\bar{U}_j$ is disjoint from $\cY$.
\end{itemize}
\end{defn}
Since $\cY$ is closed in $\cX$, each finite $P\subset\cX$ admits some adapted covering.
Note that if $\uU$ is an adapted covering of $P$, then the inclusion $\C\setminus\uU\hookrightarrow\CmP$ is
a homotopy equivalence. See Figure \ref{fig:uUcovering}.
\begin{figure}[ht]
 \begin{tikzpicture}[scale=4,decoration={markings,mark=at position 0.38 with {\arrow{>}}}]
  \draw[dashed,->] (-.1,0) to (1.1,0);
  \draw[dashed,->] (0,-.2) to (0,1.1);
  \fill[black, opacity=.2, looseness=.8] (.1,.1) to[out=180, in=-90] (-.1,.3) to[out=90,in=-100] (.4,.9) to[out=80,in=-180]
  (.8,1) to[out=0,in=0] (.8,0) to[out=180,in=0] (.4,0) to[out=180,in=0] (.1,.1);
  \fill[black, opacity=.3, looseness=.8] (.8,0) to[out=180,in=-100] (.4,.9) to[out=80,in=-180]
  (.8,1) to[out=0,in=0] (.8,0);
  \fill[pattern=vertical lines, opacity=.6, rotate around={60:(.6,.3)}] (.6,.3) ellipse (.4cm and .1cm);
  \fill[pattern=vertical lines, opacity=.6, rotate around={45:(.4,.5)}] (.4,.5) ellipse (.33cm and .09cm);
  \fill[pattern=vertical lines, opacity=.6] (.6,.9) ellipse (.12cm);
  \node at (.2,.3){$\bullet$};\node at (.27,.3){\tiny$z_1$};
  \node at (.47,.1){$\bullet$};\node at (.53,.1){\tiny$z_2$};
  \node at (.67,.5){$\bullet$};\node at(.73,.5){\tiny$z_3$};
  \node at (.57,.9){$\bullet$};\node at (.63,.9){\tiny$z_4$};
  \node at (.3,.9){$\cX$};  \node at (.9,.2){$\cY$};
  \node at (.77,.94){\tiny$U'_1$};
  \node at (.3,.6){\tiny$U'_2$};
  \node at (.7,.3){\tiny$U'_3$};
\begin{scope}[shift={(1.5,0)}]
  \draw[dashed,->] (-.1,0) to (1.1,0);
  \draw[dashed,->] (0,-.2) to (0,1.1);
  \fill[black, opacity=.2, looseness=.8] (.1,.1) to[out=180, in=-90] (-.1,.3) to[out=90,in=-100] (.4,.9) to[out=80,in=-180]
  (.8,1) to[out=0,in=0] (.8,0) to[out=180,in=0] (.4,0) to[out=180,in=0] (.1,.1);
  \fill[black, opacity=.3, looseness=.8] (.8,0) to[out=180,in=-100] (.4,.9) to[out=80,in=-180]
  (.8,1) to[out=0,in=0] (.8,0);
  \fill[pattern=vertical lines, opacity=.6, rotate around={30:(.6,.45)}] (.6,.45) ellipse (.2cm and .1cm);
  \fill[pattern=vertical lines, opacity=.6] (.47,.0) ellipse (.1cm and .2cm);
  \fill[pattern=vertical lines, opacity=.6, rotate around={120:(.2,.3)}] (.2,.3) ellipse (.2cm and .1cm);
  \fill[pattern=vertical lines, opacity=.6] (.6,.9) ellipse (.12cm);
  \node at (.2,.3){$\bullet$};\node at (.27,.3){\tiny$z_1$};
  \node at (.47,.1){$\bullet$};\node at (.53,.1){\tiny$z_2$};
  \node at (.67,.5){$\bullet$};\node at(.73,.5){\tiny$z_3$};
  \node at (.57,.9){$\bullet$};\node at (.63,.9){\tiny$z_4$};
  \node at (.3,.9){$\cX$};  \node at (.9,.2){$\cY$};
  \node at (.77,.94){\tiny$U_4$};
  \node at (.5,-.1){\tiny$U_2$};
  \node at (.2,.2){\tiny$U_1$};
  \node at (.7,.32){\tiny$U_3$};
\end{scope}
 \end{tikzpicture}
 \caption{On left, a nice couple $\fT=(\cX,\cY)$, and a covering $\uU'$ of a configuration $P\subset\cX$;
 on right, an adapted covering $\uU$ of $P$.}
 \label{fig:uUcovering}
\end{figure}

\begin{nota}
 \label{nota:uUcovering}
Let $\uU=(U_1,\dots,U_\kappa)$ be a covering of $P$ as in Definition \ref{defn:uUcovering}.
By abuse of notation we denote also by $\uU$ the union $U_1\cup\dots\cup U_\kappa\subset\C$.
We assume that there is $0\leq \lambda\leq \kappa$ such that $U_1,\dots,U_\lambda$ are precisely
the open sets of $\uU$ with $\bar{U}_i$ disjoint from $\cY$.
If $\uU$ is an adapted covering of $P$, using Notation \ref{nota:P},
we also assume $z_i\in U_i$ for $1\leq i\leq k$.
\end{nota}
The notion of covering is classically used to give a topology to the Ran space $\Ran(\cX)$, as we also recall in Subsection \ref{subsec:Ranspaces}.
Intuitively, a perturbation of a configuration $P$ will be a new configuration $P'$ obtained by slightly moving the points of $P$ and by splitting some points $z_i\in P$ in two or more points of $P'$;
all these splittings occur inside an adapted covering $\uU$, which is also a covering of $P'$.

\subsection{Fundamental group and admissible generating sets}
\label{subsec:admgenset}
\label{subsec:fundamentalgroup}
\begin{defn}
 \label{defn:fGP}
 Let $P$ be as in Notation \ref{nota:P}. The \emph{fundamental group of }$P$,
 denoted $\fG(P)$, is by definition the fundamental group of the complement of $P$ in the plane:
 \[
  \fG(P):=\pi_1(\CmP,*).
 \]
\end{defn}
 For $P$ as in Notation \ref{nota:P}, the group $\fG(P)$ is a free group on $k$ generators:
 in the following we construct an explicit set of free generators for $\fG(P)$. See Figure \ref{fig:admgenset}.
 
\begin{figure}[ht]
 \begin{tikzpicture}[scale=4,decoration={markings,mark=at position 0.38 with {\arrow{>}}}]
  \draw[dashed,->] (-.1,0) to (1.1,0);
  \draw[dashed,->] (0,-1.1) to (0,1.1);
  \node at (0,-1) {$*$};
  \fill[black, opacity=.2, looseness=.8] (.1,.1) to[out=180, in=-90] (-.1,.3) to[out=90,in=-100] (.4,.9) to[out=80,in=-180]
  (.8,1) to[out=0,in=0] (.8,0) to[out=180,in=0] (.4,0) to[out=180,in=0] (.1,.1);
  \fill[black, opacity=.3, looseness=.8] (.8,0) to[out=180,in=-100] (.4,.9) to[out=80,in=-180]
  (.8,1) to[out=0,in=0] (.8,0);
  \node at (.17,.3){$\bullet$};\node at (.23,.3){\tiny$z_1$};  \draw (.2,.3) ellipse (.1cm);
  \node at (.46,.1){$\bullet$};\node at (.52,.1){\tiny$z_2$};  \draw (.5,.1) ellipse (.07cm);
  \node at (.67,.5){$\bullet$};\node at (.73,.5){\tiny$z_3$};  \draw (.7,.5) ellipse (.1cm);
  \node at (.57,.9){$\bullet$};\node at (.63,.9){\tiny$z_4$};  \draw (.6,.9) ellipse (.1cm);
  \node at (.3,.9){$\cX$};  \node at (.93,.2){$\cY$};
  \draw (0,-1) to[out=80,in=-90] node{\tiny$\arc_2$} (.5,.03);
  \draw (0,-1) to[out=50,in=-90] node{\tiny$\arc_3$} (.7,.4);
  \draw (0,-1) to[out=10,in=-10] node{\tiny$\arc_4$}(.7,.9);
  \draw (0,-1) to[out=20,in=0] (.6,.7) node{\tiny$\arc_1$} to[out=180,in=60] (.2,.4);
\begin{scope}[shift={(1.3,0)}]
  \draw[dashed,->] (-.1,0) to (1.1,0);
  \draw[dashed,->] (0,-1.1) to (0,1.1);
  \node at (0,-1) {$*$};
  \fill[black, opacity=.2, looseness=.8] (.1,.1) to[out=180, in=-90] (-.1,.3) to[out=90,in=-100] (.4,.9) to[out=80,in=-180]
  (.8,1) to[out=0,in=0] (.8,0) to[out=180,in=0] (.4,0) to[out=180,in=0] (.1,.1);
  \fill[black, opacity=.3, looseness=.8] (.8,0) to[out=180,in=-100] (.4,.9) to[out=80,in=-180]
  (.8,1) to[out=0,in=0] (.8,0);
  \node at (.17,.3){$\bullet$};\node at (.23,.3){\tiny$z_1$};
  \node at (.46,.1){$\bullet$};\node at (.52,.1){\tiny$z_2$};
  \node at (.67,.5){$\bullet$};\node at (.73,.5){\tiny$z_3$};
  \node at (.57,.9){$\bullet$};\node at (.63,.9){\tiny$z_4$};
  \node at (.3,.9){$\cX$};  \node at (.93,.2){$\cY$};
  \draw[thin, looseness=1.3, postaction={decorate}] (0,-1) to[out=80,in=-90] node[left]{\tiny$\gen_2$} (.42,.05) to[out=90,in=-90] (.34,.12)  to[out=90,in=90] (.57,.12)  to[out=-90,in=90] (.47,.05) to[out=-90,in=75] (0,-1);
  \draw[thin, looseness=1.2, postaction={decorate}] (0,-1) to[out=60,in=-90] node[right]{\tiny$\gen_3$} (.65,.45)
  to[out=90,in=-90] (.55,.5) to[out=90,in=90] (.77,.5) to[out=-90,in=90] (.68,.45) to[out=-90,in=55]  (0,-1);
  \draw[thin, looseness=1.2, postaction={decorate}] (0,-1) to[out=10,in=0] (.7,.85) to[out=180,in=0] (.6,.77) node[right]{\tiny$\gen_4$}
  to[out=180,in=180] (.58,.96) to[out=0,in=180] (.7,.88) to[out=0,in=5] (0,-1);
  \draw[thin, looseness=1.2, postaction={decorate}] (0,-1) to[out=25,in=0] (.6,.65) node{\tiny$\gen_1$} to[out=180,in=90] (.18,.38)
  to[out=-90,in=90] (.25,.3) to[out=-90,in=-90] (.05,.3) to[out=90,in=-90] (.15,.4) to[out=90,in=180] (.6,.7)  to[out=0,in=20] (0,-1);
\end{scope}
 \end{tikzpicture}
 \caption{On left, a nice couple $\fT=(\cX,\cY)$, a configuration $P\subset\cX$, the boundary curves
 of an adapted covering $\uU$ of $P$ and a choice of
 arcs $\arc_i$; on right, the loops representing the corresponding admissible generating set of $\fG(P)$.
}
 \label{fig:admgenset}
\end{figure}

 We choose an adapted covering $\uU$ of $P$, and use Notation \ref{nota:uUcovering}.
 The boundary curves of $\bar{U}_1,\dots,\bar{U}_k$ are denoted by
 $\partial U_1,\dots, \partial U_k$ respectively,
 and are oriented clockwise. We also choose embedded
 arcs $\arc_1,\dots,\arc_k$ joining the basepoint $*$ with the curves $\partial U_1,\dots,\partial U_k$. We assume the following:
 \begin{itemize}
  \item  for all $1\leq i\leq k$, the arc $\arc_i$ has an endpoint at $*$, and the other endpoint on
  $\partial U_i$;
  \item there is no other intersection point between two arcs $\arc_i,\arc_j$ or between an arc $\arc_i$ and a boundary curve $\del U_j$, for $1\le i,j\le k$.
 \end{itemize}
 For $1\leq i\leq k$
 let $\gen_i\in\fG(P)$ be the element represented by a loop that begins at
 $*$, runs along $\arc_i$ until it reaches the intersection with $\partial U_i$, spins clockwise around $\partial U_i$
 and runs back to $*$ along $\arc_i$.
 Then $\gen_1,\dots,\gen_k$ exhibit $\fG(P)$ as a free group on $k$ generators.
\begin{defn}
  \label{defn:admgenset}
  Let $P$ be as in Notation \ref{nota:P}. A set of generators $\gen_1,\dots,\gen_k$ of $\fG(P)$
  obtained as described above is called an \emph{admissible generating set}.
\end{defn}

\subsection{Fundamental PMQ}
\label{subsec:fundamentalPMQ}
In the following we introduce several PMQs arising as subsets of $\fG(P)$, for $P$ as in Notation \ref{nota:P}.
Recall that a conjugacy class in $\fG(P)$ corresponds to a free (i.e. unbased) homotopy
class of maps $S^1\to\CmP$.
\begin{defn}
\label{defn:fQP}
 Let $P$ be as in Notation \ref{nota:P}.
 For all $1\leq i\leq l$ we denote by $\fQ(P,z_i)\subset\fG(P)$ the conjugacy class
 corresponding to a small (unbased) simple closed curve that spins once, clockwise, around $z_i$;
 we define
 \[
 \fQ(P)=\fQ_{\fT}(P):=\set{\one}\cup\bigcup_{1\leq i\leq l} \fQ(P,z_i)\subset \fG(P),
 \]
 and call it the \emph{fundamental PMQ} of $P$ relative to the nice couple $\fT$. We consider on $\fQ(P)$ the PMQ structure inherited
 from $\fG(P)$ (see \cite[Definition 2.8]{Bianchi:Hur1}).
\end{defn}
Note that a \emph{based} loop representing a class in $\fQ(P)$ may intersect essentially $\cY$.
\begin{figure}[ht]
 \begin{tikzpicture}[scale=4,decoration={markings,mark=at position 0.38 with {\arrow{>}}}]
  \draw[dashed,->] (-.1,0) to (1.1,0);
  \draw[dashed,->] (0,-1.1) to (0,1.1);
  \node at (0,-1) {$*$};
  \fill[black, opacity=.2, looseness=.8] (.1,.1) to[out=180, in=-90] (-.1,.3) to[out=90,in=-100] (.4,.9) to[out=80,in=-180]
  (.8,1) to[out=0,in=0] (.8,0) to[out=180,in=0] (.4,0) to[out=180,in=0] (.1,.1);
  \fill[black, opacity=.3, looseness=.8] (.8,0) to[out=180,in=-100] (.4,.9) to[out=80,in=-180]
  (.8,1) to[out=0,in=0] (.8,0);
  \node at (.17,.3){$\bullet$};\node at (.23,.3){\tiny$z_1$};
  \node at (.46,.1){$\bullet$};\node at (.52,.1){\tiny$z_2$};
  \node at (.67,.5){$\bullet$};\node at (.73,.5){\tiny$z_3$};
  \node at (.57,.9){$\bullet$};\node at (.63,.9){\tiny$z_4$};
  \node at (.3,.9){$\cX$};  \node at (.85,.2){$\cY$};
  \draw[thin, looseness=1.3, postaction={decorate}] (0,-1) to[out=80,in=-90]  (.42,-.05) to[out=90,in=-90] (.34,.12)  to[out=90,in=90] (.57,.12)  to[out=-90,in=90] (.47,-.05) to[out=-90,in=75] (0,-1);
  \draw[thin, looseness=1.3, postaction={decorate}, dotted] (.34,.1)  to[out=90,in=90] (.57,.1)  to[out=-90,in=-90] (.34,.1) ;
  \draw[thin, looseness=1.2, postaction={decorate}] (0,-1) to[out=25,in=0] (.6,.65) to[out=180,in=180] (.6,1.05) to[out=0,in=0]
  (.6,.55) to[out=180,in=90] (.18,.38)
  to[out=-90,in=90] (.28,.3) to[out=-90,in=-90] (.05,.3) to[out=90,in=-90] (.15,.4) to[out=90,in=180] (.6,.6) to[out=0,in=0]
  (.6,1) to[out=180,in=180] (.6,.7)  to[out=0,in=20] (0,-1);
\draw[thin, looseness=1.3, postaction={decorate}, dotted] (.1,.3)  to[out=90,in=90] (.3,.3) to[out=-90,in=-90] (.1,.3) ;
  \begin{scope}[shift={(1.3,0)}]
  \draw[dashed,->] (-.1,0) to (1.1,0);
  \draw[dashed,->] (0,-1.1) to (0,1.1);
  \node at (0,-1) {$*$};
  \fill[black, opacity=.2, looseness=.8] (.1,.1) to[out=180, in=-90] (-.1,.3) to[out=90,in=-100] (.4,.9) to[out=80,in=-180]
  (.8,1) to[out=0,in=0] (.8,0) to[out=180,in=0] (.4,0) to[out=180,in=0] (.1,.1);
  \fill[black, opacity=.3, looseness=.8] (.8,0) to[out=180,in=-100] (.4,.9) to[out=80,in=-180]
  (.8,1) to[out=0,in=0] (.8,0);
  \node at (.17,.3){$\bullet$};\node at (.23,.3){\tiny$z_1$};
  \node at (.46,.1){$\bullet$};\node at (.4,.1){\tiny$z_2$};
  \node at (.67,.5){$\bullet$};\node at (.73,.5){\tiny$z_3$};
  \node at (.57,.9){$\bullet$};\node at (.63,.9){\tiny$z_4$};
  \node at (.3,.9){$\cX$};  \node at (.85,.2){$\cY$};
  \draw[thin, looseness=1.2, postaction={decorate}] (0,-1) to[out=25,in=0] (.6,.6) to[out=180,in=90] (.5,.1) to[out=-90,in=-90] (.1,.3) to[out=90,in=-180] (.6,.65) to[out=0,in=20] (0,-1);   
  \draw[thin, looseness=1.2, postaction={decorate}, dotted] (.4,.4)
  to[out=-50,in=90] (.5,.12) to[out=-90,in=-90] (.12,.3) to[out=90,in=130] (.4,.4);   
  \end{scope}
\end{tikzpicture}
\caption{On left, two elements in $\fQ(P)$, lying in $\fQ(P,z_1)$ and $\fQ(P,z_2)$. On right, an element lying in $\fQext(P)$ but not in $\fQ(P)$.}
\label{fig:fQandfQext}
\end{figure}
Let $P$ be as in Notation \ref{nota:P}, and fix an admissible generating set $\gen_1,\dots,\gen_k$ of $\fG(P)$ (see Definition \ref{defn:admgenset}): then the elements of
$\fQ(P)$ are precisely $\one$ and all conjugates in $\fG(P)$ of the elements $\gen_1,\dots,\gen_l$.
In particular the group isomorphism $\fG(P)\cong\bF^k$, given by choosing an admissible generating set,
restricts to a bijection $\fQ(P)\cong\FQ^k_l$ (see \cite[Definition 3.2]{Bianchi:Hur1}), and the
hypotheses required by \cite[Definition 2.8]{Bianchi:Hur1} are fulfilled.
The partial product of $\fQ(P)$ is trivial, and $(\fQ(P),\fG(P))$ is a PMQ-group pair. See Figure \ref{fig:fQandfQext}, left, for examples of elements in $\fQ(P)$.

\subsection{Extended fundamental PMQ}
\label{subsec:extendedfundamentalPMQ}
We extend Definition \ref{defn:fQP} by considering more general simple closed curves.
\begin{defn}
 \label{defn:fQextP}
Let $P$ be as in Notation \ref{nota:P}.
We denote by $\fQext(P)=\fQext_{\fT}(P)\subset\fG(P)$ the union of all conjugacy classes corresponding to (unbased) oriented simple
closed curves $\beta\subset\C\setminus\cY$, such that $\beta$ spins clockwise and $\beta$ bounds
a disc contained in $\C\setminus\cY$. We consider on $\fQext(P)$ the PMQ structure inherited from $\fG(P)$,
and call it the \emph{extended fundamental PMQ} of $P$ relative to the nice couple $\fT$.
\end{defn}
Note that there is an inclusion of sets $\fQ(P)\subseteq\fQext(P)$. See Figure \ref{fig:fQandfQext}, right,
for a non-trivial example of an element in $\fQext(P)$.
The fact that the hypotheses of \cite[Definition 2.8]{Bianchi:Hur1} are fulfilled by $\fQext(P)\subseteq\fG(P)$
needs some explanation: this is contained in the following two propositions. In the following, recall from \cite[Definition 3.5]{Bianchi:Hur1} that a \emph{decomposition} of an element
$\fg\in\fQext(P)$ \emph{with respect to }$\fQ(P)$ is a sequence
$(\fg_1,\dots,\fg_r)$ of elements of $\fQ(P)$ whose product $\fg_1\dots\fg_r$, computed in $\fG(P)$, is equal to $\fg$.

\begin{prop}
 \label{prop:fQgeneratesfQext} 
 Let $P$ be as in Notation \ref{nota:P}.
 The set $\fQext(P)$ is generated under partial product by $\fQ(P)$, i.e.,
 every element $\fg$ of $\fQext(P)$ admits a decomposition $(\fg_1,\dots,\fg_r)$
 with respect to $\fQ(P)$.
\end{prop} 
\begin{prop}
 \label{prop:fQextwelldefined}
 Let $P$ be as in Notation \ref{nota:P}. Then $\fQext(P)\subset\fG(P)$ satisfies
 the hypotheses of \cite[Definition 2.8]{Bianchi:Hur1}, and hence inherits a structure
 of PMQ; as a consequence $(\fQext(P),\fG(P))$ is a PMQ-group pair.
\end{prop}
The proofs of Propositions \ref{prop:fQgeneratesfQext} and \ref{prop:fQextwelldefined} are in Subsections \ref{subsec:fQgeneratesfQext} and \ref{subsec:fQextwelldefined} of the appendix, respectively.

It follows that the inclusion $\fQ(P)\subseteq\fQext(P)$ is a map of PMQs, and
the inclusion $(\fQ(P),\fG(P))\subset(\fQext(P),\fG(P))$ is a map of PMQ-group pairs. In the following we study the problem
of extending over $\fQext(P)$ maps of PMQs defined over $\fQ(P)$.
\begin{defn}
\label{defn:fQextPpsi}
Let $P$ be as in Notation \ref{nota:P}, let $\Q$ be a PMQ and let $\psi\colon\fQ(P)\to\Q$ be a map of PMQs.
Let $\fg\in\fQext(P)$ and let $(\fg_1,\dots,\fg_r)$ be a decomposition of $\fg$ with respect to $\fQ(P)$.
We say that $\psi$ \emph{can be extended over} $\fg$ if the product $\psi(\fg_1)\dots\psi(\fg_r)$ is defined in $\Q$.
We denote by $\fQext(P)_{\psi}=\fQext_{\fT}(P)_{\psi}\subseteq\fQext(P)$ the subset containing all elements $\fg$ over which $\psi$ can be extended.
\end{defn}
Some comments on Definition \ref{defn:fQextPpsi} are needed.
For $\fg\in\fQext(P)$, the existence of a decomposition $(\fg_1,\dots,\fg_r)$ of $\fg$ with respect to $\fQ(P)$
is granted by Proposition \ref{prop:fQgeneratesfQext}. This decomposition is in general not unique;
nevertheless, by \cite[Proposition 3.7]{Bianchi:Hur1},
if $(\fg'_1,\dots,\fg'_r)$ is another decomposition of $\fg$ with respect to $\fQ(P)$, then the two decompositions
are connected by a sequence of standard moves (see \cite[Definition 3.6]{Bianchi:Hur1}).
Since $\psi$ is a map of PMQs, we obtain that the sequence $(\psi(\fg_1),\dots,\psi(\fg_r))$ of elements of $\Q$
can be transformed into the sequence $(\psi(\fg'_1),\dots,\psi(\fg'_r))$ by a sequence of standard moves;
it is then a direct consequence of the definition of PMQ,
that the product $\psi(\fg_1)\dots\psi(\fg_r)$ is defined if and only
if the product $\psi(\fg'_1)\dots\psi(\fg'_r)$ is defined, and if both products are defined then
they are equal to each other.
This shows that, whether $\psi $ can be extended over $\fg$, only depends on the element $\fg$ but not on the
decomposition $(\fg_1,\dots,\fg_r)$ of $\fg$ with respect to $\fQ(P)$; moreover
the assignment $\fg\mapsto \psi(\fg_1)\dots\psi(\fg_r)$ gives a well-defined 
map of sets $\psiext\colon\fQext(P)_{\psi}\to\Q$, which extends the map $\psi\colon\fQ(P)\to\Q$.
\begin{prop}
\label{prop:fQextPpsi}
The subset $\fQext(P)_{\psi}\subseteq\fG(P)$ satisfies the requirements
of \cite[Definition 2.8]{Bianchi:Hur1}, and therefore $\fQext(P)_{\psi}$ inherits a structure of PMQ.
The map $\psiext\colon\fQext(P)_{\psi}\to\Q$ is the unique map of PMQs
$\fQext(P)_{\psi}\to\Q$ restricting to $\psi\colon\fQ(P)\to\Q$ on $\fQ(P)$.
\end{prop}
The proof of Proposition \ref{prop:fQextPpsi} is in Subsection \ref{subsec:fQextPpsi} of the appendix.
As a consequence of Proposition \ref{prop:fQextPpsi}, $(\fQext(P)_{\psi},\fG(P))$ is naturally
a PMQ-group pair.

\subsection{PMQs from coverings}
\label{subsec:coverings}
We extend the previous definitions by replacing a configuration of points
$P$ with a configuration of open, convex sets $\uU$ in $\C$.

\begin{defn}
\label{defn:fGUfQPU}
Let $\uU$ be a covering of $P$ (see Definition \ref{defn:uUcovering}).
We define $\fG(\uU)$ as the group $\pi_1\pa{\C\setminus\uU,*}$,
and call it the \emph{fundamental group} of $\uU$.

We use Notation \ref{nota:uUcovering} and define $\fQ(\uU)\subset\fG(\uU)$ as the union of $\set{\one}$
and the $\lambda$ conjugacy classes corresponding to the simple closed curves $\partial U_1,\dots,\partial U_\lambda$, oriented
clockwise. Similarly as in Definition \ref{defn:fQP}, we consider on $\fQ(\uU)$ the PMQ structure
inherited from $\fG(\uU)$. The PMQ $\fQ(\uU)$ is called the \emph{fundamental PMQ} of $\uU$.

Finally, we define $\fQ(P,\uU)=\fQ_{\fT}(P,\uU)\subset\fQext(P)$
as the union of all conjugacy classes in $\fG(P)$ represented by simple closed curves $\beta$
which are oriented clockwise and are contained in one of the regions $U_i\setminus P\subset\C\setminus\cY$,
for some $1\leq i\leq \lambda$.
The set $\fQ(P,\uU)$ is called the \emph{relative fundamental PMQ} of $P$ with respect to $\uU$.
\end{defn}
Note that, for $\uU$ as in Notation \ref{nota:uUcovering}, $\fG(\uU)$ is a free group on $\kappa$ generators,
and an admissible generating set $\gen_1,\dots,\gen_\kappa$ can be constructed in the same way as in
Subsection \ref{subsec:fundamentalgroup} to give an isomorphism $\fG(\uU)\cong\bF^\kappa$.
By the same arguments used in Subsection \ref{subsec:fundamentalPMQ}, the previous identification
restricts to an identification $\fQ(\uU)\cong\FQ^\kappa_\lambda$, and therefore, analogously
as in the case of $\fQ(P)\subseteq\fG(P)$, the set $\fQ(\uU)$ inherits from $\fG(\uU)$ a structure of PMQ,
and $(\fQ(\uU),\fG(\uU))$ is a PMQ-group pair.

\begin{lem}
 \label{lem:fQPUinherits}
Using the notation above, the set $\fQ(P,\uU)\subset\fG(P)$ inherits a structure of PMQ
from $\fG(P)$ in the sense of  \cite[Definition 2.8]{Bianchi:Hur1}, and thus
$(\fQ(P,\uU),\fG(P))$ is a PMQ-group pair.
\end{lem}
\begin{proof}
We have to check that if $\fg=\fg_1\dots\fg_r$ is a decomposition of $\fg\in\fQ(P,\uU)$
with all factors $\fg_j\in\fQ(P,\uU)$, then for each $1\le j\le j'\le r$ the product $\fg_j\dots\fg_{j'}$ also lies in $\fQ_{\fT}(P,\uU)$. Let $\gen_1,\dots,\gen_k$ be an admissible generating set for $\fG(P)$, and 
consider the abelianisation map $\ab\colon\fG(P)\to\fG(P)^{\ab}\cong\Z^k$, where the latter isomorphism is given by the basis $\ab(\gen_1),\dots,\ab(\gen_k)$.
For each $1\le i\le \lambda$ consider the nice couple $\fT_i=(\bH,\bH\setminus U_i)$. Then
$\fQ_{\fT}(P,\uU)\subset\fG(P)$ can be identified with the subset $\bigcup_{i=1}^\lambda\fQext_{\fT_i}(P)\subset\fG(P)$.
Moreover $\ab(\fg)$ is a vector with all entries equal to 0 or 1, as $\fg$ is in the conjugacy class represented by a simple closed curve spinning clockwise; if $\fg=\one$,
then all entries are zero, and if $\fg\in\fQext_{\fT_i}(P)\setminus\set{\one}$ for some $1\le i\le \lambda$ then at least one entry
is equal to 1, and all entries equal to 1 correspond to generators $\gen_j$
spinning around points of $P\cap U_i$.
Since $\ab(\fg)=\ab(\fg_1)+\dots\ab(\fg_r)$ is a sum with no cancelation in any coordinate, we conclude that 
there is a unique $1\le i\le \lambda$ such that $\fg_1,\dots,\fg_r\in\fQext_{\fT_i}(P)$ (such $\lambda$ is unique unless all $\fg_j=\one$).
We can now apply Proposition \ref{prop:fQextwelldefined} to $\fQext_{\fT_i}(P)$
to show that, for all $1\le j\le j'\le r$ the product $\fg_j\dots\fg_{j'}$ also lies in $\fQext_{\fT_i}(P)$,
and hence in $\fQ_{\fT}(P,\uU)$.
\end{proof}

We conclude the subsection by analysing which inclusions hold between the groups and PMQs introduced so far.
The inclusion $\C\setminus\uU\subset\CmP$ induces an injection of PMQ-group pairs $(\fQ_\fT(\uU),\fG(\uU))\subseteq(\fQ_\fT(P,\uU),\fG(P))$.

On the other hand we have a chain of inclusions $\fQ_\fT(P,\uU)\subset\fQext_\fT(P)\subset\fG(P)$; observe that for a generic covering $\uU$ we do not have an inclusion $\fQ_\fT(P)\subset\fQ_\fT(P,\uU)$;
yet Proposition \ref{prop:fQgeneratesfQext} specialises to the fact
that $\fQ(P,\uU)$ is generated by $\fQ(P)\cap\fQ(P,\uU)$ under partial multiplication.

\subsection{Maps induced by forgetting points}
\begin{nota}
\label{nota:fri}
Let $\fT=(\cX,\cY)$ be a nice couple and let $P\subseteq P'\subset\cX$. We denote by
$\fri^{P'}_P\colon\fG(P')\to\fG(P)$ the map induced by the inclusion $\CmP'\subseteq\CmP$.
\end{nota}
 
Note that $\fri^{P'}_P$ restricts to maps $\fQ(P')\to\fQ(P)$ and $\fQext(P')\to\fQext(P)$.
To see this, let $[\gamma]\in\fQ(P')$ (respectively $[\gamma]\in\fQext(P')$)
be represented by a loop $\gamma$ which is freely homotopic in $\CmP'$ to a simple curve $\beta\subset\C\setminus(P'\cup\cY)$ spinning clockwise around at most one point (respectively, some points) of $P'\setminus\cY$;
then the same properties hold for $\fri^{P'}_P([\gamma])$ in $\fQ(P)$ (respectively, in $\fQext(P)$).

\section{Hurwitz-Ran spaces with monodromies in a PMQ-group pair}
\label{sec:defnHurPMQ}
In this section we define, for a PMQ-group pair $(\Q,G)$
and a nice couple $\fT=(\cX,\cY)$ as in Definition \ref{defn:nicecouple}, the
Hurwitz-Ran space $\Hur(\fT;\Q,G)$, containing configurations of points in $\fT$ with monodromies in $(\Q,G)$.

Throughout the section we fix a PMQ-group pair $(\Q,G)=(\Q,G,\fe,\fr)$
as in \cite[Definition 2.15]{Bianchi:Hur1}
and a nice couple $\fT=(\cX,\cY)$ as in Definition \ref{defn:nicecouple}.

\subsection{Ran spaces}\label{subsec:Ranspaces}
We recall the definition and the main properties of the Ran space $\Ran(\cX)$, focusing on the case
of a connected subspace $\cX\subset\bH$. We use \cite[Subsection 5.5.1]{LurieHA} as main reference.
\begin{defn}
 \label{defn:Ran}
 Let $\cX\subset\bH$ be a subspace. We define $\Ran(\cX)$ as the set of all finite subsets
 $P\subset\cX$, including $\emptyset$; we denote by $\Ran_+(\cX)$ the set $\Ran(\cX)\setminus\set{\emptyset}$.
 
 We define a topology on $\Ran(X)$. For $P\in\Ran(\cX)$ and $\uU$ an adapted covering of $P$ with respect
 to the nice couple $(\cX,\emptyset)$ (see Definition \ref{defn:uUcovering}), we let
 $\fU(P,\uU)=\fU_{\cX}(P,\uU)\subset\Ran(\cX)$
 be the subset of all $P'\in\Ran(\cX)$ satisfying the following:
 \begin{itemize}
  \item $P'\subset\uU$;
  \item $P'\cap U_i\neq\emptyset$ for all $1\leq i\leq \kappa$, using Notation \ref{nota:uUcovering}.
 \end{itemize}
 A subset of the form $\fU(P,\uU)$ is called a \emph{normal neighbourhood} of $P$ in $\Ran(\cX)$.
 Normal neighbourhoods form the basis of a Hausdorff topology on $\Ran(\cX)$.
 
 For $\emptyset\neq P_0\subset\cX$ we denote by $\Ran(\cX)_{P_0}\subset\Ran_+(\cX)$ the subspace containing
 all $P\subset\cX$ with $P_0\subseteq P$.
 Similarly, for $z_0\in\cX$ we denote $\Ran(\cX)_{z_0}=\Ran(\cX)_{\set{z_0}}$.
\end{defn}
Our definition of $\Ran(\cX)$ differs from the usual one in the literature (e.g. \cite[Definition 5.5.1.2]{LurieHA}) because we allow also $\emptyset$ as a point in $\Ran(\cX)$. Note however
that our $\Ran(\cX)$ is the topological disjoint union of the singleton
$\set{\emptyset}$ and $\Ran_+(\cX)$.

The following results are originally due to Beilinson
and Drinfeld \cite{BDChiral}.
See also \cite[Lemma 5.5.1.8, Theorem5.5.1.6]{LurieHA}.
\begin{lem}
 \label{lem:RancontractibleP0}
 Let $\cX\subset\bH$ be path connected and let $P_0\subset\cX$ be a finite non-empty subset.
 Then $\Ran(\cX)_{P_0}$ is weakly contractible.
\end{lem}
\begin{lem}
\label{lem:Ran+contractible}
 Let $\cX\subset\bH$ be path connected. Then $\Ran_+(\cX)$ is weakly contractible.
\end{lem}
\begin{nota}
 \label{nota:RanfT}
 For a nice couple $\fT=(\cX,\cY)$ we write $\Ran(\fT)=\Ran(\cX)$ and similarly for the subspaces
 introduced in Definition \ref{defn:Ran}.
\end{nota}

\subsection{Hurwitz sets}
\label{subsec:Hurwitzspaces}
We first define $\HurTQG$ as a set, and discuss later in Subsection \ref{subsec:Hurtopology} its Ran topology, mimicking the topology on $\Ran(\fT)$.
\begin{defn}
 \label{defn:Hurset}
Let $(\Q,G)$ be a PMQ-group pair and let $\fT=(\cX,\cY)$ a nice couple.
An element of the \emph{Hurwitz set} $\HurTQG$ is a configuration $\fc=(P,\psi,\phi)$, where
\begin{itemize}
 \item $P=\set{z_1,\dots,z_k}$ is a finite subset of $\cX$, i.e. $P\in\Ran(\cX)$;
 \item $(\psi,\phi)\colon(\fQ(P),\fG(P))\to(\Q,G)$ is a map of PMQ-group pairs (see Definitions \ref{defn:fGP} and  \ref{defn:fQP} for the PMQ-group pair $(\fQ(P),\fG(P))$).
\end{itemize}
If $\fc=(P,\psi,\phi)$, we say that $\fc$ is \emph{supported on }$P$; if $\cS$ is a subspace of $\cX$ and
$P\subset\cS$, we say that $\fc$ is \emph{supported in }$\cS$. The maps $\psi$ and $\phi$ are the \emph{$\Q$-valued} and \emph{$G$-valued monodromies} of $\fc$.
\end{defn}

Roughly speaking, the monodromy $\phi$, with values in $G$, is defined
around \emph{every point} of $P$, whereas the monodromy $\psi$, with values in $\Q$, is defined only around
points of $P$ which lie in $\cX\setminus\cY$. We can think of $\psi$ as a \emph{refinement of $\phi$
away from $\cY$}: indeed the composition $\fe\circ\psi\colon\fQ(P)\to G$ is equal to $\phi|_{\fQ(P)}$,
where the map of PMQs $\fe\colon\Q\to G$ is part of the structure of PMQ-group pair of $(\Q,G)$.
\begin{nota}
 \label{nota:fc}
We usually present a configuration $\fc\in\HurTQG$ as $\fc=(P,\psi,\phi)$ and use Notation \ref{nota:P} for $P$.
Similarly we present another configuration $\fc'$ as $(P',\psi',\phi')$, and write $P'=\set{z'_1,\dots,z'_{k'}}$.
\end{nota}

\subsection{The topology on Hurwitz-Ran spaces}
\label{subsec:Hurtopology}
We introduce a topology on the set $\HurTQG$, in the spirit of the topology of the Ran space $\Ran(\fT)$.\begin{defn}
  \label{defn:normalneigh}
We use Notations \ref{nota:fc} and \ref{nota:uUcovering}.
Let $\fc\in\HurTQG$ and let $\uU$ be an adapted covering of $P$; 
we denote by $\fU(\fc;\uU)=\fU_{\fT}(\fc;\uU)$ the subset of $\HurTQG$ containing all configurations
$\fc'$ satisfying the following conditions:
\begin{itemize}
 \item $P'\in\fU(P,\uU)$ (see Definition \ref{defn:Ran});
 as a consequence there is a natural inclusion of PMQ-group pairs
 $(\fQ(\uU),\fG(\uU))\subseteq(\fQ(P',\uU),\fG(P'))$ (see Definition \ref{defn:fGUfQPU});
 \item $\fQ(P',\uU)$ is contained in $\fQext(P')_{\psi'}$ (see Definition \ref{defn:fQextPpsi}),
 and the following composition of maps of PMQ-group pairs is equal to $(\psi,\phi)$: 
 \[
 \begin{tikzcd}
  (\fQ(P),\fG(P)) & (\fQ(\uU),\fG(\uU)) \ar[l,"\cong"'] \ar[r,hook] & (\fQ(P',\uU),\fG(P')) \ar[ld,"\subseteq"'] \\
  & (\fQext(P')_{\psi'},\fG(P')) \ar[r,"{(\psipext,\phi')}"] & (\Q,G),
 \end{tikzcd}
 \]
 where we use the isomorphism discussed in the remark after Definition \ref{defn:fGUfQPU},
 and the map $\psipext\colon\fQext(P')_{\psi'}\to\Q$ from Proposition \ref{prop:fQextPpsi}.
\end{itemize}
 Each subset $\fU(\fc;\uU)$ is called a \emph{normal neighbourhood}
 of $\fc$ in $\HurTQG$.
\end{defn}

\begin{figure}[ht]
 \begin{tikzpicture}[scale=4,decoration={markings,mark=at position 0.38 with {\arrow{>}}}]
  \draw[dashed,->] (-.1,0) to (1.1,0);
  \draw[dashed,->] (0,-1.1) to (0,1.1);
  \node at (0,-1) {$*$};
  \fill[black, opacity=.2, looseness=.8] (.1,.1) to[out=180, in=-90] (-.1,.3) to[out=90,in=-100] (.4,.9) to[out=80,in=-180]
  (.8,1) to[out=0,in=0] (.8,0) to[out=180,in=0] (.4,0) to[out=180,in=0] (.1,.1);
  \fill[black, opacity=.3, looseness=.8] (.8,0) to[out=180,in=-100] (.4,.9) to[out=80,in=-180]
  (.8,1) to[out=0,in=0] (.8,0);
  \node at (.1,.3){$\bullet$}; 
  \node at (.45,.5){$\bullet$}; 
  \node at (.5,.9){$\bullet$};
  \fill[pattern=vertical lines, opacity=.6] (.07,.3) ellipse (.15cm and .05cm);
  \fill[pattern=vertical lines, opacity=.6] (.45,.52) ellipse (.2cm and .14cm);
  \fill[pattern=vertical lines, opacity=.6] (.5,.9) ellipse (.1cm and .14cm);
  \draw[thin, looseness=1.7, postaction={decorate}] (0,-1) to[out=85,in=-90] (.05,.1) to[out=90,in=-90]  (-.1,.3)  to[out=90,in=90] node[below]{\tiny$a_1=a'_1a'_2$} (.25,.3)  to[out=-90,in=90]  (.08,.1) to[out=-90,in=80] (0,-1);
  \draw[thin, looseness=1.2, postaction={decorate}] (0,-1) to[out=80,in=-90] node[right]{\tiny$g_2=\fe(a'_3)g'_4$} (.4,.3)
  to[out=90,in=-90] (.2,.5) to[out=90,in=90] (.7,.5) to[out=-90,in=90] (.45,.3) to[out=-90,in=75]  (0,-1);
  \draw[thin, looseness=1.2, postaction={decorate}] (0,-1) to[out=10,in=0] (.7,.85) to[out=180,in=0] (.5,.7)
  to[out=180,in=180] (.5,1.05) node[right]{\tiny$g_3=g'_5$}    to[out=0,in=180] (.7,.88) to[out=0,in=5] (0,-1);
\begin{scope}[shift={(1.4,0)}]
  \draw[dashed,->] (-.1,0) to (1.1,0);
  \draw[dashed,->] (0,-1.1) to (0,1.1);
  \node at (0,-1) {$*$};
  \fill[black, opacity=.2, looseness=.8] (.1,.1) to[out=180, in=-90] (-.1,.3) to[out=90,in=-100] (.4,.9) to[out=80,in=-180]
  (.8,1) to[out=0,in=0] (.8,0) to[out=180,in=0] (.4,0) to[out=180,in=0] (.1,.1);
  \fill[black, opacity=.3, looseness=.8] (.8,0) to[out=180,in=-100] (.4,.9) to[out=80,in=-180]
  (.8,1) to[out=0,in=0] (.8,0);
  \node at (.18,.3){$\bullet$}; 
  \node at (-.05,.3){$\bullet$}; 
  \node at (.6,.52){$\bullet$}; 
  \node at (.3,.48){$\bullet$}; 
  \node at (.45,.83){$\bullet$};
  \fill[pattern=vertical lines, opacity=.4] (.07,.3) ellipse (.15cm and .05cm);
  \fill[pattern=vertical lines, opacity=.4] (.45,.52) ellipse (.2cm and .14cm);
  \fill[pattern=vertical lines, opacity=.4] (.5,.9) ellipse (.1cm and .14cm);
  \draw[thin, looseness=1.2, postaction={decorate}] (0,-1) to[out=87,in=-90] (.02,.1) to[out=90,in=-90] node[left]{\tiny$a'_1$}  (-.1,.3)  to[out=90,in=90] (.08,.3)  to[out=-90,in=90] (.04,.1) to[out=-90,in=84] (0,-1);
  \draw[thin, looseness=1.2, postaction={decorate}] (0,-1) to[out=81,in=-90] (.15,.1) to[out=90,in=-90] (.1,.3)  to[out=90,in=90] (.25,.3)  to[out=-90,in=90] node[left]{\tiny$a'_2$} (.17,.1) to[out=-90,in=78] (0,-1);
  \draw[thin, looseness=1.2, postaction={decorate}] (0,-1) to[out=82,in=-90]  (.4,.3)
  to[out=90,in=-90]  (.2,.5) to[out=90,in=90] (.47,.5) to[out=-90,in=90] node[left]{\tiny$a'_3$} (.43,.3) to[out=-90,in=79]  (0,-1);
  \draw[thin, looseness=1.2, postaction={decorate}] (0,-1) to[out=76,in=-90]  (.55,.3)
  to[out=90,in=-90] node[right]{\tiny$g'_4$} (.5,.5) to[out=90,in=90] (.7,.5) to[out=-90,in=90] (.6,.3) to[out=-90,in=73]  (0,-1);
  \draw[thin, looseness=1.2, postaction={decorate}] (0,-1) to[out=10,in=0] (.7,.85) to[out=180,in=0] (.5,.7)
  to[out=180,in=180] (.5,1.05) node[right]{\tiny$g'_5$}    to[out=0,in=180] (.7,.88) to[out=0,in=5] (0,-1);
\end{scope}
 \end{tikzpicture}
 \caption{On left, a configuration $\fc=(P,\psi,\phi)$ in the space $\Hur(\cX,\cY,\Q,G)$ and an adapted covering $\uU$ of $P$; on right, another configuration
 $\fc'$ in the normal neighbourhood $\fU(\fc,\uU)$. The drawn loops are labelled with their $\Q$-valued monodromy if they belong to
 $\fQ(P)$ and $\fQ(P')$ respectively, and are labelled with their $G$-valued monodromy otherwise.}
\label{fig:normalneigh}
\end{figure}

Roughly speaking, if $\fc'\in\fU(\fc;\uU)$, then $P'$ is obtained from $P$ by splitting
each $z_i$ into $r_i\geq 1$ points $z'_{i,1},\dots,z'_{i,r_i}$ inside the neighbourhood $U_i$ of $z_i$.
For all $1\leq i\leq k$ the $G$-valued monodromy around $z_i$ is decomposed as a product of the $G$-valued
monodromies $\phi'$ around the points $z'_{i,1},\dots,z'_{i,r_i}$;
similarly, for $1\leq i\leq l$ the $\Q$-valued monodromy around $z_i$ is decomposed as a product of the
$\Q$-valued monodromies around the points $z'_{i,1},\dots,z'_{i,r_i}$.
See Figure \ref{fig:normalneigh} for an example of two configuration $\fc$ and $\fc'$ with $\fc'$ in a normal neighbourhood of $\fc$.

\begin{prop}
 \label{prop:Hurtopology}
 The subsets $\fU(\fc;\uU)$ for varying $\fc$ and $\uU$ form the basis of a Hausdorff
 topology on the set $\HurTQG$.
\end{prop}
\begin{proof}
 Let $\fc_1=(P_1,\psi_1,\phi_1)$ and $\fc_2=(P_2,\psi_2,\phi_2)$ denote two configurations in $\HurTQG$,
 let $k_1=\abs{P_1}$ and $k_2=\abs{P_2}$, and
 let $\uU_1=(U_{1,1},\dots,U_{1,k_1})$ and $\uU_2=(U_{2,1},\dots,U_{2,k_2})$ be adapted
 coverings of $P_1$ and $P_2$ respectively; finally,
 let $\fU(\fc_1;\uU_1)$ and $\fU(\fc_2;\uU_2)$ be the corresponding normal neighbourhoods.

 Suppose that $\fc'=(P',\psi',\phi')$ lies in $\fU(\fc_1;\uU_1)\cap\fU(\fc_2;\uU_2)$.
 Then we can define $\uU'=(U'_1,\dots,U'_{\kappa'})$ as the family
 of all convex open sets of the form $U_{1,i}\cap U_{2,j}$ that contain at least one
 point of $P'$. By construction $\uU'$ is a covering of $P'$;
 we can then find a covering $\uU''=(U''_1,\dots,U''_{\kappa''})$ of $P'$ which is adapted to $P'$
 and is finer than $\uU'$, i.e. each $U''_i$ is contained in some $U'_j$.
 It follows from Definition \ref{defn:normalneigh} that
\[
 \fU(\fc';\uU'')\subseteq \fU(\fc_1;\uU_1)\cap\fU(\fc_2;\uU_2).
\]
Hence normal neighbourhoods are the basis of a topology on $\HurTQG$.

To see that this topology is Hausdorff, let $\fc,\fc'\in\HurTQG$ and use Notation \ref{nota:fc}.
If $P=P'$, then for any adapted covering $\uU$ of $P$ the two normal neighbourhoods
$\fU(\fc;\uU)$ and $\fU(\fc';\uU)$ are disjoint. If $P\neq P'$, without
loss of generality we can assume that there is a point $z\in P\setminus P'$; let $\uU$ and $\uU'$ be adapted coverings
of $P$ and $P'$ respectively, such that the connected component of $\uU$ containing $z$ is disjoint from $\uU'$;
then $\fU(\fc;\uU)$ and $\fU(\fc';\uU')$ are again disjoint.
\end{proof}
\begin{nota}
 \label{nota:HurTQGrgh}
 The space $\HurTQG$ from Proposition \ref{prop:Hurtopology} is called the \emph{Hurwitz-Ran space} associated with the nice couple $\fT$ and the PMQ-group pair $(\Q,G)$.
\end{nota}

\begin{defn}
 \label{defn:epsilon}
We define $\epsilon\colon\Hur(\fT;\Q,G)\to\Ran(\fT)$ as the map given by the assignment
$\epsilon\colon (P,\psi,\phi)\mapsto P$, i.e. a configuration is sent to its support.
\end{defn}
Note that the preimage of $\fU(P,\uU)\subset\Ran(\fT)$ along $\epsilon$ is the disjoint
union of all normal neighbourhoods $\fU(\fc,\uU)$ for $\fc$ varying in the configurations
of $\Hur(\fT;\Q,G)$ supported on $P$; this shows continuity of $\epsilon$.
Note also that $\epsilon\colon\Hur(\fT;\Q,G)\to\Ran(\fT)$ is a homeomorphism if
$(\Q,G)=(\one,\one)$, where we use the following notation.
\begin{nota}
 \label{nota:oneone}
We denote by $(\one,\one)$ the initial and terminal PMQ-group pair, consisting of the trivial PMQ $\set{\one}$
and of the trivial group $\set{\one}$.
\end{nota}

\begin{nota}
 \label{nota:emptysetoneone}
 For all nice couples $\fT$ we denote by $(\emptyset,\one,\one)\in\Hur(\fT;\Q,G)$ the unique configuration $(P,\psi,\phi)$ with
 $P=\emptyset$; note that in this case the maps $\psi$ and $\phi$ are defined on
 the trivial PMQ and on the trivial group respectively, so they have as images $\set{\one}\subset\Q$
 and $\set{\one}\subset G$ respectively.

 Note that $(\emptyset,\one,\one)$ is an isolated point of the space $\Hur(\fT;\Q,G)$; we denote by $\Hur_+(\fT;\Q,G)$ the closed subspace
 $\Hur(\fT;\Q,G)\setminus\set{(\emptyset,\one,\one)}\subset\Hur(\fT;\Q,G)$.
\end{nota}

\begin{defn}
 \label{defn:basedHur}
  Let $P_0\subset\cX$ be a finite non-empty subset. We denote
 by $\Hur(\fT;\Q,G)_{P_0}\subset\Hur_+(\fT;\Q,G)$ the
 preimage of $\Ran(\fT)_{P_0}$ along $\epsilon$.
 For $\fc\in\Hur(\fT;\Q,G)_{P_0}$ and an adapted covering $\uU$ of $\epsilon(\fc)$ we denote
 \[
 \fU(\fc,\uU)_{P_0}=\fU(\fc,\uU)\cap\Hur(\fT;\Q,G)_{P_0}.
 \]
\end{defn}

\subsection{Change of ambient space}
\label{subsec:changeambient}
Let $\bT\subset\C$ be a contractible space containing $*$, and denote by $\mathring{\bT}$ the interior of $\bT$, i.e. the set of all $z\in\bT$ for which there is an open disc $z\in U\subset\bT$. Let $\fT=(\cX,\cY)$ be a nice couple with $\cY\subset\cX\subset\mathring{\bT}$;
then for all finite subsets $P\subset \cX$ the inclusion $\bT\setminus P\hookrightarrow\CmP$ induces an identification $\pi_1(\bT\setminus P,*)\cong\fG(P)$.
Similarly we identify $\fQ_{\fT}(P)$ with the subset of $\pi_1(\bT\setminus P,*)$ consisting of
$\set{\one}$ and all conjugacy classes
corresponding to small simple closed curves in $\bT\setminus P$ spinning clockwise around one
of the points of $P\setminus\cY$.
We thus have an alternative construction of the Hurwitz set $\Hur(\fT;\Q,G)$, in which we use as \emph{ambient space} the subspace $\bT$
instead of the entire $\C$. The topology
on $\Hur(\fT;\Q,G)$ from Proposition \ref{prop:Hurtopology}
can be obtained by only considering coverings $\uU\subset\bT$.
\begin{defn}
 \label{defn:ambient}
 We denote by $\Hur^{\bT}(\fT;\Q,G)$ the Hurwitz-Ran space constructed using $\bT$ as ambient space.
 For $P\subset \cX$ we let $\fG^{\bT}(P):=\pi_1(\bT\setminus P,*)\cong\fG(P)$, and we denote by
 $\fQ^{\bT}_{\fT}(P)\subseteq\fG^{\bT}(P)$ the subset corresponding to $\fQ_{\fT}(P)\subset\fG(P)$.
\end{defn}

Suppose now that we have a nice couple $\fT=(\cX,\cY)$ and two contractible subspaces $\bT,\bT_1\subset\C$
satisfying the following properties:
\begin{itemize}
 \item $*\in\bT_1\subset\bT$;
 \item $\cX\subset\mathring{\bT}$;
 \item $\cX$ splits as a disjoint union $\cX_1\sqcup\cX_2$,
 with $\cX_1\subset\mathring{\bT}_1$ and $\cX_2$ contained in the interior
 of $\bT\setminus\bT_1$.
\end{itemize}
Denote by $\cY_1=\cY\cap\cX_1$ and by $\fT_1$ the nice couple $(\cX_1,\cY_1)$;
then every finite subset $P\subset\cX$
decomposes as a union of $P_1=P\cap \cX_1$ and $P_2=P\cap \cX_2$;
moreover the inclusion $\bT_1\setminus P_1\hookrightarrow\bT\setminus P$ induces an inclusion of PMQ-group pairs
\[
\iota_{\bT_1}^{\bT}(P_1,P)\colon(\fQ_{\fT_1}^{\bT_1}(P_1),\fG^{\bT_1}(P_1))\hookrightarrow(\fQ_{\fT}^{\bT}(P),\fG^{\bT}(P)),
\]
and if $(\phi,\psi)\colon(\fQ_{\fT}^{\bT}(P),\fG^{\bT}(P))\to(\Q,G)$ is a map of PMQ-group pairs, we can
consider the restriction
$(\phi,\psi)\circ\iota^{\bT}_{\bT_1}(P_1,P)\colon (\fQ_{\fT_1}^{\bT_1}(P_1),\fG^{\bT}(P_1))\to(\Q,G)$.
\begin{defn}
\label{defn:fribT}
The above construction gives a map of sets
\[
\fri_{\bT_1}^{\bT}\colon\Hur^{\bT}(\fT;\Q,G)\to\Hur^{\bT_1}(\fT_1;\Q,G),
\]
defined by sending $\fc=(P,\psi,\phi)$ to
$\fc'=(P',\psi',\phi')$, where $P'=P_1=P\cap \bT_1$
and $(\psi',\phi')=(\phi,\psi)\circ\iota^{\bT}_{\bT_1}(P_1,P)$. See Figure \ref{fig:fribT}.
\end{defn}

\begin{figure}[ht]
 \begin{tikzpicture}[scale=4,decoration={markings,mark=at position 0.38 with {\arrow{>}}}]
  \draw[dashed,->] (-.1,0) to (1.1,0);
  \draw[dashed,->] (0,-1.1) to (0,.8);
  \node at (0,-1) {$*$};
  \draw[pattern=horizontal lines] (-.1,-1.05) to (-.1,.6) to node[above]{\tiny$\bT_1$} (.5,.5)  to (.1,-1.05) to (-.1,-1.05); 
  \fill[white!65!black, looseness=2] (.1,.1) to[out=180, in=-90] (-.05,.3) to[out=90,in=120] (.3,.4) to[out=-60,in=0] (.1,.1);
  \fill[white!50!black, looseness=1] (.1,.2) to[out=180, in=-90] (.05,.3) to[out=90,in=120] (.3,.3) to[out=-60,in=0] (.1,.2);
  \node at (.3,.45){\tiny$\cX_1$}; \node at (.25,.3) {\tiny$\cY_1$}; \node at (.1,.3){$\bullet$}; \node at (.2,.18){$\bullet$};
  \fill[white!50!black, looseness=1] (.8,.3) to[out=180, in=-90] (.6,.5) to[out=90,in=120] (1,.6) to[out=-60,in=0] (.8,.3);
  \fill[white!65!black, looseness=1] (.8,.4) to[out=180, in=-90] (.75,.5) to[out=90,in=120] (.9,.5) to[out=-60,in=0] (.8,.4);
  \node at (1,.65){\tiny$\cX_2$}; \node at (.95,.5) {\tiny$\cY_2$}; \node at (.8,.5){$\bullet$};
  \draw[thin, looseness=2, postaction={decorate}] (0,-1) to[out=88,in=-90] (.08,.2) to[out=90,in=-90] node[below]{\tiny$g_1$} (0,.3)  to[out=90,in=90] (.15,.3)  to[out=-90,in=90] (.1,.2) to[out=-90,in=86] (0,-1);
  \draw[thin, looseness=2, postaction={decorate}] (0,-1) to[out=84,in=-90] (.18,.1) to[out=90,in=-90] (.15,.2)  to[out=90,in=90]
  (.25,.2)  to[out=-90,in=90] node[right]{\tiny$a_2$} (.2,.1) to[out=-90,in=82] (0,-1);
  \draw[thin, looseness=2, postaction={decorate}] (0,-1) to[out=60,in=-90] node[left]{\tiny$a_3$} (.78,.4) to[out=90,in=-90] (0.7,.5)  to[out=90,in=90]
  (.85,.5)  to[out=-90,in=90] (.8,.4) to[out=-90,in=58] (0,-1);
\begin{scope}[shift={(1.5,0)}]
  \draw[dashed,->] (-.1,0) to (1.1,0);
  \draw[dashed,->] (0,-1.1) to (0,.8);
  \node at (0,-1) {$*$};
  \draw[pattern=horizontal lines] (-.1,-1.05) to (-.1,.6) to node[above]{\tiny$\bT_1$} (.5,.5)  to (.1,-1.05) to (-.1,-1.05); 
  \fill[white!65!black, looseness=2] (.1,.1) to[out=180, in=-90] (-.05,.3) to[out=90,in=120] (.3,.4) to[out=-60,in=0] (.1,.1);
  \fill[white!50!black, looseness=1] (.1,.2) to[out=180, in=-90] (.05,.3) to[out=90,in=120] (.3,.3) to[out=-60,in=0] (.1,.2);
  \node at (.3,.45){\tiny$\cX_1$}; \node at (.25,.3) {\tiny$\cY_1$}; \node at (.1,.3){$\bullet$}; \node at (.2,.18){$\bullet$};
  \draw[thin, looseness=2, postaction={decorate}] (0,-1) to[out=88,in=-90] (.08,.2) to[out=90,in=-90] node[below]{\tiny$g_1$} (0,.3)  to[out=90,in=90] (.15,.3)  to[out=-90,in=90] (.1,.2) to[out=-90,in=86] (0,-1);
  \draw[thin, looseness=2, postaction={decorate}] (0,-1) to[out=84,in=-90] (.18,.1) to[out=90,in=-90] (.15,.2)  to[out=90,in=90]
  (.25,.2)  to[out=-90,in=90] node[right]{\tiny$a_2$} (.2,.1) to[out=-90,in=82] (0,-1);
\end{scope}
 \end{tikzpicture}
 \caption{On left: a contractible subspace $\bT_1\subset\C$; a nice couple $\fT=(\cX,\cY)$ decomposing as a disjoint union $\fT_1\sqcup\fT_2$, with
 $\fT_1\subset\mathring{\bT}_1$ and $\fT_2$ contained in the interior of $\C\setminus \bT_1$; and a configuration $\fc\in\Hur(\fT;\Q,G)$. On right,
 the image of $\fri^{\C}_{\bT_1}(\fc)$ in $\Hur^{\bT_1}(\fT_1;\Q,G)$.}
\label{fig:fribT}
\end{figure}
To prove that $\fri_{\bT_1}^{\bT}$ is continuous, let $\fc$ and $\fc'$ be as in Definition \ref{defn:fribT}
and choose an adapted covering $\uU'\subset\bT$ of $P'$ with respect to the nice couple $\fT_1$; since
$P_2$ is contained in the interior of $\bT\setminus\bT_1$, we can extend $\uU'$ to an adapted
covering of $P$ with respect to the nice couple $\fT$, by adjoining open sets contained in
$\bT\setminus\bT_1\subset\C$ and covering $P_2$. We then have that $\fri_{\bT_1}^{\bT}$ sends
$\fU(\fc;\uU)\subset\Hur^{\bT}(\fT;\Q,G)$ inside $\fU(\fc';\uU')\subset\Hur^{\bT_1}(\fT_1;\Q,G)$.

Note also that the canonical homeomorphism from Definition \ref{defn:ambient}
can be rewritten as $\fri^{\C}_{\bT}\colon\Hur(\fT;\Q,G)\cong\Hur^{\bT}(\fT;\Q,G)$.

Consider now the following setting. Let $\bT_1,\bT_2$ be contractible
subspaces of $\C$ containing $*$, such that $\bT_1\cap\bT_2$ and $\bT:=\bT_1\cup\bT_2$, are contractible. Let $\fT=(\cX,\cY)$ be a nice couple of subspaces of $\mathring{\bT}$, and assume that we have a splitting
$\fT=\fT_1\sqcup\fT_2=(\cX_1,\cY_1)\sqcup(\cX_2,\cY_2)$, with $\cX_1\subseteq\mathring{\bT}_1\setminus\bT_2$
and $\cX_2\subseteq\mathring{\bT}_2\setminus\bT_1$.
Given $\fc_1=(P_1,\psi_1,\phi_1)\in\Hur^{\bT_1}(\fT_1;\Q,G)$ and
$\fc_2=(P_2,\psi_2,\phi_2)\in\Hur^{\bT_1}(\fT_1;\Q,G)$, we can define a new configuration
$(P,\phi,\psi)\in\Hur^{\bT}(\fT;\Q,G)$ as follows:
\begin{itemize}
 \item $P=P_1\cup P_2$;
 \item by the theorem of Seifert and van Kampen the group $\fG^{\bT}(P)$ decomposes
 naturally as a free product $\fG^{\bT_1}(P_1)\star\fG^{\bT_2}(P_2)$; we define
 $\phi\colon \fG^{\bT}(P)\to G$ as $\phi_1\star\phi_2$;
 \item the inclusions $\fG^{\bT_1}(P_1)\subset\fG^{\bT}(P)$ and $\fG^{\bT_2}(P_2)\subset\fG^{\bT}(P)$
 restrict to inclusions $\fQ^{\bT_1}_{\fT_1}(P_1)\subset\fQ^{\bT}_{\fT}(P)$
 and $\fQ^{\bT_2}_{\fT_2}(P_2)\subset\fQ^{\bT}_{\fT}(P)$; using \cite[Theorem 3.3]{Bianchi:Hur1}
 we define $\psi\colon\fQ^{\bT}_{\fT}(P)\to G$ by imposing that it restricts to $\psi_1$ on
 $\fQ^{\bT_1}_{\fT_1}(P_1)$ and to $\psi_2$ on $\fQ^{\bT_2}_{\fT_2}(P_2)$,
 and that $(\psi,\phi)\colon(\fG^{\bT}(P),\fQ^{\bT}_{\fT}(P))\to(\Q,G)$ is a map of PMQ-group pairs.
\end{itemize}
\begin{defn}
 \label{defn:sqcupHur}
 The above construction gives a map of sets
 \[
  -\sqcup-\colon \Hur^{\bT_1}(\fT_1;\Q,G)\times \Hur^{\bT_2}(\fT_2;\Q,G)\to \Hur^{\bT}(\fT;\Q,G).
 \]
\end{defn}
To prove that $-\sqcup-$ is continuous, note that if $\uU_1\subset\mathring{\bT}_1$ is an
adapted covering of $P_1$ with respect to $\fT_1$, and $\uU_2\subset\mathring{\bT}_2$
is an adapted covering of $P_2$ with respect to $\fT_2$, then $-\sqcup-$
restricts to a bijection between $\fU(\fc_1,\uU_1)\times\fU(\fc_2,\uU_2)$ and $\fU(\fc,\uU)$,
where $\uU=\uU_1\cup\uU_2\subset\mathring{\bT}$ is also an adapted covering of $P$. This argument
shows in fact that $-\sqcup-$ is a homeomorphism, with inverse given by the map
\[
 \pa{\fri^{\bT}_{\bT_1},\fri^{\bT}_{\bT_2}}
 \colon \Hur^{\bT}(\fT;\Q,G)\to  \Hur^{\bT_1}(\fT_1;\Q,G)\times \Hur^{\bT_2}(\fT_2;\Q,G).
\]

\section{Functoriality}
\label{sec:functoriality}
The construction of the space $\HurTQG$ depends on
a nice couple $\fT$ and a PMQ-group pair $(\Q,G)$. In this section we study how maps of PMQ-group pairs
and maps of nice couples induce maps on the corresponding Hurwitz-Ran spaces.

\subsection{Functoriality in the PMQ-group pair}
\label{subsec:functorialityHurQG}
In this subsection
we fix a nice couple $\fT=(\cX,\cY)$ as in Definition \ref{defn:nicecouple} and prove the following theorem.
\begin{thm}
 \label{thm:funPMQgroup}
 The assignment $(\Q,G)\mapsto\Hur(\fT;\Q,G)$ extends to a functor from the category of PMQ-group pairs to the category of topological spaces.
\end{thm}
\begin{proof}
Let $(\Q,G)$ and $(\Q',G')$ be two PMQ-group pairs, and let $(\Psi,\Phi)\colon(\Q,G)\to(\Q',G')$
be a morphism of PMQ-group pairs. In the following we define an induced map
$(\Psi,\Phi)_*\colon\HurTQG\to\HurTQGp$.

Given a configuration $\fc=(P,\psi,\phi)$ in the set $\HurTQG$,
we associate with it the configuration $\fc'=(P,\psi',\phi')$ in the set $\HurTQGp$, where the map of PMQ-group pairs
$(\psi',\phi')\colon(\fQ(P),\fG(P))\to(\Q',G')$ is the composition $(\Psi,\Phi)\circ(\psi,\phi)$:
 \[
  \begin{tikzcd}
   (\psi',\phi')\colon(\fQ(P),\fG(P)) \ar[r,"{(\psi,\phi)}"] & (\Q,G) \ar[r,"{(\Psi,\Phi)}"] &  (\Q',G').
  \end{tikzcd}
 \]
We obtain a map of sets
$(\Psi,\Phi)_*=\Hur(\fT;\Psi,\Phi)\colon\HurTQG\to\HurTQGp$.
To show that $(\Psi,\Phi)_*$ is continuous, let $\fc'=(P',\psi',\phi')\in\HurTQGp$ and
let $\fU(\fc',\uU')\subset\HurTQGp$ be the normal neighbourhood associated with an adapted covering $\uU'$ of $P'$ (see Definition \ref{defn:normalneigh}).
Then the preimage of $\fU(\fc',\uU')$ along $(\Psi,\Phi)_*$ is the disjoint
union $\coprod_\fc\fU(\fc,\uU')$, where $\fc$ ranges over all configurations
in the fibre $(\Psi,\Phi)_*^{-1}(\fc')$; in particular it is an open set.

The assignment $(\Psi,\Phi)\mapsto(\Psi,\Phi)_*$ makes $\Hur(\fT;-)$ into a functor from the category $\PMQGrp$ of PMQ-group pairs to the category $\Top$ of topological spaces.
\end{proof}
Note also that if $(\Psi,\Phi)$ is an injective map of PMQ-group pairs, then it
induces an inclusion of spaces $
(\Psi,\Phi)_*\colon\HurTQG\hookrightarrow\HurTQGp$,
i.e. $(\Psi,\Phi)_*$ is a homeomorphism onto its image.
In particular we can take $\Q'=\hQ$ to be the completion of $\Q$, and consider the
inclusion of PMQ-group pairs $(\Q,G)\subset(\hQ,G)$, yielding an inclusion of $\HurTQG$
into the Hurwitz-Ran space $\Hur(\fT;\hQ,G)$ associated with a PMQ-group pair consisting of a \emph{complete} PMQ and a group.
We further notice that the inclusion $\HurTQG\subset\Hur(\fT;\hQ,G)$ is open: given $\fc\in\HurTQG$
and an adapted covering $\uU$ of $\epsilon(\fc)$, the normal neighbourhood $\fU(\fc;\uU)\subset\HurTQG$ is mapped
\emph{bijectively} onto the corresponding normal neighbourhood $\fU'(\fc;\uU)\subset\Hur(\fT;\hQ,G)$. 

Another consequence of the functoriality in the PMQ-group pair is the following. For all PMQ-group
pairs $(\Q,G)$ there is a unique inclusion of PMQ-group pairs $(\one,\one)\hookrightarrow(\Q,G)$
(see Notation \ref{nota:oneone}).
This induces an inclusion $\Hur(\fT;\one,\one)\to\Hur(\fT;\Q,G)$,
and using the homeomorphism $\epsilon\colon\Hur(\fT;\one,\one)\cong\Ran(\fT)$ (see Definition \ref{defn:epsilon}),
we obtain a natural inclusion $\Ran(\fT)\subset\Hur(\fT;\Q,G)$,
for all nice couples $\fT$ and all PMQ-group pairs $(\Q,G)$.
Viceversa, we can consider
the map $\epsilon\colon\Hur(\fT;\Q,G)\to\Ran(\fT)$ as the map
$\Hur(\fT;\Q,G)\to\Hur(\fT;\one,\one)$ induced by the unique map of PMQ-group pairs
$(\Q,G)\to(\one,\one)$.

\subsection{Two categories of nice couples}
We now fix a PMQ-group pair $(\Q,G)$ throughout the rest of the section. To discuss functoriality of Hurwitz-Ran spaces
in the nice couple $\fT$, we first need a good notion of map between nice couples.
\begin{defn}
 \label{defn:mapnicecouples}
 Recall Definition \ref{defn:nicecouple}, and let $\fT=(\cX,\cY)$ and $\fT'=(\cX',\cY')$ be two nice couples.
 A \emph{morphism of nice couples} $\xi\colon\fT\to\fT'$ is a continuous, pointed map $\xi\colon(\C,*)\to(\C,*)$
 such that the following properties hold:
 \begin{enumerate}
  \item $\xi$ is semi-algebraic and proper from $\C$ to $\C$;
  \item $\xi$ is \emph{orientation-preserving} in the sense that
  the induced map $\xi^*\colon H^2_{c}(\C)\to H^2_{c}(\C)$ in cohomology with compact support is the identity;
  \item $\xi$ restricts to maps $\cX\to\cX'$ and $\cY\to\cY'$;
  \item for all $z\in\C$ the fibre $\xi^{-1}(z)\subset\C$ is non-empty, compact and \emph{contractible};
  \item for all $z\in\cX'\setminus\cY'$ the fibre $\xi^{-1}(z)$ contains at most one point of $\cX\setminus\cY$.
 \end{enumerate}
 The composition of two morphisms $\xi\colon\fT\to\fT'$ and $\xi'\colon\fT'\to\fT''$ is defined as the composition of maps
 $\xi'\circ\xi\colon(\C,*)\to(\C,*)$. We denote by $\NC$ the category of nice couples and morphisms of nice couples.
\end{defn}
Property (4) ensures that
a morphism $\xi$ of nice couples is in particular a \emph{local homotopy equivalence} in the following sense:
if $\bJ\subset\C$ is a semi-algebraic set,
then the restriction $\xi\colon\xi^{-1}(\bJ)\to\bJ$ is a homotopy equivalence; more generally,
if $\bJ\subset\bJ'\subset\C$ are two semi-algebraic sets, then $\xi\colon(\xi^{-1}(\bJ),\xi^{-1}(\bJ'))\to(\bJ,\bJ')$ is a homotopy equivalence of couples. This follows from the main theorem of \cite{Smale57}.

The previous remark holds in particular
when $\bJ=(\xi')^{-1}(z)$ is a fibre of another morphism of nice couples $\xi'$, over some point $z\in\C$:
thus the composition
$\xi\circ\xi'$ also satisfies property (4) of Definition \ref{defn:mapnicecouples}; properties
(1),(2),(3) and (5) are also automatically satisfied by the composition $\xi\circ\xi'$. This proves in particular that Definition \ref{defn:mapnicecouples} is a good definition.
We will sometimes need to relax condition (5) in Definition \ref{defn:mapnicecouples}, hence
we give the following definition.

\begin{defn}
 \label{defn:laxmapnicecouples}
 Let $\fT$ and $\fT'$ be nice couples as in Definition \ref{defn:mapnicecouples}. A \emph{lax morphism}
 of nice couples is a map $\xi\colon(\C,*)\to(\C,*)$ satisfying all conditions in Definition
 \ref{defn:mapnicecouples} except, possibly, condition (5).
 We denote by $\LNC$ the category of nice couples and lax morphisms of nice couples.
\end{defn}
Note that $\NC$ is a subcategory of $\LNC$ containing all objects, but not all morphisms. Whenever
we refer to a morphism of nice couples without specifying the word ``lax'', we will assume that
condition (5) in Definition \ref{defn:mapnicecouples} holds.
 
We conclude the subsection with the following remark: if $D'\subset\C$ is a subspace homeomorphic to a disc and $\xi\colon(\C,*)\to(\C,*)$ satisfies properties (1),(2) and (4) from Definition \ref{defn:mapnicecouples}, then also $\xi^{-1}(D')$ is homeomorphic to a disc.

\subsection{Functoriality in \texorpdfstring{$\NC$}{NC}}
\label{subsec:functorialityHurNC}
In this subsection we prove the following theorem
\begin{thm}
 \label{thm:funnicecouple}
 The assignment $\fT\mapsto\Hur(\fT;\Q,G)$ extends to a functor from the category $\NC$ to the category of topological spaces.
\end{thm}
Fix two nice couples $\fT$ and $\fT'$ and let $\xi\colon\fT\to\fT'$ be a morphism of nice couples.
In the following we construct an induced map $\xi_*\colon\HurTQG\to\HurTpQG$.

Given a configuration $\fc=(P,\psi,\phi)$ in the set $\HurTQG$,
we associate with it a configuration $\fc'=(P',\psi',\phi')$ in the set $\HurTpQG$ as follows.
First, we define $P':=\xi(P)\in\Ran(\cX')$.
To define $\phi'$, note that $\xi$ restricts to a homotopy equivalence $\C\setminus\xi^{-1}(P')\to\CmP$;
 in particular we obtain an isomorphism of groups $\fG(P')\cong \pi_1\pa{\C\setminus\xi^{-1}(P'),*}$.
 Moreover the inclusion $\C\setminus\xi^{-1}(P')\hookrightarrow\CmP$ induces a map of groups
 $\pi_1\pa{\\C\setminus\xi^{-1}(P'),*}\to\fG(P)$.
 We denote by $\xi^*\colon\fG(P')\to\fG(P)$ the composition $\fG(P')\cong\pi_1\pa{\C\setminus\xi^{-1}(P'),*}\to \fG(P)$, and
 define $\phi'\colon\fG(P')\to G$ as the composition $\phi\circ\xi^*$.

\begin{lem}
 \label{lem:xirestrictsfQ}
 The map of groups $\xi^*\colon\fG(P')\to\fG(P)$
 restricts to a map of PMQs $\fQ_{\fT'}(P')\to\fQ_{\fT}(P)$.  
\end{lem}
The proof of Lemma \ref{lem:xirestrictsfQ} is in Subsection \ref{subsec:xirestrictsfQ} of the appendix.
We can now define $\psi':=\psi\circ\xi^*\colon\fQ_{\fT'}(P')\to\Q$.
Since $\xi^*\colon(\fQ_{\fT'}(P'),\fG(P'))\to(\fQ_{\fT}(P),\fG(P))$ is a map of PMQ-group pairs, also $(\psi',\phi')=(\psi,\phi)\circ\xi^*$ is a map of PMQ-group pairs; hence
$\fc'=(P',\psi',\phi')$ is a well-defined configuration in $\HurTpQG$. This construction gives a map of sets $\xi_*\colon\HurTQG\to\HurTpQG$; 
see Figure \ref{fig:functoriality}.

\begin{figure}[ht]
 \begin{tikzpicture}[scale=4,decoration={markings,mark=at position 0.38 with {\arrow{>}}}]
  \draw[dashed,->] (-.1,0) to (1.1,0);
  \draw[dashed,->] (0,-1.1) to (0,1.1);
  \node at (0,-1) {$*$};
  \fill[black, opacity=.2] (0,0) rectangle (1,1);
  \fill[black, opacity=.3] (.7,0) rectangle (1,1);
  \fill[pattern=horizontal lines, opacity=.4] (.3,0) rectangle (1,1);
  \node at (.1,.3){$\bullet$};\node at (.15,.28) {\tiny$z_1$}; 
  \node at (.4,.5){$\bullet$};\node at (.4,.56) {\tiny$z_2$};
  \node at (.9,.5){$\bullet$};\node at (.85,.52) {\tiny$z_3$};
  \node at (.3,1.05){$\cX$};  \node at (.9,.9){$\cY$};
  \draw[thick] (.3,.5) to (1, .5);
  \draw[thin, looseness=1.2, postaction={decorate}] (0,-1) to[out=85,in=-90]  (.05,.1) to[out=90,in=-90] (-.1,.3)  to[out=90,in=90] (.25,.3)  to[out=-90,in=90] node[below]{\tiny$a_1$} (.1,.1) to[out=-90,in=80] (0,-1);
  \draw[thin, looseness=1.2, postaction={decorate}] (0,-1) to[out=70,in=-90] (.45,.3)
  to[out=90,in=-90] (.3,.5) to[out=90,in=90] (.6,.5)  to[out=-90,in=90] node[below]{\tiny$a_2$} (.49,.3)   to[out=-90,in=67]  (0,-1);
  \draw[thin, looseness=1.2, postaction={decorate}] (0,-1) to[out=66,in=-90]  (.85,.3)
  to[out=90,in=-90] node[left]{\tiny$g_3$} (.8,.5) to[out=90,in=90] (.96,.5) to[out=-90,in=90] (.9,.3) to[out=-90,in=63]  (0,-1);
  \draw[dashed, looseness=1.2, postaction={decorate}] (0,-1) to[out=72,in=-90] (.4,.3)
  to[out=90,in=-90] (.26,.5) to[out=90,in=90] node[below]{\tiny$\fe(a_2)g_3$} (1.05,.5) to[out=-90,in=90] (.93,.3) to[out=-90,in=60]  (0,-1);
\begin{scope}[shift={(1.5,0)}]
  \draw[dashed,->] (-.1,0) to (1.1,0);
  \draw[dashed,->] (0,-1.1) to (0,1.1);
  \node at (0,-1) {$*$};
  \fill[black, opacity=.2] (0,0) rectangle (.7,1);
  \fill[black, opacity=.6] (.7,0) rectangle (.71,1);
  \node at (.3,1.05){$\cX'$};  \node at (.71,.94){$\cY'$};
  \node[anchor=west] at (.2,.3){$\bullet$};\node at (.1,.3) {\tiny$\xi_*(z_1)$}; 
  \node at (.71,.5){$\bullet$};\node at (.7,.55) {\tiny$\xi_*(z_2)=\xi_*(z_3)$};
  \draw[thin, looseness=1.2, postaction={decorate}] (0,-1) to[out=80,in=-90]  (.15,.1) to[out=90,in=-90] (-.1,.3) 
  to[out=90,in=90] (.4,.3)  to[out=-90,in=90] node[below]{\tiny$a_1$}(.2,.1)  to[out=-90,in=75] (0,-1);
  \draw[dashed, looseness=1.2, postaction={decorate}] (0,-1) to[out=72,in=-90] (.65,.3)
  to[out=90,in=-90] (.4,.5) to[out=90,in=90] node[above]{\tiny$\fe(a_2)g_3$} (1,.5) to[out=-90,in=90] (.7,.3) to[out=-90,in=60]  (0,-1);
\end{scope}
 \end{tikzpicture}
 \caption{On left, a configuration $\fc\in\Hur(\cX,\cY,\Q,G)$; on right, its image $\fc'\in\Hur(\cX',\cY';\Q,G)$
 along the map $\xi_*$ induced by a morphism of nice couples $\xi$. The morphism $\xi$ has the effect of collapsing
 horizontally a rectangular region of $\cX$ onto the vertical segment $\cY'$, and of expanding horizontally the complement
 of this rectangular region. The thick horizontal segment is the preimage along $\xi$ of $\xi(z_2)=\xi(z_3)$. The dashed loop on left
 is the image of the dashed loop on right along $\xi^*$.}
\label{fig:functoriality}
\end{figure}

\begin{proof}[Proof of Theorem \ref{thm:funnicecouple}]
It is left to check that the map $\xi_*$ constructed above is continuous.
Let $\fc\in\HurTQG$, denote $\fc'=\xi_*(\fc)$, and use Notation \ref{nota:fc};
let $\fU(\fc',\uU')\subset\HurTpQG$ be a normal neighbourhood associated with an adapted covering $\uU'$
of $P'$.
We have $\xi(P)\subset\uU'$, hence we can find an adapted covering $\uU$
of $P$, such that $\xi(\uU)\subset\uU'$.
Then the entire normal neighbourhood $\fU(\fc;\uU)\subseteq\HurTQG$
is mapped along $\xi_*$ inside $\fU(\fc',\uU')$.
Thus sending $\xi\mapsto\xi_*$ makes $\Hur(-;\Q,G)$ into a functor from $\NC$ to $\Top$.
\end{proof}

Note that if $\fT=(\cX,\cY)$ and $\fT'=(\cX',\cY')$ are two nice couples, and if $\cX\subseteq\cX'$, $\cY\subseteq\cY'$
and $\cY=\cX\cap\cY'$, then $\Id_{\C}$
is a morphism of nice couples $\fT\to\fT'$. The induced map $(\Id_{\C})_*\colon\HurTQG\to\HurTpQG$ is an inclusion of spaces: more precisely,
$\HurTQG$ contains all configurations of $\HurTpQG$ supported in $\cX$.

In particular, given a nice couple $\fT=(\cX,\cY)$, by Definition \ref{defn:nicecouple} the space $\cY$ is closed in $\cX\subseteq\bH$: this means that the closure $\bar{\cY}$ of $\cY$
in $\bH$ is contained in $\bar{\cX}$, and $\cY=\cX\cap\bar{\cY}$. Then $\Id_{\C}$ is
a morphism of nice couples $\fT=(\cX,\cY)\to\bar{\fT}:=(\bar{\cX},\bar{\cY})$.
Thus every
Hurwitz-Ran space $\HurTQG$ can be regarded as a subspace of a Hurwitz-Ran space $\Hur(\bar{\fT};\Q,G)$
associated with a nice couple of \emph{closed} subspaces of $\bH$.

\subsection{A weak form of enriched functoriality}
One can consider $\NC$ as a category enriched in topological spaces:
for all nice couples $\fT$ and $\fT'$ one can consider the compact-open topology on the
set of morphisms $\xi\colon\fT\to\fT'$, considered as a subset of all continuous maps $\xi\colon\C\to\C$.
The functor $\Hur(-;\Q,G)$ is then likely to be a $\Top$-enriched
functor from $\NC$ to $\Top$.  We will not attempt to prove this property in general,
and we will restrict our attention to the following proposition.
\begin{prop}
 \label{prop:functorialityhomotopy}
 Let $\fT=(\cX,\cY)$ and $\fT'=(\cX',\cY')$ be nice couples and let $(\Q,G)$ be a PMQ-group pair. Let $\cS$ be a
 topological space, and let $\cH\colon\C\times \cS\to\C$
 be a continuous map, such that for all $s\in \cS$ the map $\cH(-,s)\colon\C\to\C$ is a morphism of nice couples
 $\fT\to\fT'$ (see Definition \ref{defn:mapnicecouples}). Let
 \[
  \cH_*\colon\HurTQG\times \cS\to\HurTpQG
 \]
 be the map of sets defined by $\cH_*(\fc,s)=(\cH(-,s))_*(\fc)$. Then $\cH_*$ is continuous.
\end{prop}
\begin{proof}
 Fix $(\fc,s)\in\HurTQG\times \cS$, let $\fc'=\cH_*(\fc,s)$ and use Notation \ref{nota:fc}. Let $\uU'$ be an adapted
 covering of $P'$ and let $\fU(\fc',\uU')\subset\HurTpQG$ be the corresponding normal neighbourhood. By continuity of
 $\cH$ we can find a neighbourhood $V\subset \cS$ of $s$
 and an adapted covering $\uU$ of $P$ such that $\cH$ sends $\uU\times V$
 inside $\uU'$: here we regard $\uU$ and $\uU'$ as subsets of $\C$. Then $\cH_*$ sends the product
 neighbourhood $\fU(\fc,\uU)\times V\subset\HurTQG\times \cS$ inside $\fU(\fc',\uU')$.
\end{proof}

\subsection{Functoriality in \texorpdfstring{$\LNC$}{LNC}}
\label{subsec:LNCfunctoriality}
In this subsection we assume that $\Q$ is a complete PMQ (all products are defined), and write $\hQ=\Q$ to stress this choice; hence we work with the PMQ-group pair $(\hQ,G)$. We prove the following theorem.
\begin{thm}
 \label{thm:funlaxnicecouple}
 The assignment $\fT\mapsto\Hur(\fT;\hQ,G)$ extends to a functor from the category $\LNC$ to the category of topological spaces.
\end{thm}
Let $\fT=(\cX,\cY)$ be a nice couple, and let $\fc=(P,\psi,\phi)\in\Hur(\fT;\hQ,G)$.
By Definition \ref{defn:Hurset}, $\psi$ is a map of PMQs defined on $\fQ(P)$;
using the completeness of $\hQ$, Proposition \ref{prop:fQgeneratesfQext} implies the equality
$\fQext(P)=\fQext(P)_{\psi}$ (see also Definitions \ref{defn:fQextP} and \ref{defn:fQextPpsi}).
Proposition \ref{prop:fQextPpsi} yields a map of PMQ-group pairs
\[
(\psiext,\phi)\colon(\fQext(P),\fG(P))\to(\hQ,G),
\]
extending $(\psi,\phi)\colon(\fQ(P),\fG(P))\to(\hQ,G)$.
\begin{defn}
 \label{defn:Hurext}
 We define a set $\Hurext(\fT;\hQ,G)$: it contains triples $\fc=(P,\psi,\phi)$, where $P\in\Ran(\cX)$ is
 a finite subset of $\cX$, and $(\psi,\phi)\colon(\fQext(P),\fG(P))\to(\hQ,G)$ is a map of PMQ-group pairs.
 \end{defn}
 The previous discussion implies that the sets
 $\Hurext(\fT;\hQ,G)$ and $\Hur(\fT;\hQ,G)$ are in natural bijection. We can use this bijection to transfer the topology
 of $\Hur(\fT;\hQ,G)$ to $\Hurext(\fT;\hQ,G)$: in particular, for a configuration $\fc=(P,\psi,\phi)\in\Hurext(\fT;\hQ,G)$
 and for an adapted covering $\uU$ of $P$,
 the normal neighbourhood $\fU(\fc,\uU)\subset\Hurext(\fT;\hQ,G)$ contains all configurations $(P',\psi',\phi')$
 such that
 \begin{itemize}
 \item $P'\subset\uU$; as a consequence there is a natural inclusion of PMQ-group pairs
 $(\fQ(\uU),\fG(\uU))\subseteq(\fQ(P',\uU),\fG(P'))$ (see Definition \ref{defn:fGUfQPU});
 \item the following composition of maps of PMQ-group pairs is equal to the restriction of $(\psi,\phi)$
 on the PMQ-group pair $(\fQ(P),\fG(P))$:
 \[
 \begin{tikzcd}
  (\fQ(P),\fG(P)) & (\fQ(\uU),\fG(\uU)) \ar[l,"\cong"'] \ar[r,"\subseteq"] & (\fQ(P',\uU),\fG(P')) \ar[ld,"\subseteq"'] \\
  & (\fQext(P'),\fG(P')) \ar[r,"{(\psipext,\phi')}"] & (\hQ,G).
 \end{tikzcd}
 \]
\end{itemize}
Given a \emph{lax} morphism of nice couples $\xi\colon\fT\to\fT'$, we can now follow the same
procedure used in Subsection \ref{subsec:functorialityHurNC} and define a continuous map
$\xi_*\colon\Hurext(\fT;\hQ,G)\to\Hurext(\fT',\hQ,G)$. The only difference is that
Lemma \ref{lem:xirestrictsfQ} is replaced by the following lemma, whose proof is in Subsection
\ref{subsec:xirestrictsfQLNC} of the appendix.
\begin{lem}
 \label{lem:xirestrictsfQLNC}
Let $\xi\colon\fT\to\fT'$ be a lax morphism of nice couples,
let $P\subset\cX$ and let $P'=\xi(P)\subset\cX'$. Then the map of groups $\xi^*\colon\fG(P')\to\fG(P)$
restricts to a map $\fQext_{\fT'}(P')\to\fQext_{\fT}(P)$ of PMQs.
\end{lem}

Continuity of $\xi_*\colon\Hurext(\fT;\hQ,G)\to\Hurext(\fT';\hQ,G)$ is proved in the same
way as in the case of a (non-lax) morphism of nice couples; similarly one can generalise
Proposition \ref{prop:functorialityhomotopy} to the following.

\begin{prop}
 \label{prop:functorialityhomotopyLNC}
 Let $\fT=(\cX,\cY)$ and $\fT'=(\cX',\cY')$ be nice couples and let $(\hQ,G)$ be a PMQ-group pair with $\hQ$ complete. Let $\cS$ be a topological space,
 and let $\cH\colon\C\times \cS\to\C$
 be a continuous map, such that for all $s\in \cS$ the map $\cH(-,s)\colon\C\to\C$ is a lax morphism of nice couples
 $\fT\to\fT'$ (see Definition \ref{defn:mapnicecouples}). Let
 \[
  \cH_*\colon\Hur(\fT;\hQ,G)\times \cS\to\Hur(\fT';\hQ,G)
 \]
 be the map of sets defined by $\cH_*(\fc,s)=(\cH(-,s))_*(\fc)$. Then $\cH_*$ is continuous.
\end{prop}

\section{Applications of functoriality}
\label{sec:applfunctoriality}
In this section we apply the results from Section \ref{sec:functoriality} to obtain
basic information about Hurwitz-Ran spaces; moreover we introduce the operation of \emph{external product}.

\subsection{Product structure for normal neighbourhoods}
The first application combines the discussion of Subsection \ref{subsec:changeambient} with the functoriality
with respect to inclusions of nice couples.
\begin{prop}
 \label{prop:productneighbourhood}
 Let $\fc\in\Hur(\fT;\Q,G)$, use Notations \ref{nota:fc} and \ref{nota:uUcovering},
 and let $\uU$ be an adapted covering of $P$. Then there exist configurations
 $\fc'_i\in\Hur(\fT;\Q,G)$ supported on $\set{z_i}$ and a homeomorphism
 \[
  \fU(\fc,\uU)\cong\prod_{i=1}^k\fU(\fc'_i,U_i).
 \]
\end{prop}
\begin{proof}
We can fix arcs $\arc_1,\dots,\arc_k$ as in Definition \ref{defn:admgenset}: the arc $\arc_i$ joins $*$ with
a point on $\del U_i$. For all $1\le i\le k$ we define $\bT_i=\arc_i\cup \bar{U}_i$, $\cX_i=\cX\cap U_i$ and
$\cY_i=\cY\cap U_i$. Let moreover $\bT=\bigcup_{i=1}^k\bT_i$, $\ucX=\bigcup_{i=1}^k\cX_i$ and $\ucY=\bigcup_{i=1}^k\cY_i$. Finally, let $\fT_i$ be the nice couple $(\cX_i,\cY_i)$, and let
$\ufT$ be the nice couple $(\ucX,\ucY)$.

We have an open inclusion $(\Id_\C)_*\colon \Hur(\ufT;\Q,G)\subset\Hur(\fT;\Q,G)$ restricting to
a homeomorphism of normal neighbourhoods $\fU_{\ufT}(\fc,\uU)\cong\fU_{\fT}(\fc,\uU)$.
In fact, the normal neighbourhood $\fU_{\ufT}(\fc,\uU)$ coincides with the entire space
$\Hur(\ufT;\Q,G)$, because $\ucX$ is contained in $\uU$.
Since $\ufT$ is a nice couple of spaces that are contained in the interior
of $\bT$, we can identify $\Hur(\ufT;\Q,G)$ with $\Hur^{\bT}(\ufT;\Q,G)$ along $\fri^{\C}_{\bT}$.
Via the decomposition given in Definition \ref{defn:sqcupHur} we write
\[
\Hur^{\bT}(\ufT;\Q,G)\cong\prod_{i=1}^k \Hur^{\bT_i}(\fT_i;\Q,G),
\]
where the homeomorphism is given by the product of the maps $\fri^{\bT}_{\bT_i}$.
We thus obtain a homeomorphism $\fU_{\fT}(\fc;\uU)\cong \prod_{i=1}^k \Hur^{\bT_i}(\fT_i;\Q,G)$.

Let $\fc\in\fU_{\fT}(\fc;\uU)$ correspond to the sequence $(\fc_1,\dots,\fc_k)$ along the above identification, where $\fc_i\in\Hur^{\bT_i}(\fT_i;\Q,G)$.
Note that $\fc_i$ is supported on the singleton $\set{z_i}$. We consider the composition of the homeomorphism
$(\fri^{\C}_{\bT_i})^{-1}\colon\Hur^{\bT_i}(\fT_i;\Q,G)\to\Hur(\fT_i;\Q,G)$
with the open inclusion
$(\Id_\C)_*\colon \Hur(\fT_i;\Q,G)\subset\Hur(\fT;\Q,G)$,
giving an embedding of $\Hur^{\bT_i}(\fT_i;\Q,G)$
in $\Hur(\fT;\Q,G)$. If we denote by $\fc'_i=(\Id_{\C})_*\circ(\fri^{\C}_{\bT_i})^{-1}(\fc_i)$
the image of $\fc_i$ along this embedding, we get a homeomorphism
\[
 (\Id_{\C})_*\circ(\fri^{\C}_{\bT_i})^{-1}\colon \Hur^{\bT_i}(\fT_i;\Q,G)\cong \fU_{\fT}(\fc'_i;U_i)\subset\Hur(\fT;\Q,G).
\]
\end{proof}
The homeomorphism of Proposition \ref{prop:productneighbourhood}
depends in general on the arcs $\arc_i$.

\subsection{Three useful homeomorphisms}
In this subsection we prove three homeomorphisms between Hurwitz-Ran spaces.
First, let $\fT=(\cX,\cY)$ be a nice couple
and let $(\Q,G,\fe,\fr)$ be a PMQ-group pair.
The map of PMQs $\fe\colon\Q\to G$ has an adjoint map of groups $\cG(\fe)\colon\cG(\Q)\to G$,
so that the couple of maps
$\Id_{\Q}\colon\Q\to\Q$ and $\cG(\fe)\colon\cG(\Q)\to G$ yields a map of PMQ-group pairs $(\Id_{\Q},\cG(\fe))\colon(\Q,\cG(\Q))\to(\Q,G)$. By functoriality we obtain a continuous map
\[
 (\Id_{\Q},\cG(\fe))_*\colon \Hur(\fT;\Q,\cG(\Q))\to \Hur(\fT;\Q,G).
\]
\begin{lem}
 \label{lem:usefulhomeocYempty}
 Let $\fT$ be a nice couple of the form $(\cX,\emptyset)$; then the above map $(\Id_{\Q},\cG(\fe))_*$ is a homeomorphism.
\end{lem}
The proof of Lemma \ref{lem:usefulhomeocYempty} is in Subsection \ref{subsec:usefulhomeocYempty}
of the appendix.
Roughly speaking, Lemma \ref{lem:usefulhomeocYempty} says that if we consider a nice couple of the
form $(\cX,\emptyset)$, then the space $\Hur(\fT;\Q,G)$ \emph{only depends on $\Q$}:
the monodromy $\psi$ uniquely determines the monodromy $\phi$. This motivates the following
notation, which can be thought of as an \emph{absolute} definition of Hurwitz-Ran spaces, whereas the general
one, given in Definition \ref{defn:Hurset} and depending on a nice \emph{couple} and a PMQ-group \emph{pair},
can be considered as the general, \emph{relative} definition.
\begin{nota}
 \label{nota:absHur}
 For a subspace $\cX\subset\bH$ and a PMQ $\Q$ we denote by
 $\Hur(\cX;\Q)$ the space $\Hur(\cX,\emptyset;\Q,\cG(\Q))$. A configuration $\fc\in\Hur(\cX;\Q)$ is usually
 presented as $(P,\psi)$ instead of $(P,\psi,\phi)$ as in Notation \ref{nota:fc}, since $\phi$ is uniquely
 determined by $\psi$.
\end{nota}
For the second homeomorphism, let $\fT=(\cX,\cY)$ be any nice couple, and consider the PMQ-group pair $(G,G)$, where $G$ is considered as a PMQ with full product,
the first $G$ maps to the second $G$ by $\Id_G$ and the second $G$ acts on the first $G$ by right conjugation.
Then $\Id_{\C}$ is a morphism of nice couples $(\cX,\cY)\to(\cX,\cX)$.
By functoriality we obtain a continuous map
\[
 (\Id_{\C})_*\colon \Hur(\fT;G,G)\to \Hur(\cX,\cX;G,G).
\]

\begin{lem}
 \label{lem:usefulhomeoGG}
 The above map $(\Id_{\C})_*$ is a homeomorphism.
\end{lem}
The proof of Lemma \ref{lem:usefulhomeoGG} is in Subsection \ref{subsec:usefulhomeoGG} of the appendix.

For the third homeomorphism, let $(\Q,G)$ be any PMQ-group pair and
let $\fT$ be a nice couple of the form $(\cX,\cX)$; then for all
finite subset $P\subset\cX$ we have $\fQ_{\fT}(P)=\set{\one}$; in particular for all
$\fc=(P,\psi,\phi)\in\Hur(\fT;\Q,G)$ we have that $\psi\colon\fQ_{\fT}(P)\to\Q$ is the trivial
map of PMQs: roughly speaking, this means that
we can replace $\Q$ by another PMQ fitting with $G$ into
a PMQ-group pair, without changing the topology of $\Hur(\fT;\Q,G)$. For instance we can consider
the map of PMQ-group pairs $(\fe,\Id_G)\colon(\Q,G)\to(G,G)$, thus replacing $\Q$ by $G$.
We obtain the following lemma.
\begin{lem}
\label{lem:usefulhomeoXXGG}
For all $\cX\subset\bH$ and all PMQ-group pair $(\Q,G)$ the following map is a homeomorphism
\[
 (\fe,\Id_G)_*\colon\Hur(\cX,\cX;\Q,G)\to\Hur(\cX,\cX;G,G).
\]
\end{lem}
Using Lemmas \ref{lem:usefulhomeocYempty}, \ref{lem:usefulhomeoGG} and \ref{lem:usefulhomeoXXGG},
we can simplify our notation for Hurwitz-Ran spaces $\Hur(\cX,\cY;\Q,G)$
whenever one of the following conditions is satisfied:
\begin{itemize}
 \item $\cY=\emptyset$, then we identify $\Hur(\cX,\emptyset;\Q,G)\cong\Hur(\cX;\Q)$;
 \item $\cY=\cX$, then we identify $\Hur(\cX,\cX;\Q,G)\cong\Hur(\cX,\cX;G,G)\cong\Hur(\cX;G)$;
 \item $\Q=G$, then we identify $\Hur(\cX,\cY;G,G)\cong\Hur(\cX,\cX;G,G)\cong\Hur(\cX;G)$.
\end{itemize}

\subsection{Functoriality and change of ambient space}
Recall Definition \ref{defn:ambient}, let $\fT=(\cX,\cY)$ and $\fT'=(\cX',\cY')$
be two nice couples, and let $\bT$ and $\bT'$ be two contractible semi-algebraic subspaces of $\C$
containing $*$, such that $\cX\subset\mathring{\bT}$ and $\cX'\subset\mathring{\bT}'$.

Let $\xi\colon(\bT,*)\to(\bT,*)$ be a semi-algebraic \emph{homeomorphism}
restricting to an orientation-preserving homeomorphism $\xi\colon\mathring{\bT}\to\mathring{\bT}'$,
and to maps $\xi\colon\cX\to\cX'$ and $\xi\colon\cY\to\cY'$ (we restrict to the case
of a homeomorphism for simplicity, but any map $\xi$ satisfying a suitable analogue of conditions (1)-(5)
in Definition \ref{defn:mapnicecouples} may be used).

We define an induced map $\xi_*\colon\Hur^{\bT}(\fT;\Q,G)\to\Hur^{\bT'}(\fT';\Q,G)$.
Given a configuration $\fc=(P,\psi,\phi)\in\Hur^{\bT}(\fT;\Q,G)$,
we define $\fc'=\xi_*(\fc)=(P',\psi',\phi')\in\Hur^{\bT'}(\fT';\Q,G)$ as follows:
\begin{itemize}
 \item $P'=\xi(P)$; note that $\xi$ restricts to a homeomorphism $\bT\setminus P\to\bT'\setminus P'$;
 \item $(\psi',\phi')\colon (\fQ^{\bT'}_{\fT'}(P')\to(\Q,G)$ is the following composition of maps of PMQ-group pairs
 \[
  \begin{tikzcd}
   (\fQ^{\bT'}_{\fT'}(P'),\fG^{\fT'}(P')) \ar[r,"{(\xi^{-1})_*}"] &
   (\fQ^{\bT}_{\fT}(P),\fG^{\fT}(P)) \ar[r,"{(\psi,\phi)}"] & (\Q,G).
  \end{tikzcd}
 \]
\end{itemize}
The same arguments used in Subsection \ref{subsec:functorialityHurNC} show that $\xi_*$
is continuous. In the next articles of this series
we will use this fact in the particular case in which $\xi$ restricts
also to homeomorphisms $\cX\to\cX'$ and $\cY\to\cY'$: then we can use the inverse homeomorphism
$\xi^{-1}\colon \bT'\to\bT$ to define a map 
$(\xi^{-1})_*\colon\Hur^{\bT'}(\fT';\Q,G)\to\Hur^{\bT}(\fT;\Q,G)$. The maps
$\xi_*$ and $(\xi^{-1})_*$ are inverse homeomorphism, and we obtain in particular
the following proposition.
\begin{prop}
 \label{prop:functorialityambient}
Let $\fT=(\cX,\cY)$ and $\fT'=(\cX',\cY')$ be nice couples, and let
$\bT,\bT'\subset\C$ be contractible semi-algebraic
subspaces containing $*$, with $\cX\subset\mathring{\bT}$ and
$\cX'\subset\mathring{\bT}'$. Let $\xi\colon\C\to\C$ be a map inducing a morphism of nice couples $\fT\to\fT'$ and restricting to homeomorphisms $\bT\cong\bT'$, $\cX\cong\cX'$ and $\cY\cong\cY'$.
Then the map $\xi_*\colon\Hur(\fT;\Q,G)\to\Hur(\fT',\Q,G)$ is a homeomorphism.
\end{prop}
 
\subsection{External products of Hurwitz-Ran spaces}
In this subsection we fix a nice couple $\fT=(\cX,\cY)$ and
two PMQ-group pairs $(\Q,G)$ and $(\Q',G')$.

\begin{defn}
 \label{defn:externalproduct}
Recall from \cite[Definition 2.16]{Bianchi:Hur1} the explicit description of the (categorical) product of two PMQ-group pairs. We define an \emph{external product}
\[
 -\times -\colon \HurTQG\times\HurTQGp\to\Hur(\fT;(\Q,G)\times(\Q',G')).
\]
Let $(\fc,\fc')\in\HurTQG\times\HurTQGp$, and use Notation \ref{nota:fc}. We define $\fc\times\fc'$
as the configuration $(P'',\psi'',\phi'')\in \Hur(\fT;(\Q,G)\times(\Q',G'))$, where:
\begin{itemize}
 \item $P''=P\cup P'\subset\cX$;
 \item $(\psi'',\phi'')\colon(\fQ(P''),\fG(P''))\to(\Q,G)\times(\Q',G)$ is the map of PMQ-group pairs given by
 $\pa{(\psi,\phi)\circ\fri^{P''}_P,(\psi',\phi')\circ\fri^{P''}_{P'}}$ (see Notation \ref{nota:fri}).
\end{itemize}
\end{defn}
\begin{prop}
 The external product $-\times-$ from Definition \ref{defn:externalproduct} is continuous. Denoting
  $p\colon(\Q,G)\times(\Q',G')\to (\Q,G)$ and $p'\colon(\Q,G)\times(\Q',G')\to (\Q',G')$ the projections, we have that $-\times-$ is a retraction of the map
 \[
 (p_*,p'_*)\colon \Hur(\fT;(\Q,G)\times(\Q',G')) \to \HurTQG\times\HurTQGp.
 \]
\end{prop}
\begin{proof}
 Let $\fc,\fc',\fc''$ be as in Definition \ref{defn:externalproduct}, and let $\uU''$ be an adapted covering of $P''$.
 Then we can obtain an adapted covering $\uU$ of $P$ (respectively, $\uU'$ of $P'$) by selecting the components
 of $\uU''$ containing one point of $P$ (respectively, of $P'$). We note that the product of normal neighbourhoods
 $\fU(\fc,\uU)\times\fU(\fc',\uU')$ is mapped by the external product inside $\fU(\fc'',\uU'')$: this shows continuity
 of the external product.
 
 For the second statement, let $\fc=(P,\psi,\phi)\in\Hur(\fT;(\Q,G)\times(\Q',G'))$:
 then both $p_*(\fc)=(P,p \circ(\psi,\phi))$ and $p'_*(\fc)=(P,p' \circ(\psi,\phi))$
 are supported on the set $P\subset\cX$, so that by Definition \ref{defn:externalproduct}
 also $p_*(\fc)\times p'_*(\fc)$ is supported on $P\cup P=P$. It now follows directly
 from Definition \ref{defn:externalproduct} that $p_*(\fc)\times p'_*(\fc)$ is equal to $\fc$.
\end{proof}
\begin{nota}
 \label{nota:simplifiedexternalproduct}
We will mainly use the external product in the case $(\Q',G')=(\one,\one)$.
By abuse of notation we will denote by $-\times-$ also the composition
 \[
 \begin{tikzcd}
  \HurTQG\times\Ran(\fT)\ar[r,"\cong"] &\HurTQG\times\Hur(\fT;\one,\one)\ar[dl,"-\times-",swap]\\
  \Hur(\fT;(\Q,G)\times(\one,\one)) \ar[r,"\big(\pr_{(\Q,G)}\big)_*",swap]  &[1em] \HurTQG.
 \end{tikzcd}
 \]
\end{nota}

\subsection{Contractible normal neighbourhoods}
The following lemma gives an effective way to prove contractibility of normal neighbourhoods in concrete
situations.
\begin{lem}
 \label{lem:contractiblenormalneigh}
Let $\fT=(\cX,\cY)$ be a nice couple and $(\Q,G)$ a PMQ-group pair.
Let $\fc\in\Hur(\cX,\Q)$, use Notations \ref{nota:fc} and \ref{nota:uUcovering},
and let $\uU$ be an adapted covering of $P$. Assume that there is a homotopy $\cH^{\uU}\colon\C\times[0,1]\to\C$
satisfying the following:
 \begin{enumerate}
  \item $\cH^{\uU}(-,t)$ is a lax morphism of nice couples $\fT\to\fT$ for all $0\leq t\leq 1$;
  \item $\cH^{\uU}(-,0)=\Id_{\C}$;
  \item $\cH^{\uU}(-,t)$
  restricts to a map $U_i\to U_i$, for all $1\leq i\leq k$ and all $0\leq t\leq 1$;
  \item $\cH^{\uU}(-,1)$ maps $U_i$ constantly to $z_i$ for all $1\leq i\leq k$.
 \end{enumerate}
 Then the normal neighbourhood $\fU(\fc,\uU)$ is contractible.
\end{lem}
\begin{proof}
We include $\Q$ into its completion $\hQ$, and consequently include $(\Q,G)$ into $(\hQ,G)$.
Recall that the inclusion $\Hur(\cX;\Q,G)\subset\Hur(\cX;\hQ,G)$ is open, and more precisely
it maps normal neighbourhoods bijectively onto normal neighbourhoods.
Thus suffices to prove that $\fU(\fc,\uU)$ is contractible when considered as a normal neighbourhood
in $\Hur(\cX;\hQ,G)$.
Proposition \ref{prop:functorialityhomotopyLNC} and property (1) give a homotopy
 \[
  \cH^{\uU}_*\colon\Hur(\cX;\hQ,G)\times[0,1]\to\Hur(\cX;\hQ,G),
 \]
Consider now the union $\coprod_{\tfc}\fU(\tfc;\uU)\subset\Hur(\cX;\hQ,G)$,
where $\tilde{\fc}$ ranges among all configurations of $\Hur(\cX;\hQ,G)$
supported on $P$. By property (3) the map $\cH^{\uU}_*$ restricts to a homotopy
 \[
  \cH^{\uU}_*\colon \coprod_{\tfc}\fU(\tfc;\uU)\times[0,1]\to \coprod_{\tfc}\fU(\tfc;\uU).
 \]
The argument in the proof of Proposition \ref{prop:Hurtopology} shows that $\coprod_{\tfc}\fU(\tfc;\uU)$
is the topological disjoint union of its open subspaces $\fU(\tilde{\fc};\uU)$.
The map $\cH^{\uU}_*(-;0)$ is the identity of $\Hur(\cX;\hQ,G)$ by property (2),
in particular $\cH^{\uU}_*(-;0)$ preserves each subspace $\fU(\tfc;\uU)$. It follows
that $\cH^{\uU}_*$ restricts to a homotopy
 \[
  \cH^{\uU}_*\colon\fU(\tfc,\uU)\times[0,1]\to\fU(\tfc,\uU)
 \]
for each $\tfc$ supported on $P$, in particular for $\tfc=\fc$.
By property (4) the map $\cH^{\uU}_*(-;1)$ takes values in configurations
in $\Hur(\fT;\Q',G)$ supported on the set $P=\set{z_1,\dots,z_k}$, and
the only such configuration inside $\fU(\fc,\uU)$ is $\fc$.
\end{proof}
The hypothesis that the spaces $\cX$ and $\cY$ occurring in a nice couple $\fT$ are semi-algebraic
implies that, given $\fc\in\Hur(\fT;\Q,G)$ supported on a finite set $P$,
one can choose a small enough adapted covering of $P$ for which a homotopy
$\cH^{\uU}$ as in Lemma \ref{lem:contractiblenormalneigh} exists. It follows that
the space $\Hur(\fT;\Q,G)$ is locally contractible.

\section{Total monodromy and group actions}
\label{sec:totmongroupactions}
In this section we define the total monodromy of configurations in $\Hur(\fT;\Q,G)$ and describe
several actions of $G$ on $\Hur(\fT;\Q,G)$ and on certain subspaces of it.

\subsection{Total monodromy}
The total monodromy is the simplest invariant of connected components of $\Hur(\fT;\Q,G)$.
\label{subsec:totmon}
\begin{defn}
\label{defn:totmon}
Let $\fT$ be a nice couple, $(\Q,G)$ a PMQ-group pair, and let $\fc=(P,\psi,\phi)\in\HurTQG$. Let $\gamma\colon[0,1]\to\C$ be a simple
closed loop spinning clockwise around $P$, i.e., $\gamma$ bounds a disc in $\C$ that contains $P$.
We define $\totmon(\fc)=\phi([\gamma])\in G$, and call it the \emph{total monodromy}
of the configuration $\fc$: this gives a function $\totmon\colon\HurTQG\to G$. See Figure \ref{fig:totmon}, left.
\end{defn}

\begin{figure}[ht]
 \begin{tikzpicture}[scale=4,decoration={markings,mark=at position 0.38 with {\arrow{>}}}]
  \draw[dashed,->] (-.1,0) to (1.1,0);
  \draw[dashed,->] (0,-1.1) to (0,1.1);
  \node at (0,-1) {$*$};
  \fill[black, opacity=.2, looseness=.8] (.1,.1) to[out=180, in=-90] (-.1,.3) to[out=90,in=-100] (.4,.9) to[out=80,in=-180]
  (.8,1) to[out=0,in=0] (.8,0) to[out=180,in=0] (.4,0) to[out=180,in=0] (.1,.1);
  \fill[black, opacity=.3, looseness=.8] (.8,0) to[out=180,in=-100] (.4,.9) to[out=80,in=-180]
  (.8,1) to[out=0,in=0] (.8,0);
  \node at (.1,.3){$\bullet$}; 
  \node at (.45,.5){$\bullet$}; 
  \node at (.5,.9){$\bullet$};
  \draw[dashed, looseness=1.2, postaction={decorate}] (0,-1) to[out=91,in=-90](-.13,.6)
  to[out=90,in=90] node[above]{\tiny$\omega(\fc)=\fe(a_1)g_2g_3$}(1.1,.6) to[out=-90,in=30] (0,-1);
  \draw[thin, looseness=1.2, postaction={decorate}] (0,-1) to[out=82,in=-90]  (.05,.1) to[out=90,in=-90] (-.1,.3)  to[out=90,in=90] node[below]{\tiny$a_1$} (.25,.3)  to[out=-90,in=90] (.1,.1) to[out=-90,in=80] (0,-1);
  \draw[thin, looseness=1.2, postaction={decorate}] (0,-1) to[out=80,in=-90] (.4,.3)
  to[out=90,in=-90] node[right]{\tiny$g_2$} (.2,.5) to[out=90,in=90] (.7,.5) to[out=-90,in=90] (.45,.3) to[out=-90,in=75]  (0,-1);
  \draw[thin, looseness=1.2, postaction={decorate}] (0,-1) to[out=53,in=-20] (.6,.75) to[out=160,in=0] (.5,.7)
  to[out=180,in=180] node[right]{\tiny$g_3$} (.5,1)   to[out=0,in=160] (.65,.77) to[out=-20,in=50] (0,-1);
\begin{scope}[shift={(1.5,0)}]
  \draw[dashed,->] (-.1,0) to (1.1,0);
  \draw[dashed,->] (0,-1.1) to (0,1.1);
  \node at (0,-1) {$*$};
  \fill[black, opacity=.2, looseness=.8] (.1,.1) to[out=180, in=-90] (-.1,.3) to[out=90,in=-100] (.4,.9) to[out=80,in=-180]
  (.8,1) to[out=0,in=0] (.8,0) to[out=180,in=0] (.4,0) to[out=180,in=0] (.1,.1);
  \node at (.1,.3){$\bullet$}; 
  \node at (.45,.5){$\bullet$}; 
  \node at (.5,.9){$\bullet$};
  \draw[dashed, looseness=1.2, postaction={decorate}] (0,-1) to[out=91,in=-90](-.13,.6)
  to[out=90,in=90] node[above]{\tiny$\hat\omega(\fc)=a_1a_2a_3$}(1.1,.6) to[out=-90,in=30] (0,-1);
  \draw[thin, looseness=1.2, postaction={decorate}] (0,-1) to[out=82,in=-90]  (.05,.1) to[out=90,in=-90] (-.1,.3)  to[out=90,in=90] node[below]{\tiny$a_1$} (.25,.3)  to[out=-90,in=90] (.1,.1) to[out=-90,in=80] (0,-1);
  \draw[thin, looseness=1.2, postaction={decorate}] (0,-1) to[out=80,in=-90] (.4,.3)
  to[out=90,in=-90] node[right]{\tiny$a_2$} (.2,.5) to[out=90,in=90] (.7,.5) to[out=-90,in=90] (.45,.3) to[out=-90,in=75]  (0,-1);
  \draw[thin, looseness=1.2, postaction={decorate}] (0,-1) to[out=53,in=-20] (.6,.75) to[out=160,in=0] (.5,.7)
  to[out=180,in=180] node[right]{\tiny$a_3$} (.5,1)   to[out=0,in=160] (.65,.77) to[out=-20,in=50] (0,-1);
\end{scope}
 \end{tikzpicture}
 \caption{On left, a configuration $\fc$ in $\Hur(\cX,\cY;\Q,G)$, whose total monodromy is the $G$-valued monodromy of the dashed loop; on right, a configuration $\fc$ in $\Hur(\cX,\Q)$, whose total monodromy is the $\hQ$-valued monodromy of the dashed loop.}
\label{fig:totmon}
\end{figure}

Note that the loop $\gamma$ is well-defined up to homotopy, so that $\totmon$ is well-defined as a map of sets. Since
for any given covering $\uU$ of $P$ we can choose $\gamma$ spinning clockwise also around $\uU$, we note that
$\totmon$ is constant on the normal neighbourhood $\fU(\fc;\uU)$; hence the total monodromy is locally constant
and therefore an invariant of connected components of $\HurTQG$.

Note also that if $\xi\colon\fT\to\fT'$ is a map of nice couples, then $\xi(\gamma)$ is homotopic to a simple 
loop spinning clockwise around $\xi(P)$ (see Definition \ref{defn:mapnicecouples});
in particular $\totmon(\fc)=\totmon(\xi_*(\fc))$, i.e. the total monodromy
is preserved under maps of Hurwitz-Ran spaces induced by maps of nice couples.
If $(\Psi,\Phi)\colon(\Q,G)\to(\Q',G')$ is a morphism of PMQ-group pairs, then
for all $\fc\in\HurTQG$ we have $\Phi(\totmon(\fc))=\totmon((\Psi,\Phi)_*(\fc))$.

\begin{nota}
 \label{nota:Hurwithtotmon}
 For a nice couple $\fT$, a PMQ-group pair $(\Q,G)$ and $g\in G$ we denote by $\HurTQG_g\subset\HurTQG$ the
 preimage of $g$ along $\totmon$. If $P_0$ is as in Definition \ref{defn:basedHur}, we denote by
 $\HurTQG_{P_0;g}$ the corresponding subspace of $\HurTQG_{P_0}$.
\end{nota}
For $P_0\subset\cX$ (possibly $P_0=\emptyset$) we obtain a natural decomposition
\[
 \HurTQG_{P_0} = \coprod_{g\in G}  \HurTQG_{P_0;g} .
\]

In the case $\cY=\emptyset$, we can refine Definition \ref{defn:totmon} by taking the values of $\totmon$
in the completion $\hQ$ of $\Q$, instead of $G$.
Let $\cX\subset\bH$ be a semi-algebraic subset, and let $\fc=(P,\psi)\in\Hur(\cX;\Q)$;
if $\gamma$ is a simple loop in $\CmP$ spinning clockwise around $P$,
then $[\gamma]\in\fQext(P)$, and since $\hQ$ is complete (and contains $\Q$), we can extend
$\psi\colon\fQ(P)\to\Q$ to a map $\psiext\colon\fQext(P)\to\hQ$, compare also with Subsection \ref{subsec:LNCfunctoriality}.
\begin{defn}
\label{defn:totmonhQ}
We define a locally constant map $\hat\totmon\colon\Hur(\cX;\Q)\to\hQ$
by setting $\hat\totmon(\fc):=\psiext([\gamma])$, using the notation above. See Figure \ref{fig:totmon}, right.
For $a\in\hQ$ and $\emptyset\subseteq P_0\subset\cX$ we let
$\Hur(\cX;\Q)_{P_0;a}\subset\Hur(\cX;\Q)$ be the preimage of $a$ along $\hat\totmon$.
\end{defn}
We obtain a decomposition
\[
 \Hur(\cX;\Q)_{P_0} = \coprod_{a\in \hQ}  \Hur(\cX;\Q)_{P_0;a}.
\]
The maps $\totmon\colon\Hur(\cX,\emptyset;\Q,\cG(\Q))\to\cG(\Q)$ and $\hat\totmon\colon\Hur(\cX;\Q)\to\hQ$
are related by the equality $\totmon=\eta_{\hQ}\circ\hat\totmon$, where we the groups $\cG(\Q)$ and $\cG(\hQ)$
are canonically identified, and $\eta_{\hQ}\colon\hQ\to\cG(\hQ)$ is the unit of the adjunction.

As a first application of the total monodromy, we prove the following proposition.
\begin{prop}
 \label{prop:totmoncontractibility}
 Let $\cX$ be a non-empty, semi-algebraic, convex and bounded subset of $\bH$, and let $\hQ$ be a complete PMQ.
 Then the connected components of $\Hur_+(\cX,\hQ)$ are contractible and there is a bijection
 \[
  \hat\totmon\colon\pi_0(\Hur_+(\cX;\hQ))\cong \hQ.
 \]
\end{prop}
\begin{proof}
Since $\cX$ is bounded, we can find a bounded, convex, semi-algebraic open set $\cX\subset U\subset\C\setminus\set{*}$.
Fix a point $z_0\in\cX$: for all $a\in \hQ$ we can define a configuration
$\fc_a=(\set{z_0},\psi_a)\in\Hur_+(\cX;\hQ)$ by
setting $\psi_a([\gamma])=a$ for a simple loop $\gamma$ spinning clockwise around $z_0$. By Lemma
\ref{lem:contractiblenormalneigh} each normal neighbourhood $\fU(\fc_a;U)$ deformation retracts onto the configuration $\fc_a$;
the statement follows from the observation that each $\fc\in\Hur_+(\cX;\hQ)$ is contained in one of these normal
neighbourhood, namely in $\fU(\fc_{\hat\totmon(\fc)};U)$.
\end{proof}
The reader will notice that in the proof of Proposition \ref{prop:totmoncontractibility} we only used that $\cX$ is star-shaped around a point $z_0\in\cX$, and not that $\cX$ is convex.
In fact the bijection $\hat\totmon\colon\pi_0(\Hur_+(\cX;\hQ))\cong \hQ$ holds whenever $\cX$ is
path connected; we will not use this more general fact and therefore we will leave its proof to the reader.
Proposition \ref{prop:totmoncontractibility} has the following corollary.
\begin{cor}
 \label{cor:totmoncontractibilityPMQ}
 Let $\cX\subset\bH$ be semi-algebraic, convex and bounded; then for all $a\in\Q$ the space $\Hur_+(\cX;\Q)_a$ is contractible.
\end{cor}
\begin{proof}
 The fact that $\hQ\setminus\Q$ is an ideal of the complete PMQ $\hQ$ implies that, for all $a\in\Q$,
 the natural inclusion $\Hur(\cX;\Q)_a\subset\Hur(\cX;\hQ)_a$ is in fact a homeomorphism.
\end{proof}

\subsection{Action by global conjugation}
\label{subsec:conjaction}
Let $(\Q,G)$ be a PMQ-group pair. Then the group $G$ acts (on right) by conjugation on $\Q$ and on $G$ itself.
In particular the right action of $G$ on $\Q$ takes the form of a map of groups $\fr\colon G\to\Aut_{\PMQ}(\Q)^{op}$,
which is part of the structure of PMQ-group pair.
The actions of $G$ on $\Q$ by conjugation is compatible with respect to the map of PMQs $\fe\colon\Q\to G$,
which is also part of the structure of PMQ-group pair. In the following we define a corresponding
right action of $G$ on the space $\Hur(\fT;\Q,G)$, for any nice couple $\fT$.
\begin{defn}
\label{defn:conjaction}
For $g\in G$ and $\fc=(P,\psi,\phi)\in\Hur(\fT)$ we define $\fc^g=(P,\psi^g,\phi^g)\in\Hur(\fT;\Q,G)$ as follows:
 \begin{itemize}
  \item $\psi^g$ is the composition $\fQ(P)\overset{\psi}{\to}\Q\overset{\fr(g)}{\to} \Q$ of maps of PMQs.
  \item $\phi^g$ is the composition $\fG(P)\overset{\phi}{\to} G \overset{(-)^g}{\to} G$ of maps of groups,
  where $(-)^g\colon g'\mapsto g^{-1}g'g$.
 \end{itemize}
The maps $(-)^g\colon\Hur(\fT,\Q,G)\to\Hur(\fT;\Q,G)$ are homeomorphisms (they map normal neighbourhoods bijectively to normal neighbourhoods)
and assemble into a right action of $G$ on the space $\Hur(\fT;\Q,G)$, called \emph{action by global conjugation}.
\end{defn}
Note that the total monodromy $\totmon$ (see Definition \ref{defn:totmon}) satisfies the formula
$\totmon(\fc^g)=\totmon(\fc)^g\in G$.
Note also that for $P_0\subset\cX$ as in Definition \ref{defn:basedHur} the action by global
conjugation restricts to the subspace $\Hur(\fT)_{P_0}$.

In the case $\cY=\emptyset$, we can refine Definition \ref{defn:conjaction} and let
the completion $\hQ$ of $\Q$ act on $\Hur(\cX;\Q)$. By definition, a (right) action of $\hQ$
on $\Hur(\cX;\Q)$ is a map of PMQs $\hQ\to\Aut_{\Top}(\Hur(\cX;\Q))^{op}$.
For $a\in \hQ$ and $\fc=(P,\psi)\in\Hur(\cX;\Q)$ we define $\fc^a=(P,\psi^a)\in\Hur(\cX;\Q)$
by setting $\psi^a$ to be the composition $\fQ(P)\overset{\psi}{\to}\Q\overset{(-)^a}{\to} \Q$.
Here we use that $\Q\subset\hQ$ is closed under conjugation by elements in $\hQ$.

\subsection{Left and right-based nice couples}
In this subsection we consider other natural actions of $G$, defined on suitable subspaces of Hurwitz-Ran spaces.
\begin{defn}
 \label{defn:taut}
 For $t\in\R$ we define a homeomorphism $\tau_t\colon(\C,*)\to(\C,*)$ by:
  \[
 \def\arraystretch{1.4}
 \tau_t(z)=\left\{
 \begin{array}{ll}
  z & \mbox{if }\Im(z)\leq -1\\
  z+t& \mbox{if } \Im(z)\geq 0\\
  z+(\Im(z)+1)t & \mbox{if }-1\leq \Im(z)\leq 0.
 \end{array}
 \right.  
 \]
\end{defn}
Note that $\tau_t(*)=*$ for all $t\in\R$.
Note also that the assignment $t\mapsto \tau_t$ defines a continuous, piecewise linear action of $\R$ on $\C$.

\begin{nota}
 \label{nota:bS}
 For $t\in\R$ we denote by $\C_{\Re\geq t}\subset\C$ the subspace containing all $z\in\C$ with $\Re(z)\geq t$.
 Similarly we define $\C_{\Re>t}$, $\C_{\Re\leq t}$, $\C_{\Re<t}$ and $\C_{\Re=t}$, the latter being a vertical line.
 For all $-\infty\leq t\leq t'\leq +\infty$ we define a subspace $\bS_{t,t'}\subset\C$ by
 \[
  \bS_{t,t'}=\tau_t(\C_{\Re\geq 0})\cap\tau_{t'}(\C_{\Re\leq 0}),
 \]
 where $\tau_{-\infty}(\C_{\Re\geq0})=\tau_{+\infty}(\C_{\Re\leq0})=\C$
 and $\tau_{-\infty}(\C_{\Re\leq0})=\tau_{+\infty}(\C_{\Re\geq0})=\emptyset$.
\end{nota}

\begin{defn}
 \label{defn:lrnicecouple}
 A \emph{left-based} nice couple is a nice couple $\fT=(\cX,\cY)$
 together with a choice of a point $\zleft\in\cY$ such that $\Re(\zleft)\le\Re(z)$ for all $z\in\cX$.
 We denote by $(\zleft,\fT)$ a left-based nice couple.
 
Similarly, a \emph{right-based} nice couple is a nice couple with a choice of a point $\zright\in\cY$
such that $\Re(z)\le\Re(\zright)$ for all $z\in\cX$.
We denote it by $(\fT,\zright)$.

A \emph{left-right-based} nice couple (shortly, \emph{lr-based}) is a nice couple which is both left- and right-based,
such that $\Re(\zleft)<\Re(\zright)$: we denote it by $(\zleft,\fT,\zright)$.
\end{defn}

\subsection{Action by left and right multiplication}
We define a left action of $G$ on the space $\Hur(\fT;\Q,G)_{\zleft}$, where $(\zleft,\fT)$ is a
left-based nice couple. Similarly, there is a right action of $G$ on
$\Hur(\fT;\Q,G)_{\zright}$ if $(\fT,\zright)$ is a right-based nice couple.
We will describe the construction focusing on the left-based case; the right-based case is analogous, and we will
mention the differences in paretheses.

For the entire subsection fix a left-based (right-based) nice couple
as in Definition \ref{defn:lrnicecouple}.
Choose an arc $\arcleft$ embedded in $\bS_{-\infty,\Re(\zleft)}$ and joining $*$ with $\zleft$;
assume also that the interior of $\arcleft$
is contained in $\mathring{\bS}_{-\infty,\Re(\zleft)}$.
(In the right-based case, we would choose an arc $\arcright$ embedded
in $\bS_{\Re(\zright),+\infty}$ joining $*$ with $\zright$ and whose interior is contained in 
$\mathring{\bS}_{\Re(\zright),+\infty}$.)
\begin{defn}
 \label{defn:leftadmgen}
Let $P\subset\cX$
with $\zleft\in P$ (resp. $\zright\in P$).
An admissible generating set
for $\fG_{\fT}(P)$ is \emph{left-based} (resp. \emph{right-based}) if it can be constructed as in Definition \ref{defn:admgenset}, using $\arcleft$
(resp. $\arcright$) as the arc associated with $\zleft$ (resp. $\zright$), and using only arcs
contained in $\bS_{\Re(\zleft),+\infty}$ (in $\bS_{-\infty,\Re(\zright)}$) for the other points of $P$.
\end{defn}
\begin{nota}
 \label{nota:genleft}
We denote by $\genleft\in\fG(P)$ (resp. $\genright$) the generator represented by a loop spinning around 
$\zleft$ (resp. $\zright$).
\end{nota}
\begin{defn}
\label{defn:actiongfc}
 Let $g\in G$;
 let $\fc\in\Hur(\fT;\Q,G)_{\zleft}$ (resp. $\fc\in\Hur(\fT;\Q,G)_{\zright}$),
 use Notation \ref{nota:fc},
 and let $\gen_1,\dots,\gen_k$ be a left-based (right-based) admissible generating set for $\fG(P)$. We let $g\cdot \fc$
 (resp. $\fc\cdot g$) be the configuration $(P,\psi',\phi')$, where:
 \begin{itemize}
  \item $\phi'$ is defined on the free group $\fG(P)$ by setting $\phi'(\genleft)=g\cdot \phi(\genleft)$ (by setting $\phi'(\genright)=\phi(\genright)\cdot g$)
  and by setting $\phi'(\gen_i)=\phi(\gen_i)$ for $1\leq i\leq k$
  such that $\gen_i\neq\genleft$ (respectively $\gen_i\neq\genright$).
  \item $\psi'$ is defined on $\fQ(P)$ using \cite[Theorem 3.3]{Bianchi:Hur1}, by setting $\psi'(\gen_i)=\psi(\gen_i)$ for all $1\leq i\leq l$
  and imposing that $(\psi',\phi')\colon(\fQ(P),\fG(P))\to(\Q,P)$ is a map of PMQ-group pairs. See Figure \ref{fig:rightaction}.
 \end{itemize}
\end{defn}
\begin{figure}[ht]
 \begin{tikzpicture}[scale=4,decoration={markings,mark=at position 0.38 with {\arrow{>}}}]
  \draw[dashed,->] (-.1,0) to (1.1,0);
  \draw[dashed,->] (0,-1.1) to (0,1.1);
  \node at (0,-1) {$*$};
  \fill[black, opacity=.2, looseness=.8] (.1,.1) to[out=180, in=-90] (-.1,.3) to[out=90,in=-100] (.4,.9) to[out=80,in=-180]
  (.8,1) to[out=0,in=90] (1.03,.5) to[out=-90,in=0] (.8,0) to[out=180,in=0] (.4,0) to[out=180,in=0] (.1,.1);
  \fill[black, opacity=.3, looseness=.8] (.8,0) to[out=180,in=-100] (.4,.9) to[out=80,in=-180]
  (.8,1) to[out=0,in=90] (1.03,.5) to[out=-90,in=0] (.8,0);
  \node at (.18,.3){$\bullet$}; 
  \node at (-.05,.3){$\bullet$}; 
  \node at (.6,.52){$\bullet$}; 
  \node at (.3,.48){$\bullet$}; 
  \node at (1.03,.5){$\bullet$}; \node at (.98,.53){\tiny$\zright$};
  \draw[thin, looseness=1.2, postaction={decorate}] (0,-1) to[out=87,in=-90] (.02,.1) to[out=90,in=-90] node[left]{\tiny$a_1$}  (-.1,.3)  to[out=90,in=90] (.08,.3)  to[out=-90,in=90] (.04,.1) to[out=-90,in=84] (0,-1);
  \draw[thin, looseness=1.2, postaction={decorate}] (0,-1) to[out=81,in=-90] (.15,.1) to[out=90,in=-90] (.1,.3)  to[out=90,in=90] (.25,.3)  to[out=-90,in=90] node[right]{\tiny$a_2$} (.17,.1) to[out=-90,in=78] (0,-1);
  \draw[thin, looseness=1.2, postaction={decorate}] (0,-1) to[out=82,in=-90]  (.4,.3)
  to[out=90,in=-90]  (.2,.5) to[out=90,in=90] (.47,.5) to[out=-90,in=90] node[left]{\tiny$a_3$} (.43,.3) to[out=-90,in=79]  (0,-1);
  \draw[thin, looseness=1.2, postaction={decorate}] (0,-1) to[out=76,in=-90]  (.55,.3)
  to[out=90,in=-90] node[right]{\tiny$g_4$} (.5,.5) to[out=90,in=90] (.7,.5) to[out=-90,in=90] (.6,.3) to[out=-90,in=73]  (0,-1);
  \draw[thin, looseness=1.2, postaction={decorate}] (0,-1) to[out=47,in=-90] (1.02,.35) to[out=90,in=-90] 
  (.8,.55)  to[out=90,in=90] node[above]{\tiny$g_5$}(1.08,.55)  to[out=-90,in=90] (1.04,.35) to[out=-90,in=44] (0,-1);
\begin{scope}[shift={(1.4,0)}]
  \draw[dashed,->] (-.1,0) to (1.1,0);
  \draw[dashed,->] (0,-1.1) to (0,1.1);
  \node at (0,-1) {$*$};
  \fill[black, opacity=.2, looseness=.8] (.1,.1) to[out=180, in=-90] (-.1,.3) to[out=90,in=-100] (.4,.9) to[out=80,in=-180]
  (.8,1) to[out=0,in=90] (1.03,.5) to[out=-90,in=0] (.8,0) to[out=180,in=0] (.4,0) to[out=180,in=0] (.1,.1);
  \fill[black, opacity=.3, looseness=.8] (.8,0) to[out=180,in=-100] (.4,.9) to[out=80,in=-180]
  (.8,1) to[out=0,in=90] (1.03,.5) to[out=-90,in=0] (.8,0);
  \node at (.18,.3){$\bullet$}; 
  \node at (-.05,.3){$\bullet$}; 
  \node at (.6,.52){$\bullet$}; 
  \node at (.3,.48){$\bullet$}; 
  \node at (1.03,.5){$\bullet$}; \node at (.98,.53){\tiny$\zright$};
  \draw[thin, looseness=1.2, postaction={decorate}] (0,-1) to[out=87,in=-90] (.02,.1) to[out=90,in=-90] node[left]{\tiny$a_1$}  (-.1,.3)  to[out=90,in=90] (.08,.3)  to[out=-90,in=90] (.04,.1) to[out=-90,in=84] (0,-1);
  \draw[thin, looseness=1.2, postaction={decorate}] (0,-1) to[out=81,in=-90] (.15,.1) to[out=90,in=-90] (.1,.3)  to[out=90,in=90] (.25,.3)  to[out=-90,in=90] node[right]{\tiny$a_2$} (.17,.1) to[out=-90,in=78] (0,-1);
  \draw[thin, looseness=1.2, postaction={decorate}] (0,-1) to[out=82,in=-90]  (.4,.3)
  to[out=90,in=-90]  (.2,.5) to[out=90,in=90] (.47,.5) to[out=-90,in=90] node[left]{\tiny$a_3$} (.43,.3) to[out=-90,in=79]  (0,-1);
  \draw[thin, looseness=1.2, postaction={decorate}] (0,-1) to[out=76,in=-90]  (.55,.3)
  to[out=90,in=-90] node[right]{\tiny$g_4$} (.5,.5) to[out=90,in=90] (.7,.5) to[out=-90,in=90] (.6,.3) to[out=-90,in=73]  (0,-1);
  \draw[thin, looseness=1.2, postaction={decorate}] (0,-1) to[out=47,in=-90] (1.02,.35) to[out=90,in=-90] 
  (.8,.55)  to[out=90,in=90] node[above]{\tiny$g_5g$}(1.08,.55)  to[out=-90,in=90] (1.04,.35) to[out=-90,in=44] (0,-1);
\end{scope}
 \end{tikzpicture}
 \caption{On left, a configuration $\fc\in\Hur(\fT,\Q,G)_{\zright}$; on right, its image under the right action of $g\in G$.}
\label{fig:rightaction}
\end{figure}

\begin{prop}
 \label{prop:actiongfc}
 For all $g\in\cG(\Q)$ the assignment $\fc\mapsto g\cdot\fc$ (respectively $\fc\mapsto\fc\cdot g$) does not depend on the choice of the
 left-based (right-based) admissible generating set, and gives rise to a continuous self-map $g\cdot-$
 of $\Hur(\fT;\Q,G)_{\zleft}$ (respectively a self-map $-\cdot g$ of $\Hur(\fT;\Q,G)_{\zright}$).
 
 The collection of all maps $g\cdot-$ (all maps $-\cdot g$) gives a left (right) action of $G$
 on $\Hur(\fT;\Q,G)_{\zleft}$ (on $\Hur(\fT;\Q,G)_{\zright}$).
\end{prop}
The proof of Proposition \ref{prop:actiongfc} is in Subsection \ref{subsec:actiongfc} of the appendix.

\subsection{Compatibilities of the left and right actions}
\begin{lem}
 \label{lem:actiontotmon}
 Let $\fT$ be a left-based (right-based) nice couple. Then the total monodromy $\totmon\colon\HurTQ\to G$
 is a $G$-equivariant map, where $G$ acts on itself by left (right) multiplication.
\end{lem}
\begin{proof}
 We focus on the left-based case.
 Let $\fc\in\HurTQG_{\zleft}$, let $\fc'=g\cdot \fc$ and
 use Notation \ref{nota:fc}. Let $\gen_1,\dots,\gen_k$ be a left-based
 admissible generating set for $P=P'$, suppose $\gen_1=\genleft$ (see Notation \ref{nota:genleft}),
 and suppose, up to permuting the indices from $2$ to $k$, that the product $\gen_1\dots\gen_k$ represents an element $[\gamma]\in\fG(P)$ as in Definition \ref{defn:totmon}.
 Let $\fg=\gen_2\dots\gen_k$, so that $[\gamma]=\genleft\cdot\fg$. Note that $\phi'(\fg)=\phi(\fg)$. Then
\[
  \totmon(g\cdot\fc) = \phi'([\gamma])=\phi'(\genleft)\cdot\phi'(\fg)=g\cdot\phi(\genleft)\cdot\phi(\fg)=
  g\cdot \phi([\gamma])=g\cdot\totmon(\fc).
\]
\end{proof}
Let now $(\zleft,\fT,\zright)$ be a lr-based nice couple:  both spaces $\HurTQG_{\zleft}$ and $\HurTQG_{\zright}$
contain $\HurTQG_{\zleft,\zright}$ as subspace, and this subspace is preserved under both actions of $G$,
on left on $\HurTQG_{\zleft}$ and on right on $\HurTQG_{\zright}$.

\begin{lem}
 \label{lem:leftrightcompatible}
 Let $(\zleft,\fT,\zright)$ be a lr-based nice couple.
 Then the left and the right actions of $G$ on $\HurTQG_{\zleft,\zright}$ commute, i.e.,
 for every $g,h\in\cG(\Q)$ the self-maps $g\cdot-$ and $-\cdot h$ of $\HurTQG_{\zleft,\zright}$ commute.
\end{lem}
\begin{proof}
Fix $\fc=(P,\phi,\psi)\in\HurTQG_{\zleft,\zright}$ and let $\arcleft$ and $\arcright$ be arcs as in Definition \ref{defn:leftadmgen}.
We can choose disjoint arcs $\arc_i$ contained in $\bS_{\Re(\zleft),\Re(\zright)}$ completing
$\arcleft,\arcright$ to a system of arcs as in Definition \ref{defn:admgenset} yielding an
admissible generating set for $\fG(P)$ which is both left- and right-based.
The equality $g\cdot(\fc\cdot h)=(g\cdot\fc)\cdot h$
follows directly from Definition \ref{defn:actiongfc}.
\end{proof}
We thus obtain a (left) action of the group $G\times G^{op}$ on $\HurTQG_{\zleft,\zright}$, and by Lemma \ref{lem:actiontotmon} the map $\totmon\colon\HurTQG_{\zleft,\zright}\to G$ is $G\times G^{op}$-equivariant.
\begin{lem}
 \label{lem:freedoublequotient}
 In the hypotheses of Lemma \ref{lem:leftrightcompatible}, the action of
 $G\times G^{op}$ on the space $\HurTQG_{\zleft,\zright}$
 is free and properly discontinuous.
\end{lem}
\begin{proof}
 Let $\fc=(P,\psi,\phi)\in\HurTQG_{\zleft,\zright}$, let $\uU$ be an adapted covering of $P$,
 and denote by $\Uleft$ and $\Uright$ the components of $\uU$ containing
 $\zleft$ and $\zright$ respectively.
 Fix an admissible generating set for $\fG(P)$ which is both left- and right-based,
 and let $\genleft$ and $\genright$ be as in Notation \ref{nota:genleft}.
 
 Let $(g,h)$ be a non-trivial element in $G\times G^{op}$,
 and denote by $\fc'=(P,\psi',\phi')$ the configuration $g\cdot\fc\cdot h$.
 Then either $\phi'(\genleft)=g\cdot\phi(\genleft)\neq \phi(\genleft)$,
 or $\phi'(\genright)=\phi(\genright)\cdot h\neq \phi(\genright)$, or both inequalities hold:
 in any case we conclude $\fc'\neq \fc$,
 so the action of $G\times G^{op}$ on $\HurTQG_{\zleft,\zright}$ is free.
 
 Recall Definition \ref{defn:basedHur}, and note that
 the normal neighbourhood $\fU(\fc,\uU)_{\zleft,\zright}$ is mapped by $g\cdot-\cdot h$ to the normal neighbourhood
 $\fU(g\cdot\fc\cdot h,\uU)_{\zleft,\zright}$;
 since the configurations $\fc$ and $\fc'$ are supported on the same set $P$,
 but $\phi'\neq \phi$, the argument in the proof of Proposition \ref{prop:Hurtopology}
 shows that $\fU(\fc,\uU)$ and $\fU(g\cdot\fc\cdot h,\uU)$ intersect trivially in $\HurTQG$,
 and a fortiori $\fU(\fc,\uU)_{\zleft,\zright}$ and $\fU(g\cdot\fc\cdot h,\uU)_{\zleft,\zright}$
 intersect trivially in $\HurTQG_{\zleft,\zright}$.
 Hence the action of $G\times G^{op}$ on $\HurTQG_{\zleft,\zright}$ is properly discontinuous. 
\end{proof}
Recall Notation \ref{nota:Hurwithtotmon}: we can decompose $\HurTQG_{\zleft,\zright}$
as a disjoint union of subspaces $\HurTQG_{\zleft,\zright;g}$ according to $\totmon$. If we act on $\HurTQG_{\zleft,\zright}$
only on left, these subspaces will be permuted among each other, so that the quotient
of $\HurTQG_{\zleft,\zright}$ by the left action is homeomorphic to $\HurTQG_{\zleft,\zright;\one}$.
The same holds if we quotient $\HurTQG_{\zleft,\zright}$ only by the right action of $G$; we can define
a more interesting space by quotienting $\HurTQG_{\zleft,\zright}$ by both actions.
\begin{nota}
\label{nota:Hur/GG}
 We denote by $\HurTQG_{G,G^{op}}$ the quotient of $\HurTQG_{\zleft,\zright}$ by the left and right
 actions of $G$: the points $\zleft,\zright$ will always be clear from the context and will thus be omitted from the notation. We denote by
 \[
 \pr_{G,G^{op}}\colon\HurTQG_{\zleft,\zright} \to \HurTQG_{G,G^{op}}
 \]
 the projection map.
\end{nota}
By Lemma \ref{lem:freedoublequotient} we have in particular a covering map
\[
\pr_{G,G^{op}}\colon\HurTQG_{\zleft,\zright;\one}\to \HurTQG_{G,G^{op}},
\]
with deck transformation group isomorphic to $G$ and acting transitively on fibres.

\section{Hurwitz-Ran spaces with monodromies in augmented PMQs}
\label{sec:augmented}
In this section we introduce, for an \emph{augmented} PMQ $\Q$, a
subspace $\Hur(\fT;\Q_+,G)$ of $\Hur(\fT;\Q,G)$;
under suitable conditions on $\fT$ the inclusion $\Hur(\fT;\Q_+,G)\hookrightarrow\Hur(\fT;\Q,G)$
is a weak homotopy equivalence. Recall from \cite[Definition 4.9]{Bianchi:Hur1} that $\Q$ is \emph{augmented} if
an equality $ab=\one$ in $\Q$ implies $a=b=\one$.

\begin{defn}
 \label{defn:Huraugmented}
 Let $\fT=(\cX,\cY)$ be a nice couple and let $(\Q,G)$ be a PMQ-group pair, with $\Q$
 augmented. We define $\Hur(\fT;\Q_+,G)\subseteq\Hur(\fT;\Q,G)$ as the subspace containing
 all configurations $\fc=(P,\psi,\phi)$ such that $\psi\colon\fQ(P)\to\Q$ is an \emph{augmented} map
 of PMQs, i.e. $\psi^{-1}(\one_{\Q})=\set{\one_{\fQ(P)}}$ (see \cite[Definition 4.9]{Bianchi:Hur1}).
 If $\cY$ is empty we also write $\Hur(\cX;\Q_+)\subset\Hur(\cX;\Q)$ for the space $\Hur(\fT;\Q_+,\cG(\Q))$.
\end{defn}
Roughly speaking, a configuration $\fc=(P,\psi,\phi)\in\HurTQG$ belongs to $\Hur(\fT;\Q_+,G)$
if the monodromy $\psi$ attains non-trivial values around each point of $P\setminus\cY$:
these are also all points of $P$ around which $\psi$ is defined. By ``non-trivial value''
we mean a value different from $\one_{\Q}$, i.e. a value in $\Q_+$, whence the notation.

Note that if $\xi\colon \fT\to\fT'$ is a morphism of nice couples and $\Q$ is augmented, then the induced map 
$\xi_*\colon\Hur(\fT;\Q,G)\to\Hur(\fT';\Q,G)$ restricts to a map of spaces
$\xi_*\colon\Hur(\fT;\Q_+,G)\to\Hur(\fT';\Q_+,G)$. This is true also for a lax morphism
$\xi\colon \fT\to\fT'$ (see Definition \ref{defn:laxmapnicecouples}), provided that $\Q$ is complete and augmented.

\begin{lem}
 \label{lem:Huraugmentedclosed}
 If $\Q$ is augmented, then $\Hur(\fT;\Q_+,G)$ is closed in $\Hur(\fT;\Q,G)$.
\end{lem}
\begin{proof}
 Let $\fc\in \Hur(\fT;\Q,G)\setminus\Hur(\fT;\Q_+,G)$; then, using Notation \ref{nota:fc}, there is some $1\leq i\leq l$ such that
 $\psi$ sends each element of $\fQ(P,z_i)$ to $\one$ (see also Definition \ref{defn:fQP}).
 Let $\uU$ be an adapted covering of $P$: then we claim that
 the entire normal neighbourhood $\fU(\fc;\uU)$ lies in the difference $ \Hur(\fT;\Q,G)\setminus\Hur(\fT;\Q_+,G)$.
 To see this, let $\fc'=(P',\psi',\phi')\in\fU(\fc,\uU)$, use Notation \ref{nota:uUcovering}, and 
 let $z'$ be a point in $P'\cap U_i$. Then each element $[\gamma']\in\fQ(P',z')$
 is sent by $\psi'$ to an element $\psi'([\gamma'])\in\Q$ which occurs as a factor of a decomposition of
 $\one_{\Q}$ in the partial monoid $\Q$; since $\Q$ is augmented
 we have $\psi'([\gamma'])=\one$ and therefore $\fc'$ does not lie in $\Hur(\fT;\Q_+,G)$.
\end{proof}

\subsection{Homotopy equivalences from augmented PMQs}
The rest of the section is devoted to the proof of the following technical propositions.
\begin{prop}
 \label{prop:Huraugmentedabsolute}
Let $\cX\subset\bH$ be a semi-algebraic, non-empty and connected subspace, and let $\Q$
be an augmented PMQ. Then the spaces $\Hur(\cX;\Q_+)$ and $\Hur_+(\cX;\Q)$
are homotopy equivalent.
\end{prop}
\begin{prop}
\label{prop:Huraugmentedrelative}
 Let $(\Q,G)$ be a PMQ-group pair with $\Q$ augmented, and let $\fT=(\cX,\cY)$ be a nice couple
 with both $\cX$ and $\cY$ non-empty and connected. Let $P_0\subset\cY$ be a finite, non-empty subset.
 Then the inclusion $\Hur(\fT;\Q_+,G)_{P_0}\subset\Hur(\fT;\Q,G)_{P_0}$
 is a homotopy equivalence.
\end{prop}
In the rest of the section we fix a PMQ-group pair $(\Q,G)$ with $\Q$ augmented.
Let $\fT=(\cX,\cY)$ be a nice couple.
\begin{defn}
 \label{defn:inert}
 Let $\fc=(P,\psi,\phi)\in\Hur(\fT;\Q,G)$. A point $z\in P$ is \emph{inert} for $\fc$ if
 $z\in \cX\setminus\cY$ and $\psi$ maps each element of $\fQ(P,z)$ to $\one_{\Q}$
 (see Definition \ref{defn:fQP}).
\end{defn}
\begin{defn}
 \label{defn:rhoretraction}
We define a retraction of sets
$\rho\colon\Hur(\fT;\Q,G)\to\Hur(\fT;\Q_+,G)$ of the inclusion $\Hur(\fT;\Q_+,G)\subset\Hur(\fT;\Q,G)$:
for a configuration $\fc\in \Hur(\fT;\Q,G)$, we construct $\rho(\fc)$ by ``forgetting''
its inert points.
More precisely, using Notation \ref{nota:fc}, if $P'\subset P$
is the subset of non-inert points for $\fc$, then we note
that $\phi\colon\fG(P)\to G$ and $\psi\colon\fQ(P)\to\Q$ factor
through maps $\phi'\colon\fG(P')\to G$ and $\psi'\colon\fQ(P')\to\Q$
along the surjections $\fri^{P}_{P'}\colon\fG(P)\to\fG(P')$ and $\fri^{P}_{P'}\colon\fQ(P)\to\fQ(P')$,
(see Notation \ref{nota:fri}), and we define
$\rho\colon(P,\psi,\phi)\mapsto (P',\psi',\phi')$.
\end{defn}

Unfortunately, even assuming that $\Q$ is augmented, $\rho$ is in general not continuous:
for instance, if $P$ contains a point $z_i\in\cY$ whose local monodromy with respect to $\phi$
is $\one\in G$, then $\rho(P,\psi,\phi)$ is a configuration supported also on the point
$z_i$; however if we perturb slightly $z_i$ so that it ``enters'' in $\cX\setminus\cY$ (for this,
suppose that $z_i$ is an accumulation point for $\cX\setminus\cY$),
then in defining $\rho(P,\psi,\phi)$ we forget $z_i$ and we do not
replace it by any other point close to it.

\subsection{Explosions}
The previous issue can only occur when $\cY\neq\emptyset$, and in fact if $\cY=\emptyset$,
then $\rho\colon\Hur(\fT;\Q,G)\to\Hur(\fT;\Q_+,G)$ is continuous, as we will see in Corollary \ref{cor:emptyexplosion}.
In the general case we cannot just let an inert point $z_i\in P\setminus\cY$ disappear;
what we can do is to let every point $z_i\in P$
\emph{explode} (including non-inert points),
by replacing $z_i$ with one or more other points of $\cX$. This idea is elaborated
in the following definition.

\begin{defn}
\label{defn:explosion}
Let $\fT=(\cX,\cY)$ be a nice couple. An \emph{explosion} $\expl$ of $\fT$ is a
continuous map $\expl\colon\cX\times[0,1]\to\Ran(\fT)$ such that 
for all $z\in \cY$ and all $0\leq t\leq 1$, $z\in\expl(z,t)$.
An explosion $\expl$ is \emph{standard} if $\expl(z,0)=\set{z}\in\Ran(\fT)$ for all $z\in\cX$.

Given an explosion $\expl$, a finite subset $P\subset\cX$ and a time $0\leq t\leq 1$,
we can define a subset $\expl(P,t)\in\Ran(\fT)$ as the union of the subsets
$\expl(z,t)$ for $z\in P$.
Thus an explosion $\expl$ induces a continuous map
$\Ran(\fT)\times[0,1]\to\Ran(\fT)$, that
by abuse of notation we still denote $\expl$.
If $\expl$ is standard, then $\expl(-,0)$ is the identity of $\Ran(\fT)$.
\end{defn}

\begin{prop}
 \label{prop:explosion}
 Let $\fT=(\cX,\cY)$ be a nice couple, let $\expl\colon\cX\times[0,1]\to\Ran(\fT)$ be an explosion, and
 let $(\Q,G)$ be a PMQ-group pair with $\Q$ augmented.
 Recall Definition \ref{defn:epsilon} and Notation \ref{nota:simplifiedexternalproduct}, and let
 $\expl_*\colon\Hur(\fT;\Q,G)\times[0,1]\to\Hur(\fT;\Q,G)$ be the following composition of maps of sets:
 \[
  \begin{tikzcd}[column sep=45pt]
   \Hur(\fT;\Q,G)\times[0,1]
   \ar[r,"{\big(\rho\,,\,\epsilon\big)\times\Id}"]
   & \Hur(\fT;\Q_+,G)\times \Ran(\fT) \times [0,1]
   \ar[dl,"\Id\times\expl",swap]\\
   \Hur(\fT;\Q_+,G)\times \Ran(\fT)  \ar[r,"-\times-"]  & \Hur(\fT;\Q,G).
  \end{tikzcd}
 \] 
 Then $\expl_*$ is continuous. If moreover $\expl$ is standard, then
 $\expl_*(-,0)$ is the identity of $\Hur(\fT;\Q,G)$.
\end{prop}
The proof of Proposition\ref{prop:explosion} is in Subsection \ref{subsec:explosion} of the appendix.
A particular application of Proposition \ref{prop:explosion} is the following:
\begin{cor}
 \label{cor:emptyexplosion}
Let $\cX\subset\bH$ be a semi-algebraic
set and let $\Q$ be an augmented PMQ; then the map $\rho\colon\Hur(\cX;\Q)\to\Hur(\cX;\Q_+)$ is continuous.
\end{cor}
\begin{proof}
Consider the explosion $\expl^{\emptyset}\colon\cX\times[0,1]\to\Ran(\cX)$
taking the constant value $\emptyset\in\Ran(\cX)$; then $\expl^{\emptyset}_*(-,0)=\rho$ is a continuous
map. 
\end{proof}

\subsection{Proof or Propositions \ref{prop:Huraugmentedabsolute} and \ref{prop:Huraugmentedrelative}}
\begin{proof}[Proof of Proposition \ref{prop:Huraugmentedabsolute}]
Recall Definition \ref{defn:externalproduct} and Notation \ref{nota:simplifiedexternalproduct},
and fix a point $z_0\in\cX$.
We claim that the map $-\times z_0\colon \Hur(\cX;\Q_+)\to\Hur_+(\cX;\Q)$
is a homotopy equivalence, where we denote by $z_0$ also the singleton $\set{z_0}\in\Ran(\cX)$.

Let $\Hur(\cX;\Q_+)_{z_0}$ be the subspace of $\Hur_+(\cX;\Q)$
containing
all configurations $\fc=(P,\psi)$ such that
$z_0\in P$ and all points of $P\setminus\set{z_0}$ are not inert;
then $-\times z_0$ gives a homeomorphism
$\Hur(\cX;\Q_+)\overset{\cong}{\to}\Hur(\cX;\Q_+)_{z_0}$, with inverse given by the restriction of $\rho$,
which is continuous by Corollary \ref{cor:emptyexplosion}.
It suffices therefore to prove that the inclusion $\Hur(\cX;\Q_+)_{z_0}\hookrightarrow\Hur_+(\cX;\Q)$
is a homotopy equivalence.
By Lemma \ref{lem:Ran+contractible} the space $\Ran_+(\cX)$ is weakly contractible; since $\cX$ is homeomorphic
to a CW complex, there is a homotopy $\expl^{z_0}\colon\cX\times [0,1]\to\Ran_+(\cX)$ with
$\expl^{z_0}(z,0)=\set{z}$ and $\expl^{z_0}(z,1)=\set{z_0}$ for all $z\in\cX$; $\expl^{z_0}$
is a standard explosion and gives rise to an extended explosion
$\expl^{z_0}\colon\Ran(\cX)\times[0,1]\to\Ran(\cX)$.

Proposition \ref{prop:explosion} yields a homotopy
$\expl^{z_0}_*\colon\Hur(\cX;\Q)\times[0,1]\to\Hur(\cX;\Q)$,
which restricts to a homotopy of $\Hur_+(\cX;\Q)$. We note the following:
\begin{itemize}
 \item $\expl^{z_0}_*(-,0)$ is the identity of $\Hur_+(\cX;\Q)$, again by Proposition \ref{prop:explosion};
 \item $\expl^{z_0}_*(-,1)$ restricts to the identity on the subspace $\Hur(\cX;\Q_+)_{z_0}$:
 indeed if $\fc=(P,\psi)\in\Hur(\cX;\Q_+)_{z_0}$, then $\rho(\fc)$
 is either equal to $\fc$, or is obtained by forgetting $z_0\in P$ in case $z_0$ is inert;
 since $\expl^{z_0}(-,1)$ is constant on $\Ran(\cX)$ with value $z_0$, we have anyway the equality
 $\rho(\fc)\times\expl^{z_0}(P,1)=\rho(\fc)\times z_0=\fc$, i.e. the point $z_0$ is
 added again in the further composition defining $\expl_*^{z_0}(-,1)$;
 \item $\expl_*^{z_0}(-,1)$ takes values in $\Hur(\cX;\Q_+)_{z_0}$: this follows again
from the equality $\expl_*(\fc,1)=\rho(\fc)\times z_0$, holding for all $\fc\in\Hur(\cX;\Q)$.
\end{itemize}
The homotopy $\expl^{z_0}_*$ shows that the inclusion
$\Hur(\cX;\Q_+)_{z_0}\hookrightarrow\Hur_+(\cX;\Q)$ is a homotopy equivalence.
\end{proof}

Note that in the particular case $\Q=\set{\one}$, Proposition \ref{prop:Huraugmentedabsolute}
implies that $\Ran_+(\cX)$ is contractible; this is a mild improvement of the statement of Lemma \ref{lem:Ran+contractible}.
\begin{proof}[Proof of Proposition \ref{prop:Huraugmentedrelative}]
 The proof is similar to the one of Proposition \ref{prop:Huraugmentedabsolute}.
 By Lemma \ref{lem:Ran+contractible} the spaces $\Ran_+(\cY)\subset \Ran_+(\cX)$ are weakly
 contractible. The couple $(\cX,\cY)$ is homeomorphic to a couple of CW complexes,
 therefore we can find a homotopy $\expl^{\cX,\cY}\colon\cX\times [0,1]\to\Ran_+(\cX)$ with
 $\expl^{\cX,\cY}(z,t)=\set{z}$ whenever $z\in\cY$ or $t=0$, and such that 
 $\expl^{\cX,\cY}(z,1)\in\Ran_+(\cY)$ for all $z\in\cX$.
 In particular $\expl^{\cX,\cY}$ is a standard explosion, inducing
 an extended explosion $\expl^{\cX,\cY}\colon\Ran(\cX)\times[0,1]\to\Ran(\cX)$. By
 Proposition \ref{prop:explosion}, and using that $P_0\subset\cY$, we obtain a homotopy $\expl^{\cX,\cY}_*\colon\Hur(\fT;\Q,G)_{P_0}\times[0,1]\to\Hur(\fT;\Q,G)_{P_0}$ with the following properties:
 \begin{itemize}
  \item $\expl^{\cX,\cY}_*(-,0)$ is the identity of $\Hur(\fT;\Q,G)_{P_0}$;
  \item $\expl^{\cX,\cY}(-,1)$ takes values in $\Hur(\fT;\Q_+,G)_{P_0}$.
\end{itemize}
It suffices now to prove that there is a homotopy of maps
$\Hur(\fT;\Q_+,G)_{P_0}\to \Hur(\fT;\Q_+,G)_{P_0}$ from $\expl^{\cX,\cY}(-,1)|_{\Hur(\fT;\Q_+,G)_{P_0}}$
to the identity of $\Hur(\fT;\Q_+,G)_{P_0}$.

Using weak contractibility of $\Ran_+(\cY)$ (see Lemma \ref{lem:Ran+contractible}) together with the fact
that $\cX$ is homeomorphic to a CW complex, we can find a homotopy
$\expl^{\cY}\colon\cX\times[0,1]\to\Ran_+(\cY)$ satisfying the following properties:
\begin{itemize}
 \item $\expl^{\cY}(-,0)=\expl^{\cX,\cY}(-,1)$;
 \item $\expl^{\cY}(-,1)$ is the constant map with value $P_0\in\Ran_+(\cY)$.
\end{itemize}
Denote by $\expl^{\cY}\colon\Ran(\cX)_{P_0}\times[0,1]\to\Ran_+(\cY)$ also the induced map on Ran spaces.
Consider the homotopy $\cH^{\cY}\colon\Hur(\fT;\Q_+,G)_{P_0}\times[0,1]\to\Hur(\fT;\Q_+,G)_{P_0}$
given by the composition
\[
\begin{tikzcd}[column sep=2cm]
 \Hur(\fT;\Q_+,G)_{P_0}\times[0,1]\ar[r,"{(\Id,\epsilon)\times\Id}"] &\Hur(\fT;\Q_+,G)_{P_0}\times\Ran(\cX)_{P_0}\times[0,1]\ar[dl,"\Id\times\expl^{\cY}"']\\
 \Hur(\fT;\Q_+,G)_{P_0}\times \Ran_+(\cY)\ar[r,"-\times-"] & \Hur(\fT;\Q_+,G)_{P_0},
\end{tikzcd}
\]
where $\epsilon$ is from Definition \ref{defn:epsilon}
and in the last step we use Notation \ref{nota:simplifiedexternalproduct}. 
Since $\rho$ restricts to the identity on $\Hur(\fT;\Q_+,G)_{P_0}$, the map $\cH^{\cY}(-,0)$
concides with the restriction of $\expl^{\cX,\cY}(-,1)$ to $\Hur(\fT;\Q_+,G)_{P_0}$.
On the other hand, $\cH^{\cY}(-,1)$ is the identity of $\Hur(\fT;\Q_+,G)_{P_0}$.
\end{proof}

\subsection{An application of contractibility of Ran spaces}
The following proposition deals with a generic PMQ-group pair $(\Q,G)$, with $\Q$ possibly non-augmented; its proof relies on Proposition \ref{prop:Huraugmentedabsolute}.
\begin{prop}
\label{prop:basedHur}
 Let $\fT=(\cX,\cY)$ be a nice couple with $\cX$ non-empty and connected, let $P_0\in\Ran_+(\cX)$,
 and let $(\Q,G)$ be a PMQ-group pair. Then the inclusion $\Hur(\fT;\Q,G)_{P_0}\hookrightarrow\Hur_+(\fT;\Q,G)$ is a homotopy equivalence.
\end{prop}
\begin{proof}
By Proposition \ref{prop:Huraugmentedabsolute} there is a homotopy
$\expl^{P_0}\colon\Ran_+(\fT)\times[0,1]\to \Ran_+(\fT)$
contracting $\Ran_+(\fT)$ onto the configuration $P_0$.
We use $\expl^{P_0}$ to define a homotopy $\cH^{P_0}\colon\Hur_+(\fT;\Q,G)\times[0,1]\to\Hur_+(\fT;\Q,G)$ as the composition
 \[
 \begin{tikzcd}[column sep =2cm]
  \Hur_+(\fT;\Q,G)\times[0,1] \ar[r,"{(\Id,\epsilon)\times\Id}"] &
  \Hur_+(\fT;\Q,G)\times \Ran_+(\fT) \times[0,1] \ar[dl,"\Id\times\expl^{P_0}",swap]\\ 
  \Hur_+(\fT;\Q,G)\times \Ran_+(\fT) \ar[r,"-\times-"]&  \Hur_+(\fT;\Q,G),
 \end{tikzcd}
 \]
where $\epsilon$ and $-\times-$ are from Definition \ref{defn:epsilon}
and Notation \ref{nota:simplifiedexternalproduct}. Note the following:
\begin{itemize}
 \item $\cH^{P_0}(-;0)$ is the identity of $\Hur_+(\fT;\Q,G)$;
 \item $\cH^{P_0}(-;1)$ restricts to the identity on $\Hur(\fT;\Q,G)_{P_0}$;
 \item for all $0\leq s\leq 1$ and $\fc\in\Hur_+(\fT;\Q,G)$,
 if we denote $\fc'=\cH^{P_0}(\fc;s)$ and use Notation \ref{nota:fc}, then $P\subseteq P'$; if moreover we assume $s=1$, then $P_0\subseteq P'$.
\end{itemize}
In particular $\cH^{P_0}$ preserves $\Hur(\fT;\Q,G)_{P_0}$ at all times and
$\cH^{P_0}(-;1)$ takes values in $\Hur(\fT;\Q,G)_{P_0}$. This implies the statement.
\end{proof}

\section{Cell stratifications}
\label{sec:celldecomposition}
Throughout the section we fix an augmented PMQ $\Q$ with completion $\hQ$.
\begin{nota}
 \label{nota:mcR}
 We denote by $\mcR$ the open unit square $(0,1)^2\subset\bH$, and by $\cR$ the closed unit square
 $[0,1]^2\subset\bH$.
\end{nota}
In this section we introduce a \emph{cell stratification} on $\Hur(\mcR;\Q_+)$.
More precisely, we will do the following:
\begin{enumerate}
 \item we
 regard $\Hur(\mcR;\Q_+)$ as an open subspace of $\Hur(\cR;\hQ_+)$, by applying functoriality
 to the inclusions $\mcR\hookrightarrow\cR$ and $\Q\hookrightarrow\hQ$;
 \item we consider the bisimplicial complex $\Arr(\Q)$ from \cite[Definition 6.6]{Bianchi:Hur1}; 
 and define a continuous bijection $\upsilon\colon|\Arr(\Q)|\to\Hur(\cR;\hQ_+)$,
 by glueing suitable continuous maps $e^{\ua}\colon\Delta^p\times\Delta^q\to\Hur(\cR;\hQ_+)$, one for every non-degenerate array $\ua\in\Arr_{p,q}(\Q)$;
 \item the map $\upsilon$ restricts to a continuous bijection with source the simplicial Hurwitz space
 \[
 \upsilon\colon\Hur^{\Delta}(\Q):=|\Arr(\Q)|\setminus|\NAdm(\Q)|\to \Hur(\mcR;\Q),
 \]
 and in the additional hypothesis that $\Q$ is a locally finite PMQ, this latter bijection is a homeomorphism.
\end{enumerate}

\subsection{A construction with simplices}
We start by fixing some notation and by making some constructions with simplices and products of simplices.
For $p\geq 0$, we regard the standard $p$-simplex $\Delta^p$ as the subspace
of $[0,1]^p$ containing all $p$-tuples $\us=(s_1,\dots,s_p)$ with $0\le s_1\leq\dots\leq s_p\le1$.
\begin{nota}
 \label{nota:extracoordsimplex}
 Whenever needed, we extend each $p$-tuple $\us=(s_1,\dots,s_p)$ representing a point in $\Delta^p$ to a $p+2$-tuple
 $s_0,\dots,s_{p+1}$ by setting $s_0=0$ and $s_{p+1}=1$.
\end{nota}
\begin{nota}
 \label{nota:barycentresimplex}
 For $p\geq0$ we denote by $\ubar^p=(\bary^p_1,\dots,\bary^p_p)=(\frac 1{p+1},\dots,\frac p{p+1})$
 the barycentre of $\Delta^p$.
\end{nota}

\begin{defn}
 \label{defn:bDeltapp}
 We denote by $\bDelta^{p,p}\subset\Delta^p\times\Delta^p$ the subspace containing all pairs
 $(\us,\us')$ such that the following holds: for all $0\leq i\leq p$, if $s_i=s_{i+1}$ then $s'_i=s'_{i+1}$.
\end{defn}
\begin{defn}
 \label{defn:cHp}
 We define a continuous map $\cH^p\colon\R\times\bDelta^{p,p}\to\R$ by the formula
 \[
  \cH^p(x;\us,\us')=\left\{
  \begin{array}{cl}
   x & \mbox{if } x\in\R\setminus(0,1);\\[5pt]
   \frac{x-s_i}{s_{i+1}-s_i}s'_{i+1} + \frac{s_{i+1}-x}{s_{i+1}-s_i}s'_i & \mbox{if }s_i\neq s_{i+1}\mbox{ and }x\in[s_i,s_{i+1}];\\[5pt]
   s'_i=s'_{i+1} & \mbox{if }x=s_i=s_{i+1}.
  \end{array}
  \right.
 \]
\end{defn}
Roughly speaking, $\cH^p(-;\us,\us')\colon\R\to\R$ is constructed by fixing $(-\infty,0]\cup[1,\infty)$ pointwise,
by mapping each $s_i\mapsto s'_i$ and by extending by linear interpolation on the segments
$[0,s_1],\dots,[s_p,1]$; some of these segments might be degenerate, in this case no extension is needed.
The subspace $\bDelta^{p,p}\subset\Delta^p\times\Delta^p$ is essentially defined as the subspace of couples
$(\us,\us')$ for which $\cH^p(-;\us,\us')\colon\R\to\R$ is well-defined and continuous.
The previous description of $\cH^p$ implies directly the following lemma.
\begin{lem}
 \label{lem:bDeltacompositionR}
Let $\us,\us',\us''\in\Delta^p$
such that both pairs $(\us,\us')$ and $(\us',\us'')$ lie in $\bDelta^{p,p}$; then also $(\us,\us'')\in\bDelta^{p,p}$,
and the map $\cH^p(-;\us,\us'')\colon\R\to\R$ coincides with the composition
$\cH^p(-;\us',\us'')\circ\cH^p(-;\us,\us')$.
\end{lem}

\begin{defn}
 \label{defn:cHpq}
 For all $p,q\geq0$ we define a map
 $\cH^{p,q}\colon\C\times\bDelta^{p,p}\times\bDelta^{q,q}\to\C$ by
 \[
  \cH^{p,q}(x+y\sqrt{-1}\ ;\ \us,\us'\ ;\ \ut,\ut')=\cH^p(x;\us,\us')+\cH^q(y;\ut,\ut')\sqrt{-1}.
 \]
\end{defn}
Note that for all $(\us,\us';\ut,\ut')\in\bDelta^{p,p}\times\bDelta^{q,q}$ the map
$\cH^{p,q}(-;\us,\us';\ut,\ut')\colon\C\to\C$ is a lax self morphism of the nice couple $(\cR,\emptyset)$
(see Definition \ref{defn:laxmapnicecouples}).
By Proposition \ref{prop:functorialityhomotopyLNC} we obtain a map
\[
 \cH^{p,q}_*\colon \Hur(\cR;\hQ_+)\times\bDelta^{p,p}\times\bDelta^{q,q}\to \Hur(\cR;\hQ_+).
\]

Note also that $\Delta^p$ embeds diagonally into $\bDelta^{p,p}$; the restricted map
\[
\cH^{p,q}_*\colon \Hur(\cR;\hQ_+)\times\Delta^p\times\Delta^q\to\Hur(\cR;\hQ_+)
\]
is just the projection on the first factor, since for all $(\us,\ut)\in\Delta^p\times\Delta^q$
the map
$\cH^{p,q}(-;\us,\us;\ut,\ut)\colon\C\to\C$
is the identity of $\C$.

Lemma \ref{lem:bDeltacompositionR} applied twice, together with functoriality,
yields the following lemma.
\begin{lem}
 \label{lem:bDeltacompositionC}
 Let $\us,\us',\us''\in\Delta^p$ such that $(\us,\us'),(\us',\us'')\in\bDelta^{p,p}$,
 and let $\ut,\ut',\ut''\in\Delta^q$ such that $(\ut,\ut'),(\ut',\ut'')\in\bDelta^{q,q}$.
 Then $(\us,\us'';\ut,\ut'')\in\bDelta^{p,p}\times\bDelta^{q,q}$, and the following equality
 of maps $\Hur(\cR;\hQ_+)\to\Hur(\cR;\hQ_+)$ holds:
 \[
  \cH^{p,q}_*(-;\us',\us'';\ut',\ut'')\circ\cH^{p,q}_*(-;\us,\us';\ut,\ut')=
  \cH^{p,q}_*(-;\us,\us'';\ut,\ut'').
 \]
\end{lem}

\subsection{The array filtration}
The next step is to define a filtration on $\Hur(\cR;\hQ_+)$ by closed subspaces $\Farr_\nu\Hur(\cR;\hQ_+)$,
for $\nu\ge-1$.
\begin{defn}
 \label{defn:Farr}
 Let $P\in\Ran(\cR)$ be a finite, possibly empty subset of $\cR$, and use Notation \ref{nota:P}-
 We define
 the \emph{horizontal array degree} of $P$, denoted $\arr_{\hor}(P)\geq0$, as the cardinality
 of the finite set $\Re(P)\setminus\set{0,1}=\set{\Re(z_1),\dots,\Re(z_k)}\setminus\set{0,1}$;
 similarly we define
 the \emph{vertical array degree} of $P$, denoted $\arr_{\ver}(P)\geq0$, as the cardinality
 of the finite set $\Im(P)\setminus\set{0,1}=\set{\Im(z_1),\dots,\Im(z_k)}\setminus\set{0,1}$.
 The \emph{array bidegree} $\arr(P)$ is defined as the couple $(\arr_{\hor}(P),\arr_{\ver}(P))$,
 and the \emph{total array degree} is defined as $|\arr|(P)=\arr_{\hor}(P)+\arr_{\ver}(P)$.

 For $\fc=(P,\psi)\in\Hur(\cR;\hQ_+)$ we define $\arr(\fc)=(\arr_{\hor}(\fc),\arr_{\ver}(\fc)):=\arr(P)$,
 and $|\arr|(\fc)=|\arr|(P)$. 
 For $\nu\geq -1$ we define $\Farr_\nu\Hur(\cR;\hQ_+)$ as the subspace
 of $\Hur(\cR;\hQ_+)$ containing all configurations $\fc$ with $|\arr|(\fc)\leq\nu$.
 For $\nu\geq0$ we define $\fFarr_\nu\Hur(\cR;\hQ_+)$ as the difference
 $\Farr_\nu\Hur(\cR;\hQ_+)\setminus\Farr_{\nu-1}\Hur(\cR;\hQ_+)$.
\end{defn}
Roughly speaking, $\Farr_\nu\Hur(\cR;\hQ_+)$ contains configurations $\fc=(P,\psi)$
such that the total number of horizontal and vertical lines passing through some point of $P$,
excluding the sides of $\cR$, does not exceed $\nu$.
Similarly, $\fFarr_\nu\Hur(\cR;\hQ_+)$ contains those configurations
$\fc$ for which this total number of lines is equal to $\nu$.

\begin{lem}
 \label{lem:Farrclosed}
 For $\nu\geq -1$ the subspace $\Farr_\nu\Hur(\cR;\hQ_+)\subset\Hur(\cR;\hQ_+)$ is closed.
\end{lem}
\begin{proof}
 We prove that $\Hur(\cR;\hQ_+)\setminus\Farr_\nu\Hur(\cR;\hQ_+)$ is open. Let $\fc\in\Hur(\cR;\hQ_+)$
 be a configuration with $|\arr(\fc)|\geq\nu+1$, use Notations \ref{nota:fc} and \ref{nota:uUcovering},
 and let $\uU$ be an adapted covering of $P$ with the following property: for all $1\leq i\leq k$
 the projection $\Re(U_i)\subset\R$ intersects the finite set $\Re(P)\cup\set{0,1}$ only
 in the point $\Re(z_i)$, and the projection $\Im(U_i)\subset\R$ intersects the finite set $\Im(P)\cup\set{0,1}$
 only in the point $\Im(z_i)$.
 
 We claim that $\fU(\fc;\uU)$ is contained in $\Hur(\cR;\hQ_+)\setminus\Farr_\nu\Hur(\cR;\hQ_+)$.
 Let $\fc'\in\fU(\fc;\uU)$; then $P'$ intersects each $U_i$ in at least one point;
 by the choice of $\uU$ we have $\arr_{\hor}(\fc')\geq\arr_{\hor}(\fc)$
 and $\arr_{\ver}(\fc')\geq\arr_{\ver}(\fc)$, hence $|\arr|(\fc')\geq |\arr|(\fc)\geq\nu+1$.
\end{proof}

\subsection{Standard generating set}
\label{subsec:stdgenset}
We introduce, for a finite set $P\subset\cR$, a particular admissible generating set
of $\fG(P)$, the \emph{standard generating set}.
Fix $P\subset\cR$, and let $\Re(P)\cup\set{0,1}$ consist of the points $0=x_0<x_1<\dots<x_p<x_{p+1}=1$,
where $p=\arr_{\hor}(P)$; similarly let $0=y_0<y_1<\dots<y_q<y_{q+1}=1$ be the elements of $\Im(P)\cup\set{0,1}$,
where $q=\arr_{\ver}(P)$.
For all $(i,j)\in\set{0,\dots,p+1}\times\set{0,\dots,q+1}$ denote by $z_{i,j}$ the complex number
$x_i+y_j\sqrt{-1}\in\C$, and let
$I(P)\subset\set{0,\dots,p+1}\times\set{0,\dots,q+1}$
be the subset of pairs $(i,j)$ such that $z_{i,j}$ is a point of $P$.

Recall Notation \ref{nota:bS}.For all $(i,j)\in I(P)$ with 
$0\leq i\leq p$ let $\arc^{P,\std}_{i,j}$ be an arc contained in $\bS_{x_i,x_{i+1}}$
and joining $*$ with $z_{i,j}$.
Similarly, for all $(p+1,j)\in I(P)$
let $\arc^{P,\std}_{p+1,j}$ be an arc contained in $\bS_{1,\infty}$ joining $*$ with $z_{p+1,j}$.
Up to changing the arcs by an isotopy, we may assume that the arcs $\arc^{P,\std}_{i,j}$ are disjoint away from $*$.
Note also that these arcs are uniquely determined up to an ambient isotopy of $\C$
that fixes $P$ pointwise and preserves each subspace $\bS_{x_i,x_{i+1}}$.
\begin{defn}
 \label{defn:stdgenset}
 Recall Definition \ref{defn:admgenset}. We denote by $(\gen^{P,\std}_{i,j})_{(i,j)\in I(P)}$ the admissible generating set of $\fG(P)$ associated with the arcs
$\arc^{P,\std}_{i,j}$, and call it the \emph{standard generating set} for $\fG(P)$. See Figure \ref{fig:stdgenset}.
\end{defn}

\begin{figure}[ht]
 \begin{tikzpicture}[scale=4,decoration={markings,mark=at position 0.38 with {\arrow{>}}}]
  \fill[black, opacity=.2] (0,0) rectangle (1,1);
  \draw[dashed,->] (-.3,0) to (1.3,0);
  \draw[dashed,->] (0,-1.1) to (0,1.1);
  \node at (0,-1) {$*$};
  \node at (-.05,-.05){\tiny $0$};
  \node at (1,-.05){\tiny $1$};
  \node[anchor=east] at (-.01,-1){\tiny $-\sqrt{-1}$};
  \node[anchor=east] at (-.01,1){\tiny $\sqrt{-1}$};
  \node at (.2,-.05){\tiny $s_1$};
  \node at (.6,-.05){\tiny $s_2$};
  \node[anchor=east] at (-.01,.2){\tiny $t_1\sqrt{-1}$};
  \node[anchor=east] at (-.01,.5){\tiny $t_2\sqrt{-1}$};
  \node[anchor=east] at (-.01,.9){\tiny $t_3\sqrt{-1}$};
  \draw[dotted] (0,.2) to (1,.2);
  \draw[dotted] (0,.5) to (1,.5);
  \draw[dotted] (0,.9) to (1,.9);
  \draw[dotted] (.2,0) to (.2,1);
  \draw[dotted] (.6,0) to (.6,1);
  \node at (.2,.2){$\bullet$};
  \node at (.2,.5){$\bullet$};
  \node at (.2,.9){$\bullet$};
  \node at (.6,.5){$\bullet$};
  \node at (.6,.9){$\bullet$};
  \draw[thin, looseness=1.2, postaction={decorate}] (0,-1) to[out=70,in=-90] (.25,.1) to[out=90,in=-90] (.1,.25) to[out=90,in=90] node[below]{\tiny $\gen_{1,1}$} (.28,.25)  to[out=-90,in=68] (0,-1);
  \draw[thin, looseness=1.2, postaction={decorate}] (0,-1) to[out=65,in=-90] (.32,.3) to[out=90,in=-90] (.08,.52) to[out=90,in=90] node[below]{\tiny $\gen_{1,2}$} (.34,.52)  to[out=-90,in=62] (0,-1);
  \draw[dashed, looseness=1.2, postaction={decorate}] (0,-1) to[out=80,in=-90] (.1,0) to[out=90,in=-90] (0.02,.5) to[out=90,in=90] node[left]{\tiny $c\gen_{1,3}$} (.4,.6)  to[out=-90,in=60] (0,-1);
  \draw[thin, looseness=1.2, postaction={decorate}] (0,-1) to[out=50,in=-90] (.5,.7) to[out=90,in=-90] (.12,.9) to[out=90,in=90] node[below]{\tiny $\gen_{1,3}$} (.52,.9)  to[out=-90,in=48] (0,-1);
  \draw[thin, looseness=1.2, postaction={decorate}] (0,-1) to[out=30,in=-90] (.7,.33) to[out=90,in=-90] (.55,.55) to[out=90,in=90] node[below]{\tiny $\gen_{2,1}$} (.72,.55)  to[out=-90,in=25] (0,-1);
  \draw[thin, looseness=1.2, postaction={decorate}] (0,-1) to[out=20,in=-90] (.83,.7) to[out=90,in=-90] (.55,.9) to[out=90,in=90] node[below]{\tiny $\gen_{2,2}$} (.85,.9)
  to[out=-90,in=15] (0,-1);  
 \end{tikzpicture}
 \caption{The standard generating set of a configuration $P\subset\cR$; we have $\arr_\hor(P)=2$ and $\arr_\ver(P)=3$.
 The dashed loop represents the product $c\gen_{1,3}\in\fG(P)$.}
 \label{fig:stdgenset}
\end{figure}
We will also make use of the following products of standard generators, compare with \cite[Notation 6.7]{Bianchi:Hur1}.
\begin{nota}
 \label{nota:cgenij}
 Use the notation from Definition \ref{defn:stdgenset}. For $0\leq i\leq p+1$ and $0\leq j\leq q+2$
 we denote by $c\gen^{P,\std}_{i,j}$ the product
 \[
 c\gen^{P,\std}_{i,j}=\gen^{P,\std}_{i,0}\gen^{P,\std}_{i,1}\dots\gen^{P,\std}_{i,j-1}\in\fG(P),
 \]
 where we set $\gen^{P,\std}_{i,j}=\one\in\fG(P)$ whenever $(i,j)$ does not belong to $I(P)$.
\end{nota}
One can represent $c\gen^{P,\std}_{i,j}$ by a simple loop $\gamma\subset\CmP$ satisfying the following:
\begin{itemize}
 \item $\gamma\subset\bS_{x_{i-1},x_{i+1}}\cap \set{z\in\C\,|\, \Im(z)<y_j}$, where we use the conventions $x_{-1}=-\infty$
 and $x_{q+2}=y_{q+2}=\infty$;
 \item $\gamma$ bounds a disc in $\C$ containing the points $z_{i,0},\dots,z_{i,j-1}$.
\end{itemize}
In particular $c\gen^{P,\std}_{i,j}\in\fQext(P)$, see Definition \ref{defn:fQextP}.
For $j=0$ we have $c\gen^{P,\std}_{i,j}=\one$.

\subsection{Characteristic maps of cells}
In this subsection we introduce maps
\[
e^{\ua}\colon\Delta^p\times\Delta^q\to\Hur(\cR;\hQ_+)
\]
depending on a non-degenerate array $\ua\in\Arr_{p,q}(\Q)$.
As we will see, each map $e^{\ua}$ sends the interior of $\Delta^p\times\Delta^q$ injectively inside
$\fFarr_{p+q}\Hur(\cR;\hQ_+)$, and sends
the boundary of $\Delta^p\times\Delta^q$ inside $\Farr_{p+q-1}\Hur(\cR;\hQ_+)$.

Recall from \cite[Definitions 5.8 and 6.6]{Bianchi:Hur1} that
the bisimplicial set $\Arr(\Q)$ consists of the sets $\Arr_{p,q}(\Q)\cong\hQ^{(p+2)\times(q+2)}$ for $p,q\ge0$. An element $\ua\in\Arr_{p,q}(\Q)$ is an \emph{array} of size $(p+2)\times(q+2)$ with entries in $\hQ$.
For $p\ge1$ and $0\le i\le p$, the $i$\textsuperscript{th} horizontal face map is denoted $d^{\hor}_i\colon \Arr_{p,q}(\Q)\to\Arr_{p-1,q}(\Q)$, and for $q\ge1$ and $0\le j\le q$, the $j$\textsuperscript{th} vertical face map is denoted
$d^{\ver}_j\colon \Arr_{p,q}(\Q)\to\Arr_{p,q-1}(\Q)$. Similarly $s^{\hor}_i$ and $s^{\ver}_j$ denote the horizontal and vertical degeneracy maps. Formulas for face and degeneracy maps are given in \cite[Lemma 6.8]{Bianchi:Hur1}.
An array $\ua\in\Arr_{p,q}(\Q)$ is non-degenerate if and only if it is not in the image of any
horizontal or vertical degeneracy map of the bisimplicial set $\Arr(\Q)$.
\begin{nota}
\label{nota:Iua}
For $\ua\in\Arr(p,q)$ we let $I(\ua)\subset\set{0,\dots,p+1}\times\set{0,\dots,q+1}$ denote the set
of pairs $(i,j)$ with $a_{i,j}\neq\one$.
\end{nota}
An array $\ua$ is non-degenerate if and only if the following conditions hold, compare with \cite[Subsection 6.3]{Bianchi:Hur1}:
\begin{itemize}
 \item for all $1\leq i\leq p$ there is $0\leq j\leq q+1$ with $(i,j)\in I(\ua)$;
 \item for all $1\leq j\leq q$ there is $0\leq i\leq p+1$ with $(i,j)\in I(\ua)$.
\end{itemize}
Our next goal is to define, for any non-degenerate array $\ua\in\Arr(p,q)$, a configuration $\fc_{\ua}\in\Hur(\cR;\hQ_+)$ with $\arr_{\hor}(\fc_{\ua})=p$ and $\arr_{\ver}(\fc_{\ua})=q$: it will be the ``central'' configuration in the image of $e^{\ua}$.
\begin{nota}
\label{nota:Ppq}
 We denote by $P^{p,q}\subset\cR$ the set of complex numbers $z^{p,q}_{i,j}:=\frac{i}{p+1}+\frac{j}{q+1}\sqrt{-1}$,
 with $0\leq i\leq p+1$ and $0\leq j\leq q+1$.
 
We let $P_{\ua}\subset P^{p,q}$ be the set containing all elements
$z^{p,q}_{i,j}$ for $(i,j)\in I(\ua)$.
\end{nota}
Since $\ua$ is admissible we have
$\arr_{\hor}(P_{\ua})=p$ and $\arr_{\ver}(P_{\ua})=q$.

\begin{nota}
 \label{nota:genstdua}
We denote by $(\gen^{\ua}_{i,j})_{(i,j)\in I(\ua)}$ the standard generating set $(\gen^{P_{\ua},\std}_{i,j})$
of $\fG(P_{\ua})$ (see Definition \ref{defn:stdgenset}). For $0\leq i\leq p+1$ and $0\leq j\leq q+2$
we denote by
$c\gen^{\ua}_{i,j}$ the product $c\gen^{P_{\ua},\std}_{i,j}$ (see Notation \ref{nota:cgenij}).
\end{nota}

\begin{defn}
 \label{defn:fcua}
We define $\fc_{\ua}$ as the configuration $(P_{\ua},\psi_{\ua})\in\Hur(\cR;\hQ_+)$,
where $\psi_{\ua}$ is defined by setting $\psi_{\ua}\colon\gen^{\ua}_{i,j}\mapsto a_{i,j}$ for all $(i,j)\in I(\ua)$.
\end{defn}
Note that since $\hQ$ is complete we have an equality $\fQext(P_{\ua})=\fQext(P_{\ua})_{\psi_{\ua}}$,
so that we can extend $\psi_{\ua}$ to a map of PMQs $\psiext_{\ua}\colon\fQext(P_{\ua})\to\hQ$;
the element $c\gen^{\ua}_{i,j}\in\fQext(P_{\ua})$ is mapped to the product $a_{i,0}\dots a_{i,j-1}\in\hQ$
along $\psiext_{\ua}$.

\begin{defn}
 \label{defn:eua}
Recall Notation \ref{nota:barycentresimplex} and Definition \ref{defn:bDeltapp}, and note that for all $\us\in\Delta^p$ the pair $(\ubar^p,\us)$ belongs to $\bDelta^{p,p}$.
 For a non-degenerate array $\ua\in\Arr(p,q)$ we define a continuous map
 $e^{\ua}\colon\Delta^p\times\Delta^q\to\Hur(\cR;\hQ_+)$ by the formula
 \[
  e^{\ua}(\us;\ut)=\cH^{p,q}_*\pa{\fc_{\ua}\,;\,\ubar^p\,,\,\us\,;\,\ubar^q\,,\,\ut}.
 \]
\end{defn}

\begin{lem}
 \label{lem:euafiltered}
 Let $\ua\in\Arr_{p,q}$ be non-degenerate; then $e^{\ua}$ satisfies the following:
 \begin{itemize}
 \item $e^{\ua}$ sends the interior
 of $\Delta^p\times\Delta^q$ injectively
 inside $\fFarr_{p+q}\Hur(\cR;\hQ_+)$;
 \item $e^{\ua}$ sends $\del(\Delta^p\times\Delta^q)$ inside $\Farr_{p+q-1}\Hur(\cR;\hQ_+)$.
 \end{itemize}
\end{lem}
\begin{proof}
 Let $(\us,\ut)\in\Delta^p\times\Delta^q$, and let $\fc=(P,\psi):=e^{\ua}(\us,\ut)$.
 Then $P$ is the image of $P_{\ua}$ along the map
 $\cH^{p,q}(-;\ubar^p,\us;\ubar^q,\ut)\colon\C\to\C$.
 The latter map sends $z^{p,q}_{i,j}\mapsto s_i+t_j\sqrt{-1}$ for all $0\leq i\leq p+1$
 and $0\leq j\leq q+1$, in particular for $(i,j)\in I(\ua)$. It follows
 that $\Re(P)\setminus\set{0,1}=\set{s_1,\dots,s_p}\setminus\set{0,1}$ consists of at most
 $p$ points, and $\Im(P)\setminus\set{0,1}=\set{t_1,\dots,t_q}\setminus\set{0,1}$ consists
 of at most $q$ points. More precisely, using that $\ua$ is non-degenerate,
 we have that $|\Re(P)\setminus\set{0,1}|=p$
 if $0<s_1<\dots<s_p<1$, i.e. if $\us$ is in the interior of $\Delta^p$, whereas
 $|\Re(P)\setminus\set{0,1}|<p$ if $\us\in\del\Delta^p$. Similarly
 $|\Im(P)\setminus\set{0,1}|=q$ if $\ut$ is in the interior of $\Delta^q$,
 and $|\Im(P)\setminus\set{0,1}|<q$ if $\ut\in\del\Delta^q$.
 
 Hence $|\arr|(\fc)\leq p+q$; equality holds if and only if
 $(\us,\ut)\notin\del(\Delta^p\times\Delta^q)$.
\end{proof}

\begin{lem}
 \label{lem:unioneuasurjective}
 Let $\nu\geq 0$ and let $\fc\in\fFarr_{\nu}\Hur(\cR;\hQ_+)$; then there is precisely
 one couple of indices $p,q\geq0$ with $p+q=\nu$, and precisely one admissible array $\ua\in\Arr(p,q)$,
 such that $\fc$ is in the image of $e^{\ua}$.
\end{lem}
\begin{proof}
We first show the existence of $p,q$ and $\ua$ with the required properties.
Use Notation \ref{nota:fc} for $\fc$, and let $\Re(P)\cup\set{0,1}=\set{0<s_1<\dots<s_p<1}$
have $p+2$ elements, for some $p\geq0$; similarly let $\Im(P)\cup\set{0,1}=\set{0<t_1<\dots<t_q<1}$ have $q+2$ elements, for some $q\geq0$. By the hypothesis
that $\fc\in\fFarr_{\nu}\Hur(\cR;\hQ_+)$ we have the equality $p+q=|\arr|(\fc)=\nu$.

Define an array $\ua$ of size $(p+2)\times(q+2)$
by setting $a_{i,j}=\psi(\gen^{P,\std}_{i,j})\in\hQ$ for all $(i,j)\in I(P)$, and $a_{i,j}=\one$
for all $(i,j)\in \set{0,\dots,p+1}\times\set{0,\dots,q+1}\setminus I(P)$
(see Definition \ref{defn:stdgenset}). The hypothesis that $\fc$ lies in $\Hur(\cR;\hQ_+)\subseteq\Hur(\cR;\hQ)$ ensures
that $\ua$ is a non-degenerate array in $\Arr_{p,q}(\Q)$. Moreover we have $I(\ua)=I(P)$.

We claim that $e^{\ua}$ sends $(\us,\ut)\in\Delta^p\times\Delta^q$ to
$\fc$. Indeed
$\cH^{p,q}\pa{-;\ubar^p,\us;\ubar^q,\ut}$ is a homeomorphism of $\C$, sending
$P_{\ua}$ bijectively to $P$, and sending the standard generating set $(\gen^{\ua}_{i,j})_{(i,j)\in I(\ua)}$
of $\fG(P_{\ua})$ to the standard generating set $(\gen^{P,\std}_{i,j})_{(i,j)\in I(P)}$ of $\fG(P)$.
This proves the existence of $p$, $q$ and $\ua$ as desired.

For uniqueness, suppose that we are given two integers $p',q'\geq0$ and a non-degenerate array $\ua'\in\Arr(p,q)$,
such that $p'+q'=\nu$ and such that there is a point $(\us',\ut')\in\Delta^{p'}\times\Delta^{q'}$
with $e^{\ua'}\colon(\us',\ut')\mapsto\fc$. Then by Lemma \ref{lem:euafiltered}
we have that $(\us',\ut')$ lies in the interior of $\Delta^{p'}\times\Delta^{q'}$,
since $\fc$ lies in $\fFarr_{\nu}\Hur(\cR;\hQ_+)$. Again the map
$\cH^{p',q'}\pa{-;\ubar^{p'},\us';\ubar^{q'},\ut'}$
is a homeomorphism of $\C$ mapping the set
$P_{\ua'}$ bijectively to the set $P$ by the formula $z_{i,j}^{p',q'}\mapsto s'_i+\sqrt{-1}t'_j$.

It follows that $\set{0<s'_1<\dots<s'_{p'}<1}$ is equal to $\Re(P)\cup\set{0,1}$,
and similarly $\set{0<t'_1<\dots<t'_{q'}<1}$ is equal to $\Im(P)\cup\set{0,1}$; in particular,
comparing with the construction above, we have $p=p'$, $q=q'$, $\us=\us'$, $\ut=\ut'$ and $I(\ua')=I(P)$.

Since $\cH^{p',q'}\pa{-;\ubar^{p'},\us';\ubar^{q'},\ut'}$ gives a bijection between the standard
generating sets $(\gen^{\ua'}_{i,j})_{(i,j)\in I(\ua')}$ and
$(\gen^{P,\std}_{i,j})_{(i,j)\in I(P)}$, it also follows that
\[
a'_{i,j}=\psi_{\ua'}(\gen^{\ua'}_{i,j})=\psi(\gen_{i,j})=a_{i,j}
\]
for all $(i,j)\in I(P)$, where $\psi_{\ua'}$ is the monodromy of the configuration $\fc_{\ua'}$,
see Definition \ref{defn:fcua}; hence $\ua'=\ua$.
\end{proof}

\subsection{Face restrictions and the bijection \texorpdfstring{$\upsilon$}{upsilon}}
In the following two propositions we analyse the restriction of $e^{\ua}$ to a face of
$\Delta^p\times\Delta^q$, and thus establish a link between
the bisimplicial set $\Arr(\Q)$ and the cell stratification on $\Hur(\cR;\hQ_+)$.

\begin{nota}
 \label{nota:facebisimplex}
For $0\leq i\leq p$ we denote by $d^{\hor}_i\Delta^p\times\Delta^q$ the face
$(d_i\Delta^p)\times\Delta^q\subset\Delta^p\times\Delta^q$; for $0\leq j\leq q$
we denote by $d^{\ver}_j\Delta^p\times\Delta^q$ the face
$\Delta^p\times(d_j\Delta^q)\subset\Delta^p\times\Delta^q$.
\end{nota}
Each face $d_i\Delta^p\subset\Delta^p$ can be identified with a the simplex $\Delta^{p-1}$
by using either the coordinates $(s_1,\dots,s_{i-1},s_{i+1},\dots,s_p)$,
for $i\neq 0$, or the coordinates $(s_1,\dots,s_i,s_{i+2},\dots,s_p)$, for $i\neq p$;
for $1\le i\le p-1$ the two choices give rise to the \emph{same} identification.
Similarly, there are canonical identifications of
$d^{\hor}_i\Delta^p\times\Delta^q$ with $\Delta^{p-1}\times\Delta^q$, and of
$d^{\ver}_j\Delta^p\times\Delta^q$ with $\Delta^p\times\Delta^{q-1}$.

\begin{prop}
 \label{prop:facerestrictionh}
 Let $\ua$ be a non degenerate array in $\Arr_{p,q}(\Q)$, for some $p\geq 1$ and $q\geq0$,
 and let $0\leq i\leq p$. Then the restriction of
 $e^{\ua}\colon\Delta^p\times\Delta^q\to\Hur(\cR;\hQ_+)$ to the face
 $d^{\hor}_i\Delta^p\times\Delta^q\cong\Delta^{p-1}\times\Delta^q$ is equal
 to the map $e^{\ua'}$, where $\ua'=d^{\hor}_i\ua$.
\end{prop}
\begin{prop}
 \label{prop:facerestrictionv}
 Let $\ua$ be a non degenerate array in $\Arr_{p,q}(\Q)$, for some $p\geq 0$ and $q\geq1$,
 and let $0\leq j\leq q$. Then the restriction of
 $e^{\ua}\colon\Delta^p\times\Delta^q\to\Hur(\cR;\hQ_+)$ to the face
 $d^{\ver}_j\Delta^p\times\Delta^q\cong\Delta^p\times\Delta^{q-1}$ is equal
 to the map $e^{\ua'}$, where $\ua'=d^{\ver}_j\ua$.
\end{prop}
The expressions ``$d^{\hor}_i\ua$'' and ``$d^{\ver}_j\ua$'' refer to the bisimplicial set $\Arr(\Q)$.
The proof of Propositions \ref{prop:facerestrictionh} and \ref{prop:facerestrictionv}
is in Subsections \ref{subsec:facerestrictionh} and \ref{subsec:facerestrictionv}
of the appendix.

Recall from \cite[Lemma 6.10]{Bianchi:Hur1} that there is a semi-bisimplicial set
$\Arr^{\ndeg}(\Q)$ containing all non-degenerate arrays of $\Arr(\Q)$, and with vertical and horizontal
face maps given by the restriction of those of $\Arr(\Q)$.
Consider the geometric realisation $|\!|\Arr^{\ndeg}(\Q)|\!|$ of the semi-bisimplicial complex
$\Arr^{\ndeg}(\Q)$, and note that there is a homeomorphism $|\!|\Arr^{\ndeg}(\Q)|\!|\cong|\Arr(\Q)|$.

By Propositions \ref{prop:facerestrictionh} and \ref{prop:facerestrictionv}
we obtain a continuous map
\[
\upsilon\colon|\Arr(\Q)|\to\Hur(\cR;\hQ_+).
\]
Lemmas \ref{lem:euafiltered} and \ref{lem:unioneuasurjective} imply that $\upsilon$ is bijective; if we consider
on $|\Arr(\Q)|$ the skeletal filtration and on $\Hur(\cR;\hQ_+)$ the filtration $\Farr_{\bullet}$,
then Lemma \ref{lem:euafiltered} also implies that $\upsilon$ is a map of filtered spaces.

We say that an entry $a_{i,j}$ of an array $\ua\in\Arr_{p,q}(\Q)$ 
is in \emph{boundary position} if $i\in\set{0,p+1}$ or
$j\in\set{0,q+1}$ (or both conditions hold). Recall
from \cite[Definition 6.11 and Lemma 6.12]{Bianchi:Hur1}
that we have a sub-bisimplicial set $\NAdm(\Q)\subset\Arr(\Q)$
of non-admissible arrays: an array $\ua\in\Arr(p,q)$ is non-admissible
if either of the following requirements is satisfied:
 \begin{itemize}
  \item there exists an entry $a_{i,j}$ lying in $\hQ\setminus\Q$;
  \item there exists an entry $a_{i,j}\neq \one$ in boundary position.
 \end{itemize}

\begin{lem}
 \label{lem:upsilonrestriction}
 The map $\upsilon$ restricts to continuous bijections
 \[
 \upsilon \colon  |\NAdm(\Q)|\to\Hur(\cR;\hQ_+)\setminus\Hur(\mcR;\Q_+);
 \]
 \[
 \upsilon \colon  \Hur^{\Delta}(\Q)=|\Arr(\Q)|\setminus|\NAdm(\Q)|\to\Hur(\mcR;\Q_+).
 \]
\end{lem}
\begin{proof}
Let $\fc\in\Hur(\cR;\hQ_+)$, and use Notation \ref{nota:fc}. In the proof of Lemma \ref{lem:unioneuasurjective} we have given a construction, depending on $\fc$,
of a couple of numbers
$p,q\geq0$, a point $(\us,\ut)$ in the interior of $\Delta^p\times\Delta^q$
and a non-degenerate array $\ua\in\Arr(p,q)$ such that $\fc=e^{\ua}(\us,\ut)$.
The data $(\ua;\us,\ut)$ represent a point in $|\Arr(\Q)|$, which is precisely
the preimage $\upsilon^{-1}(\fc)$; we have $\upsilon^{-1}(\fc)\in|\NAdm(\Q)|$ if and only
if $\ua$ is non-admissible.

The array $\ua$ was constructed by considering the standard generating set
of $\fG(P)$, and by setting $a_{i,j}=\psi(\gen^{P,\std}_{i,j})$
for $(i,j)\in I(P)$, and $a_{i,j}=\one$ otherwise. It follows that $\ua$
has all entries in $\Q$ if and only if $\psi\colon\fQ(P)\to\hQ$ has image in $\Q$, that is,
$\fc\in\Hur(\cR;\Q_+)$; and $\ua$ has all entries in boundary position equal to $\one$ if and only if
$P\subset\mcR$, that is, $\fc\in\Hur(\mcR;\hQ_+)$.
We have therefore
\[
\fc\in \Hur(\mcR;\Q_+)=\Hur(\cR;\Q_+)\cap\Hur(\mcR;\hQ_+)\subset\Hur(\cR;\hQ_+)
\]
if and only if $\ua$ is admissible.
\end{proof}

\section{Locally finite and Poincar\'e PMQs}
\label{sec:Poincare}
We consider the Hurwitz-Ran spaces $\Hur(\mcR;\Q)$ in the special cases of
a locally finite PMQ and of a Poincar\'e PMQ $\Q$. Recall that a PMQ $\Q$ is Poincar\'e if
each connected component of $\Hur^{\Delta}(\Q)$ is a topological manifold; this condition implies
that $\hQ$ is endowed with an intrinsic norm $h\colon\hQ\to\N$ such that $\Hur^{\Delta}(\Q)(a)$
is an orientable manifold of dimension $2h(a)$ for all $a\in\hQ$, see \cite[Proposition 6.23]{Bianchi:Hur1}. A Poincar\'e PMQ is always locally finite, and a locally finite PMQ is always augmented.

\subsection{Locally finite PMQs}
In this subsection we prove the following theorem.
\begin{thm}
 \label{thm:upsilonhomeo}
 Let $\Q$ be a locally finite PMQ. Then the bijection $\upsilon\colon \Hur^{\Delta}(\Q)\to\Hur(\mcR;\Q_+)$
 is a homeomorphism.
\end{thm}
 We first note that both $|\Arr(\Q)|$ and $\Hur(\cR;\hQ_+)$ decompose as topologial disjoint unions of subspaces
\[
 |\Arr|=\coprod_{a\in\hQ}|\Arr(\Q)(a)|,\quad\quad\quad \Hur(\cR;\hQ_+)=\coprod_{a\in\hQ}\Hur(\cR;\hQ_+)_a.
\]
Recall from \cite[Definition 6.1]{Bianchi:Hur1} the category $\hQ\borel\hQ$ and the notion of $\hQ$-crossed object in a category.
The space $|\Arr(\Q)(a)|$ is the geometric realisation of the bisimplicial set
$\Arr(\Q)(a)$, which is the value at $a\in\hQ\borel\hQ$ of the $\hQ$-crossed bisimplicial set $\Arr(\Q)$:
concretely, $\Arr_{p,q}(\Q)(a)$ contains all arrays $\ua$ satisfying $\prod_{i=0}^{p+1}\pa{\prod_{j=0}^{q+1}a_{i,j}}=a\in\hQ$.
For the space $\Hur(\cR;\hQ_+)_a$, see Notation \ref{nota:Hurwithtotmon}.

The map $\upsilon$ restricts for all $a\in\hQ$ to a bijection
$\upsilon_a\colon|\Arr(\Q)(a)|\to\Hur(\cR;\hQ_+)_a$.
Consider first an element $a\in \Q\subset\hQ$. The hypothesis that $\Q$ is locally finite implies
that $\Arr(\Q)(a)$ is a bisimplicial complex with finitely many non-degenerate arrays:
hence the geometric realisation $|\Arr(\Q)(a)|$ is compact. The bijection $\upsilon_a$ has thus
a compact space as source and a Hausdorff space as target, and is therefore a homeomorphism. Restricting
to $\Hur^{\Delta}(\Q)(a)$ and $\Hur(\mcR;\Q_+)_a$, we have a homeomorphism
$\upsilon_a\colon\Hur^{\Delta}(\Q)(a)\to\Hur(\mcR;\Q_+)_a$.

Consider now a generic element $a\in\hQ$, let $\fc\in\Hur(\mcR;\Q_+)_a$, use Notation
\ref{nota:fc} and the notation of Subsection \ref{subsec:stdgenset}.
Let $\uU$ be an adapted covering of $P$, and assume that 
for all $(i,j)\in I(P)$ the component $U_{i,j}\subset\uU$ containing
$z_{i,j}$ satisfies the following properties:
\begin{itemize}
 \item $\Re(U_{i,j})$ intersects $\Re(P)\cup\set{0,1}$ only in $\Re(z_{i,j})$;
 \item $\Im(U_{i,j})$ intersects $\Im(P)\cup\set{0,1}$ only in $\Im(z_{i,j})$.
\end{itemize}
Let $\bar{\uU}$ denote the union
$\bigcup_{(i,j)\in I(P)}\bar U_{i,j}$, i.e. the closure of $\uU$, and note that $\bar{\uU}$ is compact.
The normal neighbourhood $\fU(\fc;\uU)$ can be regarded as an open subspace of
$\Hur(\bar{\uU};\Q_+)_a$, which by the argument of
Proposition \ref{prop:productneighbourhood} is homeomorphic to a product
of spaces
\[
 \Hur(\bar{\uU};\Q_+)_a\cong \prod_{(i,j)\in I(P)}\Hur(\bar U_{i,j};\Q_+)_{a_{i,j}},
\]
for suitable elements $a_{i,j}\in\Q$. Note that if Proposition \ref{prop:productneighbourhood}
is applied using the arcs $(\zeta_{i,j})_{(i,j)\in I(P)}$ yielding the standard generating
set of $\fG(P)$, then the elements $a_{i,j}$ are precisely the entries different from $\one$
of the array $\ua$ describing the cell of $\Hur^{\Delta}(\Q)(a)$ containing $\upsilon^{-1}(\fc)$.

The previous analysis shows that each factor $\Hur(\bar U_{i,j};\Q_+)_{a_{i,j}}$ is compact;
it follows that $\Hur(\bar{\uU};\Q_+)_a$ is compact
i.e. $\Hur(\mcR;\Q_+)_a$ is locally compact.

Consider $\Hur(\bar{\uU};\Q_+)_a$ as a subspace of $\Hur(\cR;\hQ_+)_a$: the hypothesis that $\Q$ is locally
finite implies that the preimage $\upsilon_a^{-1}(\Hur(\bar{\uU};\Q_+)_a)$ intersects
only finitely many cells in the cell decomposition of $|\Arr(\Q)(a)|$. Hence
$\upsilon_a^{-1}(\Hur(\bar{\uU};\Q_+)_a)$ is compact,
being a closed subset of a finite cell sub-complex of $|\Arr(\Q)(a)|$.
Since $\upsilon_a^{-1}(\Hur(\bar{\uU};\Q_+)_a)$ contains the open neighbourhood
$\upsilon_a^{-1}(\fU(\fc;\uU))$ of $\upsilon_a^{-1}(\fc)$, we obtain that
$\Hur^{\Delta}(\Q)(a)$ is also locally compact.

We conclude that $\upsilon_a\colon\Hur^{\Delta}(\Q)(a)\to\Hur(\mcR;\Q_+)_a$ is a proper continuous
bijection between locally compact spaces, hence it is a homeomorphism.

\subsection{A counterexample to Theorem \ref{thm:upsilonhomeo} for non-locally finite PMQs}
For an augmented but not locally finite PMQ $\Q$, the bijection
$\upsilon$ may not restrict to a homeomorphism $\Hur^{\Delta}(\Q)\to\Hur(\mcR;\Q_+)$:
see the following example.
\begin{ex}
\label{ex:nothomeo}
 Let $\hQ$ be the completion of $\Q:=\bF^2$, the free group on two generators $\gen_1,\gen_2$ endowed with \emph{trivial} multiplication, as in
\cite[Example 4.13]{Bianchi:Hur1}.

Let $\fc=(P,\psi)\in\Hur(\mcR;\hQ_+)$ be a configuration supported on $P:=\set{\zcentre}=\set{\frac{1+\sqrt{-1}}{2}}$
 with $\psi$ defined by sending the unique element $[\gamma]\in\fQ(P)\setminus\set{\one}$
 to $w=\hat{\gen_1}\hat{\gen_2}\in\hQ$. 
 For all $0<\epsilon\le 1/2$ we have an adapted covering of $P$ of the form
 $U_\epsilon=\set{z\in\C\,\colon\,\abs{z-\zcentre}<\epsilon}$; the associated normal neighbouroods
 $\fU(\fc;U_\epsilon)$ form a fundamental system of neighbouroods of $\fc\in\Hur(\cR;\hQ_+)$.
 
 For $\epsilon>0$, denote by $P_\epsilon$ the set of two points $\set{\zcentre\pm \epsilon/2}$, and note that $P_\epsilon\subset U_\epsilon$. For every decomposition $w=a\cdot b$ with respect to $\hQ_1$ we can define
 a configuration $\fc_{\epsilon,a,b}=(P_\epsilon,\psi_{a,b})\in \fU(\fc;U_{\epsilon})$, where $\psi_{a,b}$
 sends the standard generators $\gen^{P_\epsilon,\std}_{1,1}$ and $\gen^{P_\epsilon,\std}_{2,1}$
 to $a$ and $b$ respectively. Using that $w$ has infinitely many non-trivial decompositions with respect to $\hQ_1$,
 we obtain for all
 $0<\epsilon\le 1/2$ an infinite family of configurations $\fc_{\epsilon,a,b}$ supported
 on the same set $P_\epsilon$ and contained in an arbitrary small normal neighbourood $\fU(\fc;U_\epsilon)$.
 
 Note that the configurations $\upsilon^{-1}(\fc_{\epsilon,a,b})$, for fixed $\epsilon$ and varying $a,b$ with $w=ab$,
 belong to different open cells of the cell stratification of $\Hur^{\Delta}(\hQ)$.
 By a diagonal argument one can find a neighbourood
 of $\upsilon^{-1}(\fc)$ in $\Hur^{\Delta}(\hQ)$ containing, for all $\epsilon>0$, only finitely many points
 $(\ua;\us,\ut)$ such that $\upsilon(\ua;\us,\ut)$ has support precisely $P_{\epsilon}$.
 Thus $\upsilon\colon\Hur^{\Delta}(\hQ)\to\Hur(\mcR;\hQ_+)$ is not a homeomorphism.
\end{ex}
In light of Example \ref{ex:nothomeo} one could argue that the topology
on $\Hur(\cR;\hQ_+)$, described in Section \ref{sec:defnHurPMQ},
is not the \emph{correct} topology to consider on Hurwitz-Ran spaces, and that one should rather consider
the CW topology induced by $|\Arr(\Q)|$ along the bijection $\upsilon$.
This would indeed simplify
the discussion in this section, by making Theorem \ref{thm:upsilonhomeo} tautological.
Nevertheless it would become much more elaborate to replace the topology
on $\Hur(\fT;\Q,G)$, for a generic nice couple $\fT$ and a generic PMQ-group pair $(\Q,G)$,
with the topology of a difference of CW complexes.
Moreover, the functoriality of Hurwitz-Ran spaces with respect to morphisms of nice couplex,
discussed Section \ref{sec:functoriality}, would also become much more complicated to prove.

\subsection{Poincar\'e PMQs}
In this subsection we prove the following theorem.
\begin{thm}
 \label{thm:Poincare}
 Let $\Q$ be a locally finite PMQ. If for all $a\in\Q\subset\hQ$ the space $\Hur^{\Delta}(\Q)(a)$
 is a topological manifold of some dimension, then $\Q$ is Poincar\'e.
\end{thm}
\begin{proof}
 By Theorem \ref{thm:upsilonhomeo} the simplicial Hurwitz space $\Hur^{\Delta}(\Q)$ is homeomorphic
 to $\Hur(\mcR;\Q_+)$, so it suffices to prove that for all $b\in\hQ$ the space $\Hur(\mcR;\Q_+)_b$
 is a topological manifold. In the following we fix $b\in\hQ$.
 
 Let $\fc\in\Hur(\mcR;\Q_+)_b$, use Notations \ref{nota:fc} and \ref{nota:uUcovering}
 and let $\uU$ be an adapted covering of $P$. By Proposition \ref{prop:productneighbourhood}
 the normal neighbourhood $\fU(\fc,\uU)\subset\Hur(\mcR;\Q)$ is homeomorphic to a product
 of normal neighbourhoods $\prod_{i=1}^k\fU(\fc'_i,U_i)$ for suitable configurations $\fc'_i\in\Hur(\mcR;\Q)$; the argument of the proof
 of Proposition \ref{prop:productneighbourhood} shows in fact that $\fc'_i$ is a configuration
 in $\Hur(\mcR;\Q_+)$ supported on the single point $z_i$, and that there is a restricted homeomorphism
 \[
  \fU(\fc,\uU)\cap\Hur(\mcR;\Q_+)_b\cong\prod_{i=1}^k\pa{\fU(\fc'_i,\uU)\cap \Hur(\mcR;\Q_+)_{b_i}},
 \]
with $b_i=\omega(\fc'_i)$. A priori $b_i\in\hQ$; since $\fc'_i$ is supported on one point we
have $b_i\in\Q$.

The hypothesis on $\Q$ ensures that each space $\Hur(\mcR;\Q_+)_{b_i}$ is a topological manifold;
thus also each open subset $\fU(\fc'_i,\uU)\cap \Hur(\mcR;\Q_+)_{b_i}$ is a topological manifold,
and therefore $\fU(\fc,\uU)\cap\Hur(\mcR;\Q_+)_b$ is a topological manifold. This shows that
each configuration in $\Hur(\mcR;\Q_+)_b$ has a neighbourhood which is a topological manifold,
and thus the space $\Hur(\mcR;\Q_+)_b$, which is Hausdorff, is a topological manifold.
\end{proof}
The proof of Theorem \ref{thm:Poincare} can be generalised to homology manifolds as follows.
\begin{defn}
 \label{defn:RPoincare}
Let $R$ be a commutative ring. A locally finite PMQ $\Q$ is \emph{$R$-Poincar\'e}
if for all $a\in\hQ$
the space $\Hur(\mcR;\Q)$ is a $R$-homology manifold of some dimension, i.e. for all
$\fc\in\Hur(\mcR;\Q_+)$ the local homology
\[
\tilde{H}_*\pa{\Hur(\mcR;\Q_+)_a\ ,\ \Hur(\mcR;\Q_+)_a\setminus\set{\fc}\ ;\ R}
\]
is isomorphic to $R$ in a single degree, and vanishes in all other degrees.
\end{defn}

\begin{lem}
 \label{lem:cone}
Let $\Q$ be a locally finite PMQ, let $a\in\Q_+$ and let $z_0\in\mcR$;
then the space $\Hur(\cR;\Q_+)_a$ is homeomorphic to the cone over the space
\[
\del\Hur(\cR;\Q_+)_a:=\Hur(\cR;\Q_+)_a\setminus\Hur(\mcR;\Q_+)_a,
\]
with vertex the unique configuration $\fc_{z_0,a}\in\Hur(\cR;\Q_+)_a$
supported on $z_0$.
\end{lem}
\begin{proof}
Let $\hQ$ be the completion of $\Q$, and note that $\Hur(\cR;\Q_+)_a$ is homeomorphic to
$\Hur(\cR;\Q_+)_a$. Without loss of generality, we may assume that $\Q$ is already complete.
Note that the space $\del\Hur(\cR;\Q_+)_a$ is a closed subspace of 
$\Hur(\cR;\Q_+)_a$, containing all configurations
$\fc\in\Hur(\cR;\Q_+)_a$ whose support intersects $\del\cR$.

Fix a map $\cH^{z_0}\colon\C\times[0,1]\to\C$ satisfying the following properties:
\begin{itemize}
 \item $\cH^{z_0}(z,s)=sz_0+(1-s)z$ for all $z\in\cR$ and $0\le s\le 1$;
 \item $\cH^{z_0}(-,s)$ is a lax morphism of nice couples $(\cR,\emptyset)\to(\cR,\emptyset)$
 for all $0\le s\le 1$.
\end{itemize}
By Proposition \ref{prop:functorialityhomotopyLNC} we obtain a continuous map
$\cH^{z_0}_*\colon \Hur(\cR;\Q_+)_a\times[0,1]\to \Hur(\cR;\Q_+)_a$, that we can restrict
to a map
\[
 \del\cH^{z_0}_*\colon \del\Hur(\cR;\Q_+)_a\times[0,1]\to \Hur(\cR;\Q_+)_a.
\]
The map $\del\cH^{z_0}_*$ sends the subspace $\del\Hur(\cR;\Q_+)_a\times\set{1}$ constantly
to the configuration $\fc_{z_0,a}$. The quotient map
\[
\del\cH^{z_0}_*\colon \del\Hur(\cR;\Q_+)_a\times[0,1]\ /\ \del\Hur(\cR;\Q_+)_a\times\set{1}\to \Hur(\cR;\Q_+)_a
\]
is a continuous bijection between compact Hausdorff spaces, hence it is a homeomorphism.
\end{proof}

In particular for all $z_0\in\mcR$ we have an isomorphism of homology groups,
where $R$-coefficients are understood:
\[
\tilde{H}_*\!\pa{\Hur(\mcR;\Q_+)_a\,,\, \Hur(\mcR;\Q_+)_a\setminus\set{\fc_{z_0,a}}}\!\cong\!
\tilde{H}_*\!\pa{|\Arr(\Q)(a)|\, ,\, |\NAdm(\Q)(a)|}.
\]
The argument used in the proof of Theorem \ref{thm:Poincare}, together with
the K\"unneth isomorphism, implies directly the following theorem.
\begin{thm}
 \label{thm:RPoincare}
 Let $R$ be a commutative ring and let $\Q$ be a locally finite PMQ.
 Suppose that for all $a\in\Q$ the relative homology groups
 \[
  \tilde H_*\pa{|\Arr(\Q)(a)|,|\NAdm(\Q)(a)|;R}
 \]
 are supported in a single degree, with corresponding homology group equal to $R$. 
 Then $\Q$ is $R$-Poincar\'e.
\end{thm}
Theorem \ref{thm:RPoincare} is the non-trivial arrow of an ``if and only if'' statement:
if $\Q$ is $R$-Poincar\'e, then in particular for all $a\in\Q$ the space
$\Hur^{\Delta}(\Q)(a)\cong\Hur(\mcR;\Q_+)_a$ is a $R$-homology manifold;
since this space is contractible (see Proposition 
\ref{prop:totmoncontractibility}), by Poincar\'e-Lefschetz duality
the relative homology groups
\[
\tilde H_*\pa{|\Arr(\Q)(a)|,|\NAdm(\Q)(a)|;R}
\]
are supported in one degree, namely the $R$-homology dimension of $\Hur^{\Delta}(\Q)(a)$,
with corresponding group isomorphic to $H^0(\Hur^{\Delta}(\Q)(a);R)\cong R$.

The proofs of \cite[Proposition 6.23]{Bianchi:Hur1} and \cite[Proposition 6.24]{Bianchi:Hur1}
generalise to give the following Proposition.
\begin{prop}
 Let $R$ be a commutative ring and let $\Q$ be a $R$-Poincar\'e PMQ. Then $\Q$ is
 coconnected and admits an intrinsic norm $h\colon\Q\to \N$.
 
 If we denote by $h\colon\hQ\to\N$ also the extension of the intrinsic norm to the completion $\hQ$
 of $\Q$, then for all $a\in\hQ$ the space $\Hur(\mcR;\Q_+)_a$ is a $R$-homology manifold
 of dimension $2h(a)$.
\end{prop}
\begin{proof}
Recall that for a locally finite PMQ $\Q$ the candidate for the intrinsic norm $h\colon\Q\to\N$
is the function of sets associating with $a\in\Q$ the maximum $r\ge0$ for which there exist
a decomposition $a=a_1\dots a_r$ with $a_i\in\Q_+$.

If $a\in\Q_+$ is irreducible, then $\Hur(\mcR;\Q_+)_a$ is homeomorphic to $\mcR$, which is a $R$-homology
manifold of dimension $2=2h(a)$; more generally, if $a=a_1\dots a_r$ is a decomposition witnessing
the equality $h(a)=r$, then we can fix a configuration $\fc=(P,\psi)\in\Hur(\mcR;\Q_+)_a$
supported on a subset $P\subset\mcR$ of precisely $r$ points. By Proposition \ref{prop:productneighbourhood}
a normal neighbourhood of $\fc$ is homeomorphic to an open subset of $(\mcR)^r$; it follows
that the $R$-homology dimension of $\Hur(\mcR;\Q_+)_a$, computed around $\fc$, is equal to $2h(a)$.

The same argument, applied to any decomposition $a=bc$ in $\Q$, shows that the $R$-homology
dimension of $\Hur(\mcR;\Q_+)_a$ is equal to the sum of the $R$-homology dimensions
of $\Hur(\mcR;\Q_+)_b$ and $\Hur(\mcR;\Q_+)_c$: in fact we can find an open set of $\Hur(\mcR;\Q_+)_a$
homeomorphic to the product of two open sets of $\Hur(\mcR;\Q_+)_b$ and $\Hur(\mcR;\Q_+)_c$
respectively. It follows that $h(a)=h(b)+h(c)$, i.e. $h$ is an intrinsic norm. The $R$-homology
dimension of $\Hur(\mcR;\Q_+)_a$ can be computed to be $h(a)$ for a generic $a\in\hQ$ by
the same argument, after fixing a decomposition of $a$ as product of elements of $\Q$.

This shows that $\Q$ admits an intrinsic norm, and in particular it is maximally decomposable.
To prove that $\Q$ is coconnected, let $\Q_{\le1}\subset\Q$ be the sub-PMQ containing elements of norm $\le 1$; the inclusion of augmented PMQs
$\Q_{\le1}\subset\Q$ induces, for all $a\in\hQ$ a surjective, bisimplicial map 
$|\Arr(\Q_{\le1})(a)|\cong|\Arr(\Q)(a)|$, which is a bijection when restricted to bisimplices of dimension $2h(a)$
and $2h(a)-1$ (see the proof of \cite[Proposition 6.23]{Bianchi:Hur1}). Here we 
write $|\Arr(\Q_{\le1})(a)|$ for the disjoint union $\coprod_{a'}|\Arr(\Q_{\le1})(a')|$,
where $a'$ ranges among all elements of $\widehat{\Q_{\le1}}$ which are sent to $a\in\hQ$ along the (surjective, but a priori not bijective) map $\widehat{\Q_{\le1}}\to\hQ$.

It follows that the induced map
\[
 H_{2h(a)}\pa{|\Arr(\Q_{\le1})(a)|,|\NAdm(\Q_{\le1})(a)|}\to
 H_{2h(a)}\pa{|\Arr(\Q)(a)|,|\NAdm(\Q)(a)|},
\]
with $R$-coefficients for homology understood, is an isomorphism of $R$-modules. The rank of the second $R$-module is 1, because
$H_{2h(a)}\pa{|\Arr(\Q)(a)|,|\NAdm(\Q)(a)|;R}\cong H_0(\Hur^{\Delta}(\Q)(a))$,
and the space $\Hur^{\Delta}(\Q)(a)$ is contractible.
Similarly the rank of the first $R$-module is the number of connected components
of $\Hur^{\Delta}(\Q_{\le1})(a)$. It follows that there is exactly one element
$a'\in\widehat{\Q_{\le1}}$ which is mapped to $a$ along the map $\widehat{\Q_{\le1}}\to\hQ$, and that 
$\Hur^{\Delta}(\Q_{\le1})(a)$ is connected. This shows that $\Q$ is coconnected. 
\end{proof}

\appendix
\section{Deferred proofs}
\label{sec:deferred}
\subsection{Proof of Proposition \ref{prop:fQgeneratesfQext}}
\label{subsec:fQgeneratesfQext}
 Let $\fg\in \fQext(P)$, and assume first that $\fg=[\gamma]$ is
 represented by a \emph{simple} loop $\gamma$ in $\CmP$. The loop
 $\gamma$ is freely isotopic to a simple closed curve in $\C\setminus\cY$.
 In particular $\gamma$ bounds a disc $D$ in $\C$ which intersects $P$ only in points of $P\setminus \cY$;
 without loss of generality, assume that $D\cap P$ consists of the points $z_1,\dots,z_r$ for some $1\leq r\leq l$.
 
 We can then find an admissible generating set $\gen_1,\dots,\gen_k$ of $\fG(P)$ such that $\fg=\gen_1\cdot\dots\cdot\gen_r\in\fQext(P)$
 (see Definition \ref{defn:admgenset}): for this it suffices to choose the arcs $\arc_1,\dots,\arc_r$
 inside $D$ in a convenient way.
 This gives a decomposition $(\gen_1,\dots,\gen_r)$ of $\fg$ with respect to $\fQ(P)$ as required.
   
 If $\fg\in\fQext(P)$ is not represented by a simple loop,
 we can still find a conjugate $\fg'$ of $\fg$ in $\fQext(P)\subset\fG(P)$, with $\fg'=[\gamma']$ represented by a simple loop
 $\gamma'$. By the previous argument we can decompose $\fg'=\fg'_1\dots\fg'_r$,
 with all $\fg'_i\in\fQ(P)$; we can then conjugate the previous decomposition in $\fG(P)$ to obtain a decomposition
 $\fg=\fg_1\dots\fg_r$, with all $\fg_i$ still lying in $\fQ(P)$.

\subsection{Proof of Proposition \ref{prop:fQextwelldefined}}
\label{subsec:fQextwelldefined}
 Let $\fg\in\fQext(P)$ and assume first that $\fg=[\gamma]$ is
 represented by a \emph{simple} loop $\gamma$ in $\CmP$. Let $\fg=\fg_1\dots\fg_{\rho}$ be a decomposition
 of $\fg$ in elements $\fg_i\in\fQext(P)$. 
 Each $\fg_i$ can be further decomposed, by Proposition \ref{prop:fQgeneratesfQext},
 as $\fg_{i,1}\dots\fg_{i,r_i}$, with $\fg_{i,j}\in\fQ(P)$; therefore we obtain a 
 decomposition
 \[
 \fg=\fg_{1,1}\dots\fg_{1,r_1}\fg_{2,1}\dots\fg_{2,r_2}\dots\dots\fg_{\rho,1}\dots\fg_{\rho,r_{\rho}}
 \]
 of $\fg$ with respect to $\fQ(P)$.
 Our aim to show that, for all $1\leq i<j\leq\rho$,
 the product $\fg_i\dots\fg_j$ belongs to $\fQext(P)$. It suffices to prove
 the same statement for the second, finer decomposition involving the elements $\fg_{i,j}$.
 Hence, from now on, we assume that the elements $\fg_1,\dots,\fg_{\rho}$ already belong
 to $\fQ(P)$.
 According to \cite[Definition 3.5]{Bianchi:Hur1}, $(\fg_1,\dots,\fg_{\rho})$ is then a decomposition of $\fg$ with respect to $\fQ(P)$.
 
 By the same argument used in the proof of Proposition \ref{prop:fQgeneratesfQext}, we can find an admissible generating set
 $\gen_1,\dots,\gen_k$ of $\fG(P)$ such that, for some $1\leq r\leq l$, we have $\fg=\gen_1\dots\gen_r$,
 and such that $\gen_1,\dots,\gen_r$ are contained in the subgroup $\pi_1\pa{D\setminus P,*}\cong\bF^r$
 of $\fG(P)$, where $D$ is the disc bounded by $\gamma$.
 
 We note that $(\gen_1,\dots,\gen_r)$ is also a decomposition of $\fg$ with respect to $\fQ(P)$, and
 a simple argument involving the projection onto the abelianisation of $\fG(P)$ shows that $\rho=r$
 (see the remark after \cite[Definition 3.5]{Bianchi:Hur1}).
 The decompositions $(\gen_1,\dots,\gen_r)$ and $(\fg_1,\dots,\fg_r)$ are connected by a sequence of standard
 moves (see \cite[Definition 3.6, Proposition 3.7]{Bianchi:Hur1}).
 
 A consequence of the previous argument is that $\fg_1,\dots,\fg_r\in\fG(P)$ can be generated using the elements $\gen_1,\dots,\gen_r$, and therefore $\fg_1,\dots,\fg_r$ also lie in the subgroup $\pi_1\pa{D\setminus P,*}\subseteq\fG(P)$.
 
 In analogy with Definition \ref{defn:admgenset}, we say that $\gen_1,\dots,\gen_r$ is an admissible generating
 set of $\pi_1\pa{D\setminus P,*}$, meaning that each $\gen_i$ is represented by a simple loop that spins around one of the $r$ points of $D\cap P$, and these loops only intersect at $*$.
 
 It is now a classical fact that standard
 moves on admissible generating sets of $\pi_1\pa{D\setminus P,*}$ can be implemented by homeomorphisms
 of $D$. More precisely, if $\pa{\tilde{\gen}_1,\dots,\tilde{\gen}_r}$ is an admissible
 generating set of $\pi_1\pa{D\setminus P,*}$ and the sequence $\pa{\tilde{\fg}_1,\dots,\tilde{\fg}_r}$
 of elements of $\pi_1\pa{D\setminus P,*}$ is obtained from the sequence $\pa{\tilde{\gen}_1,\dots,\tilde{\gen}_r}$
 by a standard move, then there is a homeomorphism $\xi\colon D\to D$ such that
 \begin{itemize}
  \item $\xi$ fixes $\gamma=\partial D$ pointwise: in particular $\xi(*)=*$;
  \item $\xi$ fixes $D\cap P$ as a set: in particular, $\xi$ restricts to a homeomorphism of $D\setminus P$;
  \item the map $\xi_*\colon\pi_1\pa{D\setminus P,*}\to \pi_1\pa{D\setminus P,*}$ sends $\tilde{\gen}_i\mapsto \tilde{\fg}_i$
  for all $1\leq i\leq r$.
 \end{itemize}
 By applying this argument several times, we obtain that $\fg_1,\dots,\fg_r$ is also an admissible
 generating set of $\pi_1\pa{D\setminus P,*}$, and the fact that the product $\fg=\fg_1\dots\fg_r$
 is represented by a simple loop implies that the elements $\fg_1,\dots,\fg_r$ are ordered in a
 standard way, so that for all $1\leq i<j\leq r$ also the product $\fg_i\dots\fg_j$ is represented
 by a simple loop in $D\setminus P\subset\CmP$; thus $\fg_i\dots\fg_j\in\fQext(P)$.

 The case in which $\fg$ is not represented by a simple loop $\gamma$ is treated in the same way
 as in the proof of Proposition \ref{prop:fQgeneratesfQext}: we can find a conjugate $\fg'\in\fG(P)$ represented
 by a simple loop $\gamma'$, in particular $\fg'\in\fQext(P)$; we conjugate the factorisation $\fg=\fg_1\dots\fg_{\rho}$ 
 to obtain a factorisation $\fg'=\fg'_1\dots\fg'_{\rho}$; by the previous argument each product
 $\fg'_i\dots\fg'_j$ lies in $\fQext(P)$, and therefore also its conjugate $\fg_i\dots\fg_j$ lies in $\fQext(P)$.
 
\subsection{Proof of Proposition \ref{prop:fQextPpsi}}
\label{subsec:fQextPpsi}
Let $\fg=\fg_1\dots\fg_{\rho}$ be a decomposition of $\fg\in\fQext(P)_{\psi}$
with $\fg_i\in\fQext(P)_{\psi}$ for all $1\leq i\leq \rho$.
As in the proof of Proposition \ref{prop:fQextwelldefined},
we replace each $\fg_i$ by a decomposition
$\fg_{i,1}\dots\fg_{i,r_i}$, with $\fg_{i,j}\in\fQ(P)$; thus we obtain a decomposition of $\fg$
with respect to $\fQ(P)$
\[
 \fg=\fg_{1,1}\dots\fg_{1,r_1}\fg_{2,1}\dots\fg_{2,r_2}\dots\dots\fg_{\rho,1}\dots\fg_{\rho,r_{\rho}}.
\]
Since $\fg\in\fQext(P)_{\psi}$, the following product is defined in $\Q$:
\[
\psiext(\fg)=\psi(\fg_{1,1})\dots\psi(\fg_{1,r_1})\psi(\fg_{2,1})\dots\psi(\fg_{2,r_2})\dots\dots\psi(\fg_{\rho,1})\dots\psi(\fg_{\rho,r_{\rho}});
\]
In particular for all $1\leq i<j\leq\rho$, the sub-product
$\psi_(\fg_{i,1})\dots\psi(\fg_{j,r_j})$ is defined in $\Q$. Together
with Proposition \ref{prop:fQextwelldefined}, this shows that $\fg_i\dots\fg_j$ lies in $\fQext(P)_{\psi}$,
hence $\fQext(P)_{\psi}\subseteq\fG(P)$ satisfies the hypotheses of \cite[Definition 2.8]{Bianchi:Hur1}.

The same argument shows also that $\psiext(\fg)=\psiext(\fg_1)\dots\psiext(\fg_r)$ in $\Q$, hence $\psiext$
is a map of partial monoids. It is also evident that $\psiext$ restricts to $\psi$ on $\fQ(P)$.
To see that $\psiext$ also preserves conjugation, let
$\fg,\fg'\in\fQext(P)_{\psi}$ and choose decompositions $(\fg_1,\dots,\fg_r)$ and
$(\fg'_1,\dots,\fg'_{r'})$ of $\fg$ and $\fg'$ respectively with respect to $\fQ(P)$.
We have a chain of equalities
{\small
\[
\begin{split}
 \psiext\pa{\fg^{\fg'}}\!& =\psiext\pa{\fg_1^{\fg'}\dots\fg_r^{\fg'}}
		     \ =\ \psiext\pa{\fg_1^{\fg'}}\dots\psiext\pa{\fg_r^{\fg'}}
		     \ =\ \psi\pa{\fg_1^{\fg'}}\dots\psi\pa{\fg_r^{\fg'}}\\
		     & =\psi\pa{\pa{\dots\pa{\fg_1^{\fg'_1}}^{\fg'_2}\dots}^{\fg'_{r'}}}\dots\psi\pa{\pa{\dots\pa{\fg_r^{\fg'_1}}^{\fg'_2}\dots}^{\fg'_{r'}}}\\
		     & =\!\pa{\!\dots\!\pa{\psi\pa{\fg_1}^{\psi\pa{\fg'_1}}}^{\psi\pa{\fg'_2}}\!\!\!\!\dots}^{\psi\pa{\fg'_{r'}}}
		     \!\!\!\!\!\!\dots\pa{\!\dots\pa{\psi\pa{\fg_r}^{\psi\pa{\fg'_1}}}^{\psi\pa{\fg'_2}}\!\!\!\!\dots}^{\psi\pa{\fg'_{r'}}}\\
		     &=\!\!\pa{\!\!\dots\!\pa{\!\psiext\!\pa{\fg_1}^{\psiext\!\pa{\fg'_1}}\!}^{\psiext\!\pa{\fg'_2}}\!\!\!\!\!\dots\!}^{\!\!\psiext\!\pa{\fg'_{r'}}}
		     \!\!\!\!\!\!\!\dots\!\pa{\!\!\dots\!\pa{\psiext\!\pa{\fg_r}^{\psiext\!\pa{\fg'_1}}\!}^{\psiext\!\pa{\fg'_2}}\!\!\!\!\!\dots\!}^{\!\!\psiext\!\pa{\fg'_{r'}}}\\
		     & =\psiext\pa{\fg_1}^{\psiext\pa{\fg'}}\dots\psiext\pa{\fg_r}^{\psiext\pa{\fg'}}
		     \ =\ \psiext\pa{\fg}^{\psiext\pa{\fg'}}. 
\end{split}
\]
}
Thus $\psiext\colon\fQext(P)_{\psi}\to\Q$ is a map of PMQs, restricting to the map $\psi\colon\fQ(P)\to\Q$.
The fact that $\psiext$ is the unique map of PMQs with these properties is a direct consequence of Proposition \ref{prop:fQgeneratesfQext}.

\subsection{Proof of Lemma \ref{lem:xirestrictsfQ}}
\label{subsec:xirestrictsfQ}
 Let $z'\in P'\setminus\cY'$ and let $\gamma'$ be a based loop in $\CmP'$
 which is freely homotopic to a simple closed curve $\beta'\subset\CmP'$ spinning clockwise
 around $z'$: in particular $\beta'$ bounds a closed disc $D'\subset\CmP'$, with $D'\cap P'=\set{z'}$.
 
 We have that $D=\xi^{-1}(D')$ is also a disc in $\CmP$, and by
 property (5) in Definition \ref{defn:mapnicecouples} and by definition of $P':=\xi(P)$,
 there is a unique $z\in P$ with $\xi(z)=z'$. We consider $\beta=\partial D$
 as a simple closed curve in $\CmP$ spinning clockwise around $z$: then $\xi$
 restricts to a homotopy equivalence $\beta\to\beta'$, since:
 \begin{itemize}
  \item both spaces are homotopy equivalent to $S^1$, hence it suffices to prove that
  $\xi$ induces a cohomology equivalence;
  \item the inclusions $\beta\subset D\setminus z$ and $\beta'\subset D'\setminus z'$
  are homotopy equivalences, in particular cohomology equivalences;
  \item the map $\xi\colon D\setminus z\to D'\setminus z' $ is a cohomology equivalence:  
  this can be seen by comparing the cohomology long exact sequences of the couples $(D,D\setminus z)$
  and $(D',D'\setminus z')$, using in particular that the map $\xi^*\colon H^2(D',D'\setminus z')\to H^2(D,D\setminus z)$ can be rewritten as $\xi^*\colon H^2_c(\C)\to H^2_c(\C)$, and is thus an isomorphism.
 \end{itemize}
 Moreover property (2) in Definition \ref{defn:mapnicecouples} implies that
 $\xi \colon\beta\to\beta'$ is orientation-preserving, if both curves are oriented clockwise.
 
 This implies that the conjugacy class represented by $\beta'$ is mapped along $\xi^*$ inside
 the conjugacy class represented by $\beta$, which is contained in $\fQ_{\fT}(P)$.

\subsection{Proof of Lemma \ref{lem:xirestrictsfQLNC}}
\label{subsec:xirestrictsfQLNC}
Let $\gamma'\subset\CmP'$ be a based loop homotopic to a simple closed curve $\beta'$, with $\beta'$ contained
in $\C\setminus\cY'$ and $\beta'$ oriented clockwise, such that $\beta'$ bounds a disc $D'\subset\C\setminus\cY'$.

Let $D=\xi^{-1}(D')$, which is a topological disc contained in $\C\setminus\cY$, and let
$\beta=\partial D$. Let $K'\subset\mathring{D}'$ be a smaller, closed disc containing
$P'\cap D'$, and denote $K=\xi^{-1}(K')$.

Then $\xi\colon \beta\to\beta'$ is a homotopy equivalence, since:
 \begin{itemize}
  \item both spaces are homotopy equivalent to $S^1$, hence it suffices to prove that
  $\xi$ induces a cohomology equivalence;
  \item the inclusions $\beta\subset D\setminus K$ and 
  $\beta'\subset D'\setminus K'$ are homotopy equivalences, in particular cohomology
  equivalences;
  \item the map $\xi\colon D\setminus K\to D'\setminus K'$ is a cohomology equivalence:
  this can be seen by comparing the cohomology long exact sequences of the couples $(D,D\setminus K)$
  and $(D',D'\setminus K')$, using in particular that the map $\xi^*\colon H^2(D',D'\setminus K')\to H^2(D,D\setminus K)$ can be rewritten as $\xi^*\colon H^2_c(\C)\to H^2_c(\C)$, and is thus
  an isomorphism.
 \end{itemize}
 Moreover property (2) in Definition \ref{defn:mapnicecouples} implies that
 $\xi \colon\beta\to\beta'$ is orientation-preserving, if both curves are oriented clockwise.
 
 It follows that $\xi^*$ maps the conjugacy class of $\beta'$ inside the conjugacy class of $\beta$,
 which is contained in $\fQext_{\fT}(P)$.
 
\subsection{Proof of Lemma \ref{lem:usefulhomeocYempty}}
\label{subsec:usefulhomeocYempty}

 Let $\fc=(P,\psi,\phi)\in\Hur(\fT;\Q,\cG(\Q))$, use Notation \ref{nota:P} and let $\gen_1,\dots,\gen_k$ be an admissible
 generating set for $\fG(P)$. Since we are dealing with the nice couple $(\cX,\emptyset)$, whose second space is empty,
 we have that $\gen_1,\dots,\gen_k\in\fQ(P)$. By Definition \ref{defn:Hurset} we have $\phi(\gen_i)=\eta_{\Q}(\psi(\gen_i))\in\cG(\Q)$;
 since $\gen_1,\dots,\gen_k$ exhibit $\fG(P)$ as a free group, we have that $\phi\colon\fG(P)\to \cG(\Q)$ is
 uniquely determined by $\psi$.
 On the other
 hand, by \cite[Theorem 3.3]{Bianchi:Hur1}, given any finite subset $P\subset\cX$ and a map of PMQs $\psi\colon\fQ(P)\to\Q$, one can
 use the assignment $\gen_i\mapsto \eta_{\Q}(\psi(\gen_i))$ to define a group homomorphism 
 $\phi\colon\fG(P)\to\cG(\Q)$ making $(\psi,\phi)\colon (\fQ(P),\fG(P))\to(\Q,\cG(\Q))$ into a map of PMQ-group pairs.
 
 Let $\fc'=(P',\psi',\phi')$ be the image of $\fc$ along $(\Id_{\Q},\cG(\fe))_*$; then
 we have $P'=P$ and $\psi'=\psi$; from the previous discussion it follows that
 $\fc$ can be reconstructed from $\fc'$, and this proves injectivity of $(\Id_{\Q},\cG(\fe))_*$.
 
 Viceversa, let $\fc'=(P,\psi,\phi')$ be any configuration in $\Hur(\fT;\Q,G)$; then the previous discussion shows that one can construct
 a configuration $\fc\in\Hur(\fT;\Q,\cG(\Q))$ which is sent to $\fc'$ along $(\Id_{\Q},\cG(\fe))_*$: it suffices to take $\fc=(P,\psi,\phi)$,
 with $\phi$ defined as above by setting $\gen_i\mapsto \eta_{\Q}(\psi(\gen_i))$; this proves surjectivity
 of $(\Id_{\Q},\cG(\fe))_*$.
 
 To conclude, note that for all adapted coverings $\uU$ of $P$, the map $(\Id_{\Q},\cG(\fe))_*$ restricts
 to a bijection from $\fU(\fc;\uU)\subset\Hur(\fT;\Q,\cG(\Q))$ to $\fU(\fc';\uU)\subset\Hur(\fT;\Q,G)$,
 where again we let $\fc'$ be the image of $\fc$ along $(\Id_{\Q},\cG(\fe))_*$. This shows that
 $(\Id_{\Q},\cG(\fe))_*$ is a homeomorphism.

\subsection{Proof of Lemma \ref{lem:usefulhomeoGG}}
\label{subsec:usefulhomeoGG}
The proof is analogous to the one of Lemma \ref{lem:usefulhomeocYempty}.
 Let $\fc=(P,\psi,\phi)\in\Hur(\fT;G,G)$, use Notation \ref{nota:P} and let $\gen_1,\dots,\gen_k$ be an admissible
 generating set for $\fG(P)$.
 Since we are dealing with the PMQ-group pair $(G,G)$, the composition $\fQ_{\fT}(P)\subset\fG(P)\overset{\phi}{\to} G$ equals
 $\psi\colon\fQ_{\fT}(P)\to G$. In particular $\psi$ can be recovered from $\phi$.
 
 On the other hand, by \cite[Theorem 3.3]{Bianchi:Hur1}, given any finite subset $P\subset\cX$ and a map of groups $\phi\colon\fG(P)\to G$, one can
 use the assignment $\psi\colon\gen_i\mapsto \phi(\gen_i)$ for $1\leq i\leq l$ (using Notation \ref{nota:P})
 to define a map of PMQs $\psi\colon\fQ_{\fT}(P)\to G$ making $(\psi,\phi)\colon (\fQ_{\fT}(P),\fG(P))\to(G,G)$
 into a map of PMQ-group pairs.

 Let $\fc'=(P',\psi',\phi')$ be the image of $\fc$ along $(\Id_{\C})_*$; then
 we have $P'=P$ and $\phi'=\phi$; from the previous discussion it follows that
 $\fc$ can be reconstructed from $\fc'$, and this proves injectivity of $(\Id_{\C})_*$.

 Viceversa, let $\fc'=(P,\psi,\phi')$ be any configuration in $\Hur(\cX,\cX;G,G)$; then the previous discussion shows that one can construct
 a configuration $\fc\in\Hur(\fT;G,G)$ mapping to $\fc'$ along $(\Id_{\C})_*$: it suffices to take $\fc=(P,\psi,\phi)$,
 with $\psi$ defined as above by setting $\psi\colon\gen_i\mapsto \phi(\gen_i)$; this proves surjectivity
 of $(\Id_{\C})_*$.
 
 To conclude, note that if $\uU$ is an adapted covering of $P$ with respect to $\fT$, then $\uU$ is also adapted with respect to $(\cX,\cX)$,
 and the map $(\Id_{\C})_*$ restricts
 to a bijection from $\fU(\fc;\uU)\subset\Hur(\fT;G,G)$ to $\fU(\fc';\uU)\subset\Hur(\cX,\cX;G,G)$,
 where again we let $\fc'$ be the image of $\fc$ along $(\Id_{\C})_*$. This proves that
$(\Id_{\C})_*$ is a homeomorphism.

\subsection{Proof of Proposition \ref{prop:actiongfc}}
\label{subsec:actiongfc}
 We focus on the left-based case.
 Let $\fc\in\Hur(\fT;\Q,G)_{\zleft}$, use Notations \ref{nota:fc} and \ref{nota:uUcovering} and let $\uU$ be an adapted covering of $P$.
 Denote by $\Uleft$ the component of $\uU$ containing $\zleft$; possibly
 up to shrinking $\Uleft$, we can assume that the simple closed curve
 $\del\Uleft$ is cut by $\bS_{\Re(\zleft),\Re(\zleft)}$ in two arcs. We decompose
 $\C$ as the union of two subspaces: the first subspace is $\bT_1$, which is defined as the closure in $\C$ of 
 $\bS_{-\infty,\Re(\zleft)}\cup \Uleft$; the other subspace is
 $\bT_2=\bS_{\Re(\zleft),\infty}\setminus\Uleft$. The first subspace contains $P_1\colon=\set{\zleft}$
 in its interior, the second subspace contains $P_2\colon= P\setminus\set{\zleft}$ in its interior.
 The two subspaces $\bT_1$ and $\bT_2$ intersect in a contractible space containing $*$. 
 
 Using the theorem of Seifert and van Kampen we can write
 $\fG(P)$ as the free product $\pi_1(\bT_1\setminus P_1,*)\star\pi_1(\bT_2\setminus P_2,*)$: the first factor is
 freely generated by $\genleft$, the second factor is freely generated by the other generators $\gen_i$
 in a left-based admissible generating set.
 The map $\phi'$ in Definition \ref{defn:actiongfc} can then be equivalently defined by
 setting $\phi'(\genleft)=g\cdot \phi(\genleft)$, and by imposing that $\phi'$ and $\phi$ agree on the second factor.
 
 Moreover, consider the nice couple $\fT_2:=(\cX\cap\bT_2,\cY\cap\bT_2)$: then
 $P_2$ is contained in $\cX\cap\bT_2$. The composition
 \[
 \begin{tikzcd}
  \fQ_{\fT_2}(P_2)\ar[r,"\subseteq"] & \fG(P_2) & \pi_1\pa{\bT_2\setminus P_2,*}\ar[l,"\cong"]\ar[r,"\subseteq"] &\fG(P)
 \end{tikzcd}
 \]
 has image in $\fQ_{\fT}(P)$ and identifies $\fQ_{\fT_2}(P)$ with the sub-PMQ of $\fQ_{\fT}(P)$
 containing homotopy classes which can be represented by a simple
 loop in $\bT_2\setminus P_2$ spinning clockwise around one of the points $z_1,\dots,z_l$.
 The map $\psi'\colon\fQ_{\fT}(P')\to\Q$
 from Definition \ref{defn:actiongfc} can be characterised by the following two properties:
 \begin{itemize}
  \item $\psi$ and $\psi'$ have the same restriction on $\fQ_{\fT_2}(P)$, regarded as a subset of $\fQ_{\fT}(P)$ as explained above.
  \item $(\psi',\phi')$ is a map of PMQ-group pairs $(\fQ_{\fT}(P),\fG(P))\to (\Q,G)$.
 \end{itemize}
 The fact that this is a characterisisation (i.e. existence and uniqueness of
 $\psi'$ with these properties) is shown using a choice
 of a left-based admissible generating set for $\fG(P)$ and using \cite[Theorem 3.3]{Bianchi:Hur1}; but the characterising
 properties of $\phi'$ and $\psi'$ are now stated without reference to a left-based admissible generating set.
 
 The fact that the collection of all maps $g\cdot-$ gives an action of $G$ on the set $\HurTQG_{\zleft}$ follows directly
 from the formulas in Definition \ref{defn:actiongfc}.
 To prove continuity of $g\cdot -$, note that for all adapted coverings $\uU$ of $P$
 the map $g\cdot -$ establishes a bijection between the open subspaces
 $\fU(\fc,\uU)_{\zleft}$ and $\fU(g\cdot \fc,\uU)_{\zleft}$ of $\HurTQG_{\zleft}$.
 In particular $g\cdot-$ is a homeomorphism of $\HurTQG$ with inverse $g^{-1}\cdot-$.
 
The right-based case is analogous; the main difference is that, in the first part, 
one considers the component $\Uright$ of $\uU$
covering $\zright$, and decomposes $\C$ as the union of $\bT_1=\bS_{-\infty,\Re(\zright)}\setminus\Uright$
and $\bT_2$ being the closure in $\C$ of $\bS_{\Re(\zright),\infty}\cup\Uright$.

 \subsection{Proof of Proposition \ref{prop:explosion}}
  \label{subsec:explosion}
Fix $(\fc,t)\in\Hur(\fT;\Q,G)\times[0,1]$, denote $\fc'=\expl_*(\fc,t)$ and $\fc''=\rho(\fc)$, and use
Notation \ref{nota:fc}. Without loss of generality
assume that $z_1,\dots,z_r\in P\setminus\cY$ are precisely the inert points of $\fc$,
for some $0\leq r\leq l$. Then $P''=P\setminus\set{z_1,\dots,z_r}$ and $P'=P''\cup\expl(P,t)$.

Let $\uU'$ be an adapted covering of $P'$. Our aim is to find a neighbourhood of
$(\fc,t)\in\Hur(\fT;\Q,G)\times[0,1]$ which is mapped by $\expl_*$ inside $\fU(\fc',\uU')$.
Let $\uU'_{\expl(P,t)}\subset\uU'$ denote the restriction of $\uU'$ to $\expl(P,t)\subset P'$,
i.e. the sequence of components of $\uU'$ containing a point in $\expl(P,t)$. Then by continuity
of $\expl$ we can find an adapted covering $\uU$ of $P$ and a neighbourhood $V$ of $t\in [0,1]$
such that $\expl$ maps the entire product neighbourhood
$\fU(P,\uU)\times V$ inside $\fU(\expl(P,t),\uU'_{\expl(P,t)})\subset\Ran(\fT)$.

Use Notation \ref{nota:uUcovering}: up to shrinking the components
of $\uU$, we may assume that $U_i\subset U'_i$ whenever $z_i$ belongs to $P''\subset P\cap P'$, that is
$U_{r+1}\cup\dots\cup U_k\subseteq \uU'$.
We claim that $\fU(\fc,\uU)\times V$ is mapped by $\expl_*$ inside $\fU(\fc',\uU')$; the rest of the proof
is devoted to this claim.

We fix $\check\fc=(\check P,\check\psi,\check\phi)\in\fU(\fc,\uU)$ and $\check t\in V$, and let $\check\fc'=(\check P',\check\psi',\check\phi')=\expl_*(\check\fc,\check t)$.
First, we prove that $\check P'\subset\uU'$. We can partition
$\check P$ into subsets $\check P_1,\dots,\check P_k$, with $\check P_i\subset U_i$.
Note that for all $1\leq i\leq r$ and for all $\check z\in \check P_i$, the point
$\check z$ is inert for $\check \fc$: indeed $\check z\in U_i\subset\C\setminus\cY$
because $\uU$ is an adapted covering of $P$, hence $\check z\in\cX\setminus\cY$;
moreover $\psi$ sends each element of $\fQ(\check P,\check z)$ to
a factor of $\one$ in the augmented PMQ $\Q$, i.e. to $\one$.
It follows that $\check P'$ is a subset of $\check P_{r+1}\cup\dots\cup\check P_k\cup\expl(\check P,\check t)$,
and the latter set is contained in $\uU'$ by our choice of $\uU$.

Second, we prove that every component of $\uU'$ intersects $\check P'$
in at least one point. Let $U'_i$ be the component of $\uU'$ containing the point $z'_i\in P'$,
for some $1\leq i\leq k'$. There are several cases to consider.
\begin{itemize}
 \item If $z'_i\in \expl(P,t)\subset P'$, then there is $z_j\in P$ with $z'_i\in\expl(z_j,t)$,
 and there is $\check z\in \check P\cap U_j$.
 We can restrict $\uU'$ to an adapted covering $\uU'_{\expl(z_j,t)}$ of $\expl(z_j,t)\subset P'$,
 by selecting the relevant connected components;
 then our hypothesis on $\uU$ implies that $\expl$ sends $(U_j\cap \cX)\times V$ inside
 $\fU(\expl(z_j,t),\uU'_{\expl(z_j,t)})$, where we use Notation \ref{nota:uUcovering}.
 In particular $\expl(\check z,\check t)\in\fU(\expl(z_j,t),\uU'_{\expl(z_j,t)})$,
 implying that $\expl(\check z,\check t)\subset\check P'$ contains a point lying in $U'_i$.
 \item If $z'_i\in\cY$, then $z'_i\in P$ and $z'_i\in\expl(z'_i,t)$, so we fall in the previous case.
 \item If $z'_i\in\cX\setminus\cY$ and $z'_i\in P'\setminus \expl(P,t)$, then $z'_i$ must be a non-inert
 point of $P$ for $\fc$. Since $\Q$ is augmented, there is a point $\check z\in \check P\cap U'_i$
 which is non-inert for $\check\fc$. This point $\check z$ also belongs to $\check P'\cap U'_i$.
\end{itemize}
The previous discussion shows that
$\fU(P,\uU)\times V$ is mapped by $\expl$ inside $\fU(P',\uU')$.

Let now $\gamma\subset\C\setminus \uU'$ be a simple loop spinning clockwise around a component $U'_i$;
up to slightly perturbing $\gamma$ we may assume that it is also disjoint from the finite set $P$,
i.e. $\gamma$ avoids the points $z_1,\dots,z_r$. 
Then we have the following chain of equalities
\[
 \phi'([\gamma])=\phi([\gamma])=\check\phi([\gamma])=\check\phi'([\gamma]).
\]
If $\gamma$ represents a class in $\fQ(P')$, we also have the following chain of equalities
\[
 \psi'([\gamma])=\psiext([\gamma])=\check\psi^{\mathrm{ext}}([\gamma])=(\check\psi')^{\mathrm{ext}}([\gamma]),
\]
where we use that $[\gamma]$ represents elements of
$\fQext(P)_{\psi}$, $\fQext(\check P)_{\check\psi}$ and $\fQext(\check P')_{\check\psi'}$,
and refer to Definition \ref{defn:fQextPpsi} and Proposition \ref{prop:fQextPpsi}.
This shows that $\check\fc'\in\fU(\fc',\uU')$.

Suppose now that $\expl$ is standard and let $\fc\in\HurTQG$;
then $\expl_*(\fc,0)$ is computed by first
deleting all inert points of $\fc$ via $\rho$, and then by readding these inert points
through the external product of $\rho(\fc)$ with $P:=\epsilon(\fc)=\expl(P,0)\in\Ran(\fT)$.

\subsection{Proof of Proposition \ref{prop:facerestrictionh}}
\label{subsec:facerestrictionh}
We introduce some notation for barycentres of faces of simplices.
Recall Notation \ref{nota:barycentresimplex}:
for $p\geq 1$ and $0\leq i\leq p$ we denote by
$\ubar^{p-1,i}=\pa{\frac 1p,\frac 2p,\dots,\frac{i-1}p,\frac ip,\frac ip,\frac{i+1}p,\dots,\frac{p-1}p}\in\Delta^p$ the barycentre of the face $d_i\Delta^p$.
\begin{lem}
 \label{lem:facerestrictionh}
Recall Definition \ref{defn:fcua}. The map $e^{\ua}$ sends
$(\ubar^{p-1,i},\ubar^q)\in\Delta^p\times\Delta^q$ to $\fc_{\ua'}$,
where $\ua'=d^{\hor}_i(\ua)$.
\end{lem}
Before proving Lemma \ref{lem:facerestrictionh}, we will argue how Proposition \ref{prop:facerestrictionh}
follows from it. Let $(\us,\ut)\in d^{\hor}_i(\Delta^p\times\Delta^q)\subset\Delta^p\times\Delta^q$ (see Notation
\ref{nota:facebisimplex}). Then the pair $(\ubar^{p-1,i},\ubar^q;\us,\ut)$ belongs to $\bDelta^{p,p}\times\bDelta^{q,q}$,
and we can factor $\cH^{p,q}_*(-;\ubar^p,\ubar^q;\us,\ut)$ as a composition
$\cH^{p,q}_*(-;\ubar^{p-1,i},\ubar^q;\us,\ut)\circ\cH^{p,q}_*(-;\ubar^p,\ubar^q;\ubar^{p-1,i},\ubar^q)$ by Lemma \ref{lem:bDeltacompositionC}.
Assuming Lemma \ref{lem:facerestrictionh}, we have that $\cH^{p,q}_*(-;\ubar^p,\ubar^q;\ubar^{p-1,i},\ubar^q)$
sends $\fc_{\ua}\mapsto\fc_{\ua'}$; then by definition the second map $\cH^{p,q}_*(-;\ubar^{p-1,i},\ubar^q;\us,\ut)$
sends $\fc_{\ua'}$ to $e^{\ua'}(\us,\ut)$ (regarding $(\us,\ut)$ as a point in $\Delta^{p-1}\times\Delta^q$),
and the composition $\cH^{p,q}_*(-;\ubar^p,\ubar^q;\us,\ut)$ sends $\fc_{\ua}\mapsto e^{\ua}(\us,\ut)$.

The rest of the subsection is devoted to the proof of Lemma \ref{lem:facerestrictionh}.
By definition,
$e^{\ua}(\ubar^{p-1,i},\ubar^q)$ is the image of $\fc_{\ua}$ under
$\cH^{p,q}_*(-;\ubar^p,\ubar^{p-1,i};\ubar^q,\ubar^q)$; in the following
we abbreviate by $\xi\colon\C\to \C$ the map $\cH^{p,q}(-;\ubar^p,\ubar^{p-1,i};\ubar^q,\ubar^q)$.

Recall Notation \ref{nota:Ppq}: the map $\xi$ sends $z^{p,q}_{i',j}\mapsto z^{p-1,q}_{i',j}$ for $0\leq i'\leq i$,
and $z^{p,q}_{i',j}\mapsto z^{p-1,q}_{i'-1,j}$ for $i+1\leq i'\leq p+1$;
it follows that the image of $P_{\ua}$ along $\xi$
is the set $P_{\ua'}$. This shows that $e^{\ua}(\ubar^{p-1,i},\ubar^q)$ is a configuration supported
on $P_{\ua'}$.

Recall Notation \ref{nota:genstdua}, and consider the standard generating sets $(\gen^{\ua'}_{i',j})_{(i',j)\in I(\ua')}$
of $\fG(P_{\ua'})$ and $(\gen^{\ua}_{(i',j)})_{(i',j)\in I(\ua)}$ of $\fG(P_{\ua})$,
and the homomorphism $\xi^*\colon \fG(P_{\ua'})\to\fG(P_{\ua})$ from Subsection
\ref{subsec:functorialityHurNC}:
for all $0\leq i'\leq p$ and $0\leq j\leq q+2$,
the homomorphism $\xi^*$ maps the product of standard generators $c\gen^{\ua'}_{i',j}$ to
\[
 \xi^*(c\gen^{\ua'}_{i',j})\mapsto \left\{
 \begin{array}{ll}
  c\gen^{\ua}_{i',j} &\mbox{if }i'\leq i-1;\\[6pt]
  c\gen^{\ua}_{i,j} \cdot c\gen^{\ua}_{i+1,j} & \mbox{if }i'=i;\\[6pt]
  c\gen^{\ua}_{i'+1,j} & \mbox{if }i'\geq i+1.
 \end{array}
 \right.
\]
This follows from the description of $c\gen^{\ua'}_{i',j}$ as the class of a simple loop supported on $\bS_{x_{i'-1},x_{i'+1}}\cap\set{\Im\leq y_j}$ and spinning clockwise around the points $z^{\ua'}_{i',0},\dots,z^{\ua'}_{i',j-1}$. For $i'=i$
we have in particular that $\xi^*(c\gen^{\ua'}_{i',j})$ is represented by a loop spinning around the 
horizontal segments joining $z^{\ua}_{i,j'}$ with $z^{\ua}_{i+1,j'}$, for $0\leq j'\leq j$: these horizontal
segments are the preimages along $\xi$ of the points $z^{\ua'}_{i,j'}$ for $0\leq j'\leq j$.

We can now use that $\xi^*$ is a group homomorphism and compute $\xi^*(\gen^{\ua}_{i',j})$
for all $(i',j)\in I(\ua')$. In particular, for $i'=i$ we obtain
\[
\begin{split}
 \xi^*(\gen^{\ua'}_{i,j}) & = \xi^*\pa{ \pa{c\gen^{\ua'}_{i,j}}^{-1}\cdot c\gen^{\ua'}_{i,j+1}}
 \ =\  \pa{c\gen^{\ua}_{i,j}\cdot c\gen^{\ua}_{i+1,j} }^{-1}\cdot c\gen^{\ua}_{i,j+1}\cdot c\gen^{\ua}_{i+1,j+1}\\
 & = \pa{c\gen^{\ua}_{i+1,j}}^{-1}\cdot \pa{c\gen^{\ua}_{i,j}}^{-1}\cdot c\gen^{\ua}_{i,j+1}\cdot c\gen^{\ua}_{i+1,j+1}
 \ =\  \pa{c\gen^{\ua}_{i+1,j}}^{-1}\cdot \gen^{\ua}_{i,j}\cdot c\gen^{\ua}_{i+1,j+1}\\
 & = \pa{c\gen^{\ua}_{i+1,j}}^{-1}\cdot \gen^{\ua}_{i,j}\cdot c\gen^{\ua}_{i+1,j}\cdot \gen^{\ua}_{i+1,j}
  \ =\  \pa{\gen^{\ua}_{i,j}}^{c\gen^{\ua}_{i+1,j}}\cdot \gen^{\ua}_{i+1,j}.
\end{split}
\]
Similarly, for $i'<i$ we have $\xi^*(\gen^{\ua'}_{i',j})=\gen^{\ua}_{i',j}$, and for $i'>i$
we have $\xi^*(\gen^{\ua'}_{i',j})=\gen^{\ua}_{i'+1,j}$.

Applying $\psi_{\ua}$ to $\xi^*(\gen^{\ua'}_{i',j})$,
we obtain the equality $a'_{i',j}=\psi_{\ua'}(\gen^{\ua'}_{i',j})$,
which together with \cite[Lemma 6.8]{Bianchi:Hur1} yields $\ua'=d^{\hor}_i\ua$.

\subsection{Proof of Proposition \ref{prop:facerestrictionv}}
\label{subsec:facerestrictionv}
Similarlz as for Proposition \ref{prop:facerestrictionh},
the statement of Proposition \ref{prop:facerestrictionv} follows directly from the following lemma.
\begin{lem}
 \label{lem:facerestrictionv}
The map $e^{\ua}$ sends
$(\ubar^p,\ubar^{q-1,j})\mapsto\fc_{\ua'}$,
where $\ua'=d^{\ver}_j(\ua)$.
\end{lem}
The rest of the subsection is devoted to the proof of Lemma \ref{lem:facerestrictionv}.
By definition,
$e^{\ua}(\ubar^p,\ubar^{q-1,j})$ is the image of $\fc_{\ua}$ under the map
$\xi_*$, where for the rest of the proof we abbreviate by $\xi\colon\C\to\C$ the map
$\cH^{p,q}(-;\ubar^p,\ubar^p;\ubar^q,\ubar^{q-1,j})$.

The map $\xi$ sends $z^{p,q}_{i,j'}\mapsto z^{p,q-1}_{i,j'}$ for $0\leq j'\leq j$,
and $z^{p,q}_{i,j'}\mapsto z^{p,q-1}_{i,j'-1}$ for $j+1\leq j'\leq q+1$;
it follows that the image of $P_{\ua}$ along $\xi$
is the set $P_{\ua'}$, and as in the horizontal case we obtain
that $e^{\ua}(\ubar^p,\ubar^{q-1,j})$ is a configuration supported
on $P_{\ua'}$.

Consider now the two standard generating sets $(\gen^{\ua'}_{i,j'})_{(i,j')\in I(\ua')}$
of $\fG(P_{\ua'})$, and $(\gen^{\ua}_{(i,j')})_{(i,j')\in I(\ua)}$ of $\fG(P_{\ua})$,
and consider the homomorphism $\xi^*\colon \fG(P_{\ua'})\to\fG(P_{\ua})$.

The key observation is that, for all $0\leq i\leq p+1$ and $0\leq j'\leq q+1$,
the homomorphism $\xi^*$ maps the product of standard generators $c\gen^{\ua'}_{i,j'}$ to
\[
 \xi^*(c\gen^{\ua'}_{i,j'})\mapsto \left\{
 \begin{array}{ll}
  c\gen^{\ua}_{i,j'} &\mbox{if }j'\leq j;\\[6pt]
  c\gen^{\ua}_{i,j'+1} & \mbox{if }j'\geq j+1.
 \end{array}
 \right.
\]
This follows from the description of $c\gen^{\ua'}_{i,j'}$  as the class of a simple loop
supported on $\bS_{x_{i-1},x_{i+1}}\cap\set{\Im\leq y_{j'}}$
and spinning clockwise around the points $z^{\ua'}_{i,0},\dots,z^{\ua'}_{i,j'-1}$.
For $j'\geq j+1$, in particular, $\xi^*(c\gen^{\ua'}_{i,j'})$ is represented by a loop spinning around the 
points $z^{\ua}_{i,0},\dots,z^{\ua}_{i,j-1},z^{\ua}_{i,j+2},\dots,z^{\ua}_{i,j'-1}$
and around the vertical segment joining $z^{\ua}_{i,j}$ with $z^{\ua}_{i,j+1}$:
note that this segment is in the preimage along $\xi$ of the point $z^{\ua'}_{i,j}$.

We can now use that $\xi^*$ is a group homomorphism and compute $\xi^*(\gen^{\ua}_{i,j'})$
for all $(i,j')\in I(\ua')$. In particular, for $j'=j$ we obtain
\[
 \xi^*(\gen^{\ua'}_{i,j}) = \xi^*\pa{ \pa{c\gen^{\ua'}_{i,j}}^{-1}\cdot c\gen^{\ua'}_{i,j+1}} = \pa{c\gen^{\ua}_{i,j}}^{-1}\cdot c\gen^{\ua}_{i,j+2}
 = \gen^{\ua}_{i,j}\cdot\gen^{\ua}_{i,j+1}.
\]
Similarly, for $j'<j$ we have $\xi^*(\gen^{\ua'}_{i,j'})=\gen^{\ua}_{i,j'}$, and for $j'>j$
we have $\xi^*(\gen^{\ua'}_{i,j'})=\gen^{\ua}_{i,j'+1}$.
Applying $\psi_{\ua}$ to $\xi^*(\gen^{\ua'}_{i,j'})$,
we obtain the equality $a'_{i,j'}=\psi_{\ua'}(\gen^{\ua'}_{i,j'})$,
which together with \cite[Lemma 6.8]{Bianchi:Hur1} yields $\ua'=d^{\hor}_i\ua$.

\bibliography{Bibliography2}{}
\bibliographystyle{alpha}

\end{document}